\documentclass[12pt]{article}

\usepackage[top=20truemm,bottom=20truemm,left=20truemm,right=20truemm]{geometry}
\usepackage{mathptmx}      
\usepackage{amsmath}
\numberwithin{equation}{section}
\usepackage{amssymb}
\usepackage{amsthm}
\usepackage{amscd}
\usepackage{enumerate}
\usepackage{tikz-cd}
\usepackage{bm}
\usepackage{hyperref}
\usepackage{algorithm}
\usepackage{algpseudocode}
\usepackage{authblk}

\usepackage{tabularx,booktabs}

\usepackage{xcolor}
\hypersetup{
setpagesize=false,
 bookmarksnumbered=true,%
 bookmarksopen=true,%
 colorlinks=true,%
 linkcolor=black,
 citecolor=black,
}
\usepackage{graphicx, caption}
\usepackage{mathrsfs}
\usepackage[all]{xy}

\theoremstyle{definition}
\newtheorem{theorem}{Theorem}[section]

\newtheorem{proposition}[theorem]{Proposition}
\newtheorem{corollary}[theorem]{Corollary}
\newtheorem{lemma}[theorem]{Lemma}
\newtheorem{assumption}[theorem]{Assumption}
\newtheorem{remark}[theorem]{Remark}
\newtheorem{definition}[theorem]{Definition}
\newtheorem{example}[theorem]{Example}

\newcommand{\revised}[1]{{\color{blue}#1}}

\providecommand{\keywords}[1]
{
  \small	
  \textbf{\textit{Keywords---}} #1
}
\title{Koopman Operators with Intrinsic Observables in \\
Rigged Reproducing Kernel Hilbert Spaces}
\date{}
\author[1,4]{Isao Ishikawa\thanks{ishikawa.isao.5s@kyoto-u.ac.jp}}
\author[2,4]{Yuka Hashimoto\thanks{yuka.hashimoto@ntt.com}}
\author[3,4]{Masahiro Ikeda\thanks{ikeda@ist.osaka-u.ac.jp}}
\author[3,4]{Yoshinobu Kawahara\thanks{kawahara@ist.osaka-u.ac.jp}}

\affil[1]{Kyoto University}
\affil[2]{NTT, Inc.}
\affil[3]{The University of Osaka}
\affil[4]{RIKEN}

\begin{document}
\maketitle

\begin{abstract}
This paper presents a novel approach for estimating the Koopman operator defined on a reproducing kernel Hilbert space (RKHS) and its spectra. 
We propose an estimation method, what we call {\em Jet Extended Dynamic Mode Decomposition (JetEDMD)}, leveraging the intrinsic structure of RKHS and the geometric notion known as jets to enhance the estimation of the Koopman operator. 
This method refines the traditional Extended Dynamic Mode Decomposition (EDMD) in accuracy, especially in the numerical estimation of eigenvalues.
This paper proves JetEDMD's superiority through explicit error bounds and convergence rate for special positive definite kernels, offering a solid theoretical foundation for its performance. 
We also investigate the spectral analysis of the Koopman operator, proposing the notion of an extended Koopman operator within a framework of a rigged Hilbert space. This notion leads to a deeper understanding of estimated Koopman eigenfunctions and capturing them outside the original function space. 
Through the theory of rigged Hilbert space, our study provides a principled methodology to analyze the estimated spectrum and eigenfunctions of Koopman operators, and enables eigendecomposition within a rigged RKHS. 
We also propose a new effective method for reconstructing the dynamical system from temporally-sampled trajectory data of the dynamical system with solid theoretical guarantee. 
We conduct several numerical simulations using the van der Pol oscillator, the Duffing oscillator, the H\'enon map, and the Lorenz attractor, and illustrate the performance of JetEDMD with clear numerical computations of eigenvalues and accurate predictions of the dynamical systems.
\end{abstract}


\section{Introduction}\label{sec: introduction}

The Koopman operator is ``linearization'' of a dynamical system introduced in \cite{Koo-1931} for analyzing a nonlinear dynamical system as a linear transform on a function space. 
In recent years, the Koopman operator has significantly evolved as a method in the data-driven analysis of dynamical systems, pioneered by \cite{Mezic2005, MezicBanaszuk2004} (see \cite{doi:10.1137/21M1401243} for a recent review).
The Koopman operator is also known as the {\em composition operator} and has developed at the interface of operator theory and analytic function theory, originating from \cite{Ryff66, Nordgren_1968}.
Over the past few decades, research has spanned various fields in analysis in mathematics, including real analysis, complex analysis, and harmonic analysis.

Dynamic Mode Decomposition (DMD) has been actively studied as a data-driven estimation method for the Koopman operator.
DMD was originally introduced as a numerical method to factorize complex dynamics into simple and essential components in fluid dynamics \cite{Schmid2009, SCHMID_2010}, and its connection with the Koopman operator has been discussed \cite{ROWLEY_MEZIC_BAGHERI_SCHLATTER_HENNINGSON_2009}.
Since DMD can be considered as a finite-dimensional approximation of the Koopman operator, linear algebraic operations such as eigendecomposition of the approximation matrix enable the decomposition of a dynamical system into sums of simple components  and the extraction of an essential component governing its complex behavior.
As a result, DMD and Koopman analysis have been attracting attention in recent years as a data-driven analysis method for dynamical systems, leading to a vast number of applications across a wide range of fields,  
for example, fluid dynamics \cite{Giannakis_Kolchinskaya_Krasnov_Schumacher_2018,SCHMID_2010,ROWLEY_MEZIC_BAGHERI_SCHLATTER_HENNINGSON_2009, Mezic2013_fluid}, epidemiology \cite{Proctor2015-xa}, neuroscience \cite{Brunton2015-zp, MS20}, plasma physics \cite{RKMNB18, KMHB20},  quantum physics \cite{Klus_2022}, finance \cite{MK16}, robotics \cite{BSVJA15, AM19, Vasudevan-RSS-19, BFGRV21}, the power grid \cite{Susuki2011, Susuki11}, and machine learning \cite{Kawahara, TKY17, TKY17-2, Takeishi_Kawahara_2021, TFTK22, kostic2023sharp, bevanda2023koopman, meanti2023estimating, ohnishi2021koopman, giannakis2023koopman}.
See \cite{Schmid22} for recent overview of variants of DMD.

Among the many extensions of DMD, Extended DMD (EDMD) \cite{Williams2015} is known as one of the most general and basic extensions of DMD and provides a foundation for various DMD-based methods.
EDMD can be described as a method for estimating the Koopman operator based on a set of observables with trajectory data from a dynamical system, and the Koopman operator is approximated through simple linear algebraic operations using the observables and the trajectory data.
Here, it is important to note that, in estimating the Koopman operator, at least implicitly, one function space containing the observables is fixed, and the Koopman operator is considered to be defined on the fixed function space.


A problem lies in the fact that, for a given set of observables, there may exist (infinitely) many function spaces that can include these observables, while the mathematical properties of the Koopman operator can drastically differ depending on the function space where it acts.
For example, on $L^2$-spaces with invariant measures, they act as unitary operators, whereas on other spaces, they might not even be bounded operators.
As mentioned above, the Koopman operator has been extensively studied in mathematical literature.
For example, there exist many studies on the operators' fundamental properties such as boundedness (see, for example, \cite{Zh07, CM95, Taniguchi2023, ISHIKAWA2023109048, SS17} and references therein).
These studies also have shown that the properties of Koopman operators drastically change depending on the choice of function spaces.
That is, EDMD actually captures information of the Koopman operator on a specific function space determined by the chosen observables.
This problem is continuous with the major challenges faced by DMD-based methods, such as spectral pollution, invariant subspaces, continuous spectrum, and the treatment of chaotic dynamics, as mentioned in \cite{colbrook_ayton_szoke_2023,ColbrookTownsend24}.
Therefore, we require a detailed analysis of both observables and function spaces.
However, in previous studies, the relationship between observables and function spaces has not been sufficiently discussed.

In this paper, we focus on the Koopman operator defined on a reproducing kernel Hilbert space (RKHS) and propose a fundamental solution to the significant challenge on the estimation of the Koopman operator by mathematically analyzing the relationship between the function spaces and observables.
An RKHS is formally defined as a Hilbert space composed of functions where any point evaluation is a continuous linear functional.
RKHS provides a general and effective theoretical framework through functional analysis, establishing a solid field in mathematics \cite{MR3560890_SaitohSawano} and machine learning \cite{HSS08}, with applications for various problems within each domain.
The Koopman operator on RKHS has been studied in complex analysis, and rapidly developed and extended in machine learning and data analysis in recent years \cite{Kawahara,HIIMK20, DAS2020573, Klus2020, HIIKKK21, IIS22, hashimoto2024koopmanbased, kostic2023sharp, bevanda2023koopman, meanti2023estimating, ohnishi2021koopman, FUJII201994}. 

The core notion we introduced is the space of {\em intrinsic observables} for estimating the Koopman operator in RKHS, constructed utilizing a geometric structure called jet on a fixed point of the dynamical system.
Jet is a geometric formulation of the Taylor expansion, providing a geometrically canonical object on a manifold as a higher-order counterpart of the tangent bundle (see, for example, \cite[Section 4]{KPS_NatOpMan93}).

Leveraging the intrinsic observables, we propose a new estimation method, referred to as {\em Jet Extended Dynamic Mode Decomposition (JetEDMD)}, that provides refinement of a special class of EDMDs, such as EDMD with monomials, and significantly improves the estimation performance of the Koopman operator (see also Section \ref{sec:NumericalSimulation} for its detailed algorithms and numerical results\footnote{The code for the numerical simulation is available at \url{https://github.com/1sa014kawa/JetEDMD}.}).
For instance, JetEDMD provides much more precise depiction in numerical estimation of eigenvalues than EDMD as shown in Figure \ref{fig: difference of spectra quadratic map}.
It is also noteworthy that we obtain a new interpretation of EDMD within the framework of JetEDMD.
For example, EDMD using monomials can be rephrased as ``JetEDMD without truncation using an exponential kernel for a dynamical system with a fixed point at the origin.''

\begin{figure}[t]
    \captionsetup{width=1.0\linewidth}
    \centering
    \includegraphics[keepaspectratio, width=0.7\linewidth]{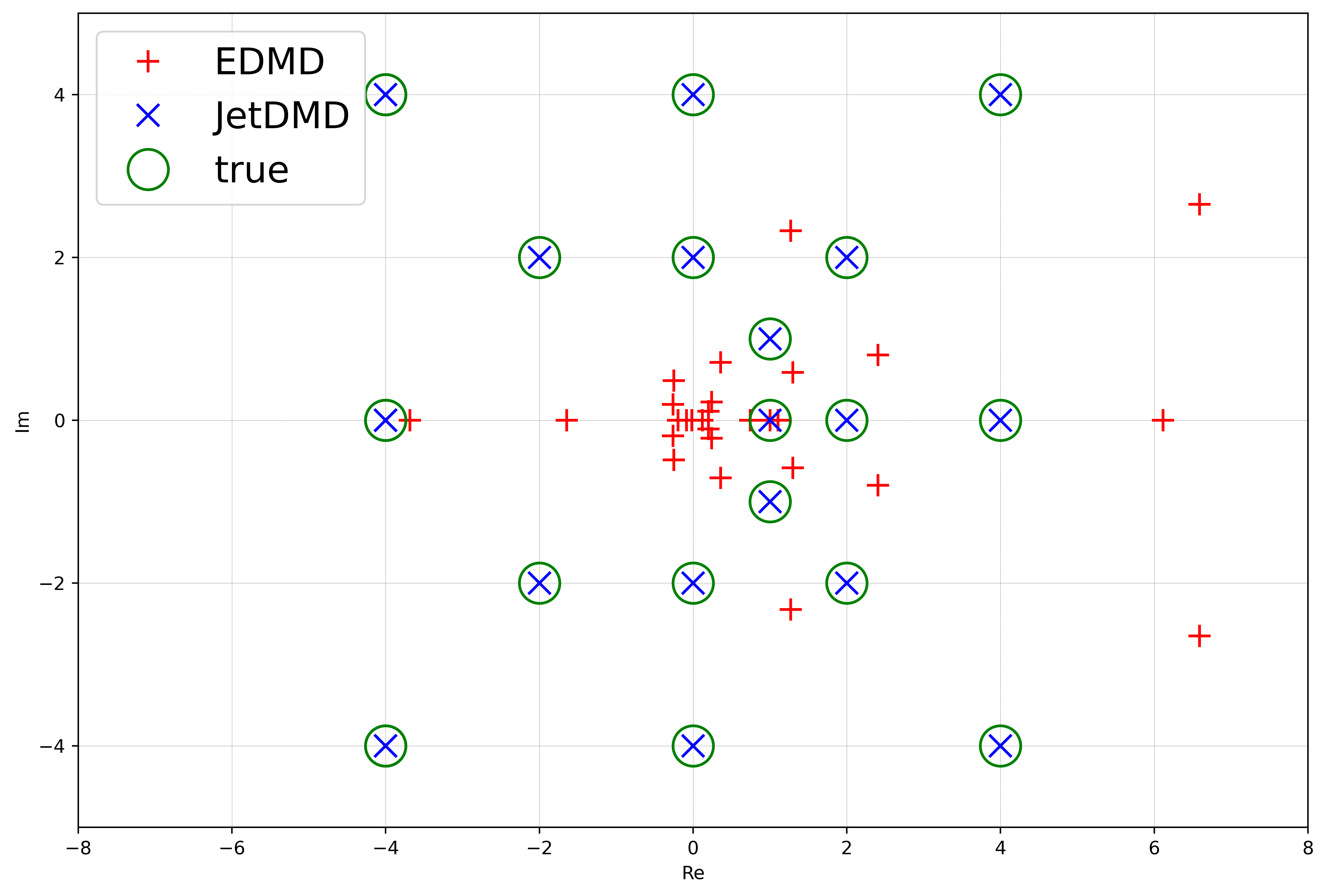} 
    \caption{
    The comparison of the computed eigenvalues of the dynamical system $f(x,y) := (x^2 - y^2 + x - y, 2xy + x + y)$ on $\mathbb{R}^2$ using data of $100$ pairs of sample from the uniform distribution on $[-1,1]^2$ and their images under $f$.
    The red $+$'s indicate the eigenvalues computed via EDMD using monomials of degree up to $10$.
    The blue $\times$'s indicate the estimated eigenvalues with JetEDMD.
    The green circles indicate set $\{\lambda_+^m\lambda_-^n\}_{m + n \le 5}$, where $\lambda_\pm = 1 \pm \mathrm{i}$ are the eigenvalues of the Jacobian matrix of $f$ at the origin.
    }
    \label{fig: difference of spectra quadratic map}
\end{figure}

Furthermore, we elucidate the mathematical machinery behind the performance of JetEDMD as described in Figure \ref{fig: difference of spectra quadratic map}.
In fact, the superior performance of JetEDMD is achieved through the surprisingly simple operation of the ``truncation to a leading principal submatrix'', and our paper clarifies the reason why this works.
We provide an explicit error bound for JetEDMD, and for the Gaussian kernel and the exponential kernel, we prove the convergence of JetEDMD with explicit convergence rate.
As a result, we show that eigenvalues and eigenvectors estimated by JetEDMD actually converge to the corresponding theoretical counterparts.
It implies that JetEDMD provides an alternative solution for the spectral pollution that always appears in the estimation of the Koopman operator.
We also emphasize that we do not impose the boundedness of the Koopman operator, and our proofs are based on a detailed mathematical analysis in RKHS.
Although some convergence results for EDMD proved in \cite{Korda2018}, their results do not cover ours because they essentially impose the boundedness on the Koopman operator and basically obtain a weak convergence of eigenvectors and eigenvalues.

In addition, JetEDMD provides a novel interpretation for the ``Koopman eigenfunctions'' estimated using observables and rigorous eigendecomposition of the ``extended'' Koopman operator.
In DMD-based methods including EDMD, the interpretation of the estimated eigenvalues and eigenfunctions has always been a significant issue, and it is an essential element in extracting quantitative information from complex and chaotic dynamical systems.
When the Koopman operator does not preserve the space of observables, the question of what the eigenfunctions estimated by EDMD represent has always been a problem.
As an answer to this problem, we show that the ``Koopman eigenfunctions'' generally exist outside the function space and that the space to which they truly belong can be captured using the theory of rigged Hilbert space.
The rigged Hilbert space is defined as a Hilbert space equipped with a dense linear subspace with a finer topology and embedded into the dual space of the dense subspace.
The triple composed of the Hilbert space, the dense linear subspace,  and its dual space is referred to the Gelfand triple.
For more detail of the accurate formulation and application, see Section \ref{sec: generalized spectrum}, \cite{BG89, CHIBA2015324, Gel2, Gel, CII} and references therein. 

More precisely, we show that JetEDMD actually approximates the ``extended'' Koopman operator, not Koopman operator itself, in the rigged RKHS defined by a Gelfand triple constructed from the space of intrinsic observables and its dual space.
This ``extended'' Koopman operator is defined as the dual operator of the Perron-Frobenius operator on the space of the observables, inducing a continuous linear map on a linear topological space (not function space) that includes the original RKHS, and its eigenvectors are not necessarily functions.
Our framework provides a theoretical methodology to deal with the estimated ``Koopman eigenfunctions'' using the observables as the approximation of the eigenvector of the ``extended'' Koopman operator in the rigged RKHS whether Koopman operator preserves the space of observables or not.
We also note that considering the Koopman operators as the dual of Perron-Frobenius operator also appears in \cite{SLIPANTSCHUK2020105179} and they have rigorous convergence result for EDMD for analytic maps on the circle.
Although the spectrum of the original Koopman operator on RKHS is quite complicated and can include continuous spectrum, the spectrum of the ``extended'' Koopman operator is much simpler and clearer, and it is possible to prove the Jordan-Chevalley decomposition and the eigendecomposition of the ``extended'' Koopman operator in the dual space of the observables.
Therefore, our framework provides an effective approach to analyze the estimated eigenvectors and the spectrum of the Koopman operators via the theory of rigged Hilbert space.


As an application of JetEDMD, we propose a method for reconstructing the original dynamical system from temporally-sampled trajectory data or vector field data of the dynamical system.
As shown in Figure \ref{fig: intro_reconstruction}, our method is capable of accurately reconstructing the chaotic system such as the Lorenz attractor from temporally-sampled trajectory data.
Our method also significantly generalizes the lifting method proposed in \cite{MauroyGoncalves2020} based on JetEDMD, and we theoretically prove the reconstruction capability of a broad class of analytic dynamical systems with arbitrary precision if we have sufficient number of samples and computational resources.
\begin{figure}
\begin{minipage}[c]{1.0\linewidth}
    \begin{minipage}[c]{0.24\linewidth}
      \includegraphics[keepaspectratio, width=\linewidth]{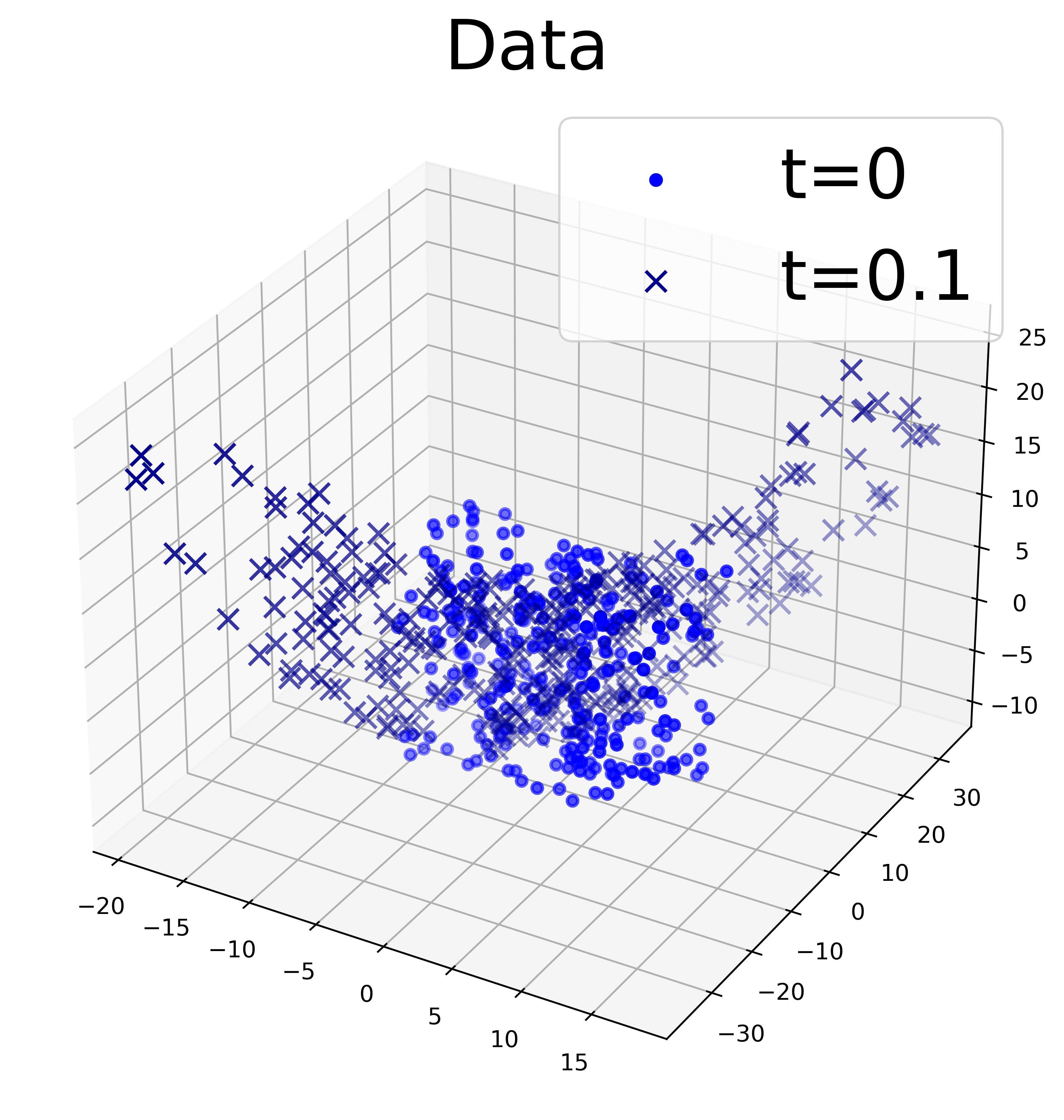}
    \end{minipage}\hfill 
    \begin{minipage}[c]{0.24\linewidth}
      \includegraphics[keepaspectratio, width=\linewidth]{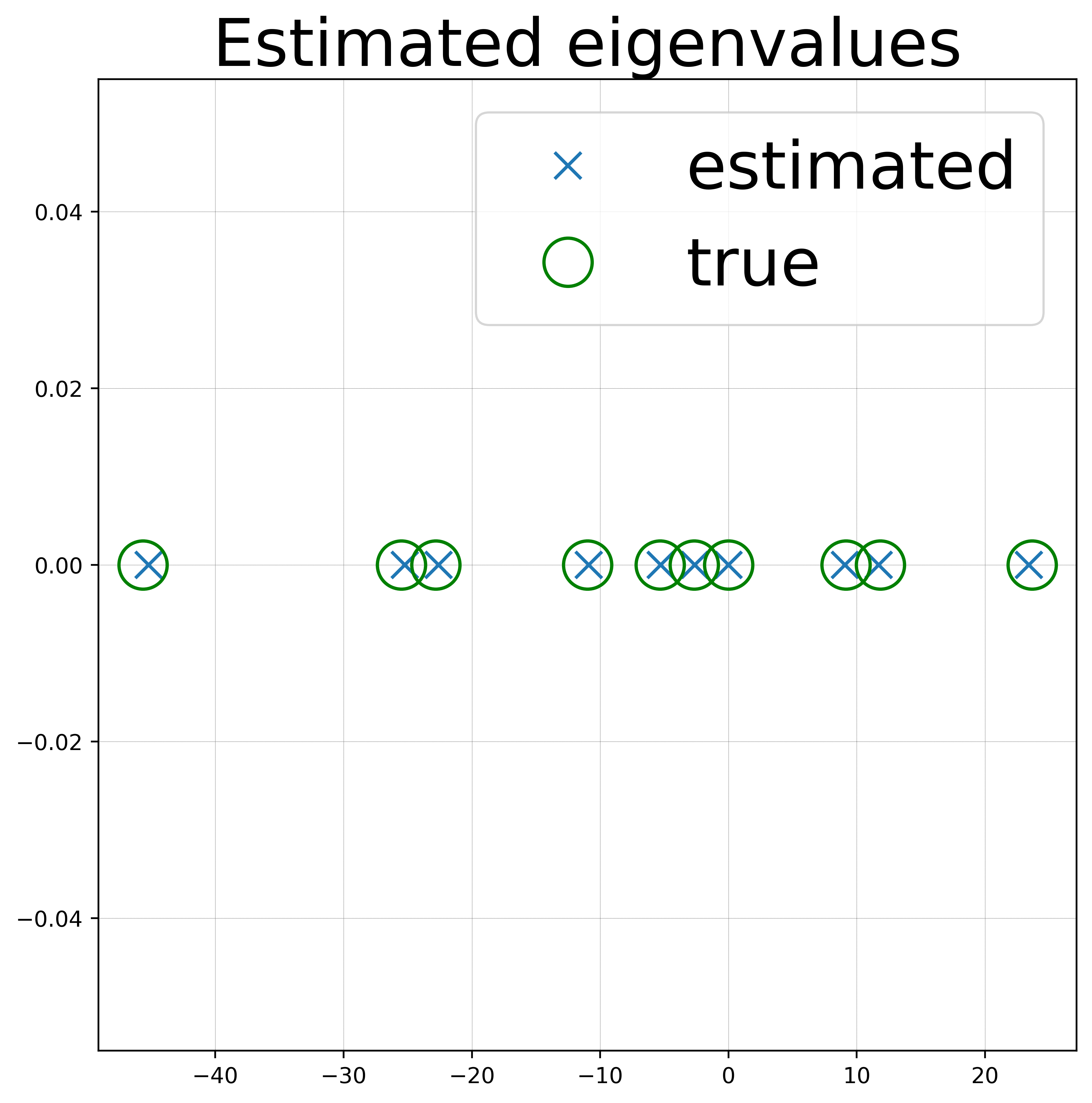}
    \end{minipage}\hfill 
    \begin{minipage}[c]{0.24\linewidth}
      \includegraphics[keepaspectratio, width=\linewidth]{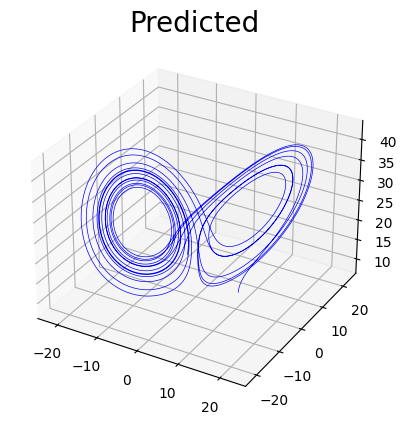}
    \end{minipage}\hfill 
    \begin{minipage}[c]{0.24\linewidth}
      \includegraphics[keepaspectratio, width=\linewidth]{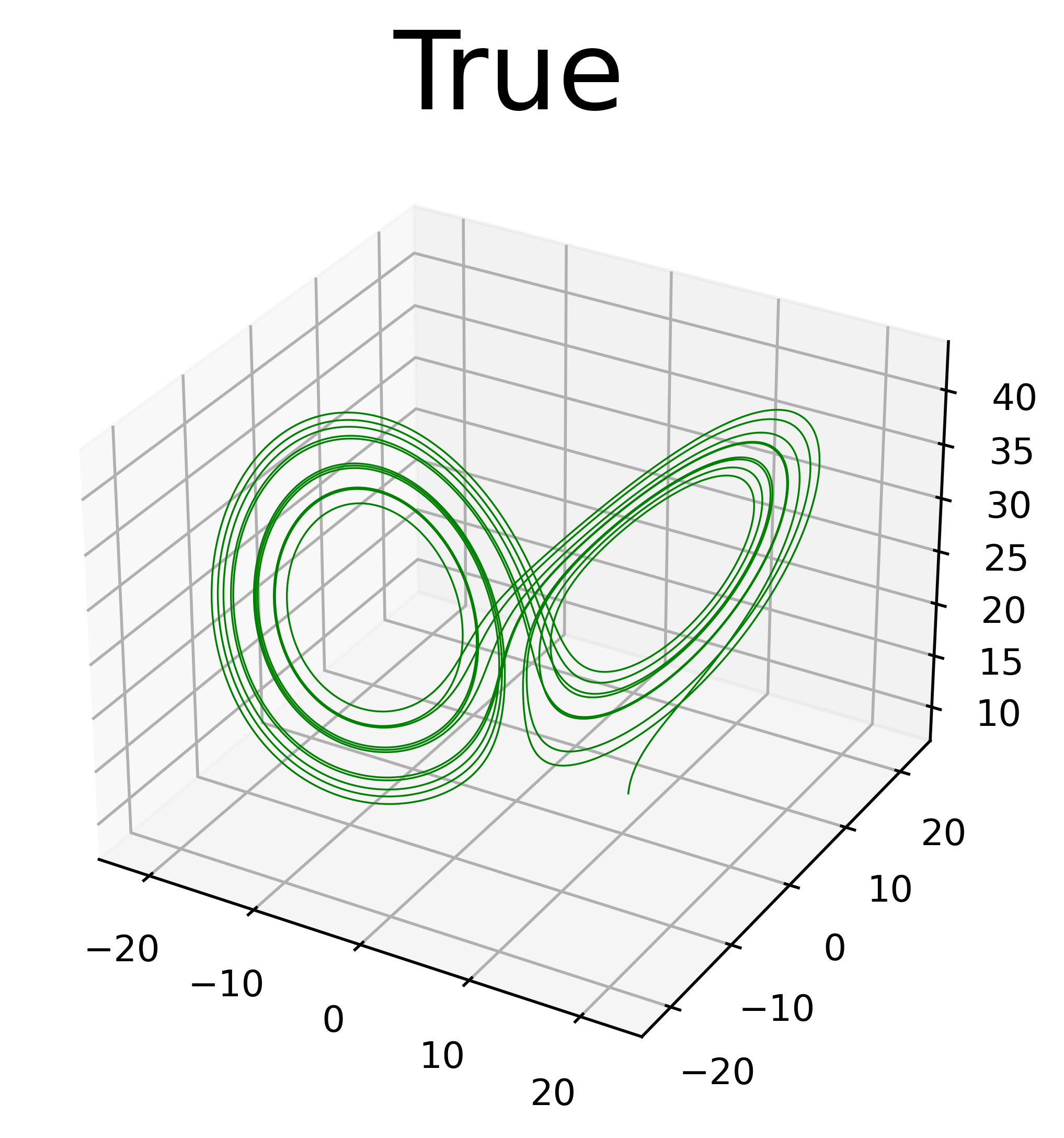}
    \end{minipage}
\end{minipage}
\caption{
Illustration of the data-driven reconstruction of the Lorenz system with Algorithm \ref{algorithm: reconstruction, conti. with disc.}, corresponding to the right images of Figure \ref{fig: reconstruction_lorenz, discrete}. {\bf Left:} data used for estimation—300 input-output pairs, with inputs sampled uniformly from $[-10,10]^3$ and outputs given by the flow map at time $0.1$. {\bf Middle-left:} estimated and theoretical eigenvalues of the generator of the Koopman operator. {\bf Middle-right:} trajectory of the dynamical system reconstructed by our method. {\bf Right:} true trajectory.
}
\label{fig: intro_reconstruction}
\end{figure}

The structure of this paper is as follows:
In Section \ref{sec: theory}, we introduce the notion of jet, construct the canonical invariant subspaces in RKHS, and prove their basic properties.
In Section \ref{sec: datadriven estimation of PF operator}, we provide several theoretical results for data-driven estimation of Perron-Frobenius operator and show the error bound in the general setting.
In Section \ref{sec: generalized spectrum}, we introduce the notion of Gelfand triple and rigged Hilbert space.
We define the extended Koopman operator $C_f^\times$ and prove the  Jordan-Chevalley decomposition of the extended Koopman operator.
We also show the eigendecomposition of the extended Koopman operator under mild condition.
In Section \ref{sec: theory for continuous dynamical systems}, we provide the framework for estimating the generator of the Koopman operator using discrete data of vector fields.
We also show the corresponding result for continuous dynamical systems to those in the previous sections.
In Section \ref{sec: special kernels}, we focus on the specific positive definite kernel, the Gaussian kernel and exponential kernel, respectively.
We provide sufficient conditions for the technical assumptions for the convergence results in the previous sections, and we also present more explicit error bounds.
In Section \ref{sec:NumericalSimulation}, we propose JetEDMD and provide its algorithm and perform numerical simulations.
We illustrate the estimated eigenvalues for the van der Pol oscillator, the Duffing oscillator, and H\'enon map, and depict the approximated eigenfunctions of the extended Koopman operators.
Furthermore, we describe an application of JetEDMD for data-driven reconstruction of dynamical systems.
We also provide the theoretical guarantee of the algorithms as well and include discussion on the comparison of JetEDMD with other popular methods, EDMD and Kernel DMD.

\subsection{Notation}\label{subsec: notation}
We denote the set of the real (resp. complex) numbers by $\mathbb{R}$ (resp. $\mathbb{C}$).
We denote by $\mathrm{i} \in \mathbb{C}$ the imaginary unit.
We denote the set of integers by $\mathbb{Z}$.
For any subset of $S \subset \mathbb{R}$, we denote by $S_{>0}$ (resp. $S_{\ge 0}$) the set of positive (resp. non-negative) elements of $S$.
We denote by $\partial_{x_j}$ the partial derivative with respect to the $j$-th variable of a differentiable function on an open set of $\mathbb{R}^d$.
We denote by ${\rm d}x$ the Lebesgue measure on $\mathbb{R}^d$.

We denote by $\mathbb{R}^{m \times n}$ (resp. $\mathbb{C}^{m \times n}$) the set of real (resp. complex) $m \times n$ matrices.
We define the Frobenius norm and the operator norm of the matrix as $\|\cdot\|$ and $\|\cdot\|_{\rm op}$, respectively.
We usually regard $x \in \mathbb{C}^d$ as an element of $\mathbb{C}^{d \times 1}$ and describe the Euclidean norm by $\|x\|$.
We always use the bold uppercase letters to describe the matrices, like $\mathbf{A}$, $\mathbf{B}$, $\mathbf{C}$.
For a matrix $\mathbf{A}$, we denote the transpose (resp. adjoint) matrix of $\mathbf{A}$ by $\mathbf{A}^\top$ (resp. $\mathbf{A}^* :=\overline{\mathbf{A}}^\top$).
We denote the identity matrix and the zero matrix of size $n$ by $\mathbf{I}_n$ and $\mathbf{O}_n$, respectively.

For any subset $S\subset V$ of a complex vector space $V$, we denote by ${\rm span}(S)$ the linear subspace generated by $S$.

For a function $h: X \to \mathbb{C}$ on a topological space $X$,
 we define ${\rm supp}(h)$ the closure of the set $\{x \in X: h(x) = 0\}$.

We basically use the multi-index notation.
For example, for $\alpha = (\alpha_1,\dots, \alpha_d) \in \mathbb{Z}_{\ge 0}^d$ and  $z=(z_1, \dots, z_d) \in \mathbb{C}^d$, we write $z^\alpha := z_1^{\alpha_1}\dots z_d^{\alpha_d}$.
We also define $|\alpha| := \sum_{i=1}^d \alpha_i$, $\partial_{x}^\alpha := \partial_{x_1}^{\alpha_1}\dots\partial_{x_d}^{\alpha_d}$, and  $\alpha! := \prod_{i=1}^d \alpha_i!$ for $\alpha = (\alpha_1,\dots, \alpha_d) \in \mathbb{Z}_{\ge 0}^d$

For an open subset $\Omega \subset \mathbb{R}^d$, we denote by $\mathfrak{E}(\Omega)$ (resp. $\mathfrak{H}(\Omega)$) the space of $\mathbb{C}$-valued infinitely differentiable functions (resp. real analytic functions) on $\Omega$.
The topology of $\mathfrak{E}(\Omega)$ is the uniform convergence topology of the higher derivatives on the compact subsets (see \cite[p.26--27]{Yosida1980}).
The topology of $\mathfrak{H}(\Omega)$ is the inductive limit topology of $\displaystyle\lim_{\longrightarrow}\mathfrak{H}(U)$.
Here, we denote by $\mathfrak{H}(U)$ the space of holomorphic function on an open subset $U \subset \mathbb{C}^d$, equipped with the uniform convergence topology on the compact subsets, and $\mathfrak{H}(U)$'s constitute the inductive system via the restriction maps and $U$ varies over the open subset of $\mathbb{C}^d$ such that $U \cap \mathbb{R} = \Omega$ (see \cite{MARTINEAU1966} for details).

For a topological vector space $V$, we denote the dual space with the strong topology by $V'$.
For $\mu \in V'$ and $h \in V$, we denote $\mu(h)$ by $\langle \mu \mid h \rangle$.
For a continuous linear map $L: V_1 \rightarrow V_2$ between topological vector spaces $V_1$ and $V_2$, we denote by $L'$ the dual linear operator $L': V_2' \rightarrow V_1'$.

For a set $X$ and $a,b \in X$, we define the Kronecker delta by
\begin{align*}
    \delta_{a,b} = \begin{cases}
        1 & \text{ if $a = b$},\\
        0 & \text{otherwise.}
    \end{cases}
\end{align*}

For a family of real numbers $\{a_\lambda\}_{\lambda \in \Lambda}$, we denote by $a_\lambda \lesssim b_\lambda$ if there exists $C>0$ independent of $\lambda \in \Lambda$ such that $a_\lambda \le C b_\lambda$ for all $\lambda$.

\subsection{Summary of main results and their implications}\label{subsec: overview}

In the following, we overview our main results for the convenience of the readers. 
We introduce minimal notation and describe the essence of our main results.
Here, although we only consider the discrete dynamical system, we have corresponding results to the continuous dynamical system, and provide a theoretical framework for estimating the generator using discrete data of vector fields.
See Section \ref{sec: theory for continuous dynamical systems} for details.
We will use notation described in Section \ref{subsec: notation}.

Let ${\Omega} \subset \mathbb{R}^d$ be a connected open subset and let $p \in \Omega$.
Let $f: {\Omega} \to {\Omega}$ be a map of class $C^\infty$.
Let $\lambda_1,\dots, \lambda_d$ be the eigenvalues (with multiplicity) of the Jacobian matrix ${\rm d}f_p$ at $p$ and let $\lambda = (\lambda_1,\dots, \lambda_d)$.
In this paper, as the function space where the Koopman operator acts, we always consider the RKHS $H$ associated with a positive definite kernel $k: \Omega \times \Omega \to \mathbb{C}$ of class $C^\infty$.
Readers should refer to, for example, \cite{MR3560890_SaitohSawano} for the general theory of RKHS.
\begin{example}
Let $\Omega = \mathbb{R}^d$ and $k(x,y) = e^{x^\top y}$, the RKHS $H$ associated with $k$ is explicitly described as follows:
\begin{align*}
H = \left\{ h|_{\mathbb{R}^d} : h\text{ is holomorphic on }\mathbb{C}^d\text{ and }\int_{\mathbb{R}^d \times \mathbb{R}^d} |h(x+y \mathrm{i})|^2 e^{-\|x\|^2-\|y\|^2}\,{\rm d}x{\rm d}y < \infty\right\}.
\end{align*}
This space is usually called the Fock space \cite{Zhu12} and recently treated in the context of the Koopman operator \cite{Mezic2020}.
\end{example}

We define the {\em Koopman operator} $C_f: H \to H$ by the linear operator defined by $C_f h := h \circ f$ with domain $D(C_f) := \{h \in H : h \circ f \in H\}$.
We assume that $C_f$ is densely defined on $H$, namely, $D(C_f)$ is a dense subspace of $H$.
We denote the Perron--Frobenius operator, that is the adjoint operator of $C_f$, by $C_f^*$.

An important fact here is that the RKHS associated with the positive definite kernel $k$ of class $C^\infty$ is characterized as a Hilbert space satisfying the following two conditions: it is included in the space $\mathfrak{E}(\Omega)$ of the functions of class $C^\infty$ and the inclusion map \[\iota: H \hookrightarrow \mathfrak{E}(\Omega)\]
is continuous \cite[Theorem 2.6]{6fc4e20b-e6b8-3651-8988-bc9975065f31}.
In the function space $\mathfrak{E}(\Omega)$, we can naturally define the continuous linear map $f^*:\mathfrak{E}(\Omega) \to \mathfrak{E}(\Omega);~h \mapsto h \circ f$, known as the pull-back, equivalent notion of the Koopman operator.
Since RKHS is a subspace of $\mathfrak{E}(\Omega)$, the Koopman operator $C_f$ on $H$ can be regarded as the restriction of the pull-back $f^*$ to $H$.
In terms of category theory, the functor $M \mapsto \mathfrak{E}(M)$ from the category of smooth manifolds to the category of $\mathbb{C}$-algebras is fully faithful \cite{NevSancho03}, and  the pull-back $f^*$ (equivalently, Koopman operator) on $\mathfrak{E}(\Omega)$ provides an algebraic counterpart that is entirely equivalent to the dynamical system on $\Omega$.
Therefore, RKHS naturally appears as a machinery for numerically handling the universal object $f^*: \mathfrak{E}(\Omega) \to \mathfrak{E}(\Omega)$, and provides a general theoretical framework for numerical treatments.
In this sense, JetEDMD can be described as a method for data-driven estimation of the pull-back $f^*$ on $\mathfrak{E}(\Omega)$ using RKHS.

Let us introduce {\em jet}.
The space $\mathfrak{J}_{p,n}$ of $n$-jet at $p \in \Omega$ is defined as the quotient space $\mathfrak{E}(\Omega)/\sim$, where the equivalence relation $\sim$ is defined as $h \sim g$ if and only if the $\partial^\alpha h(p) = \partial^\alpha g(p)$ for $|\alpha| \le n$, where we use the multi-index notation.
We denote the natural surjection $\mathfrak{E}(\Omega) \to \mathfrak{J}_{p,n}$ by $\mathfrak{j}_{p,n}$.
We note that the space of $n$-jets at $p$ is geometrically canonical object, namely, it is invariant under the change of coordinate.

Let $p \in \Omega$. 
For each integer $n \ge 0$, we introduce a finite-dimensional subspace $V_{p,n} \subset H$ as follows:
\begin{align*}
    V_{p,n} &=  \sum_{|\alpha| \le n} \mathbb{C} \cdot \left.\partial_x^\alpha k(x,\cdot)\right|_{x=p},\\
    V_p &=  \bigcup_{n \ge 0} V_{p,n} = \sum_{\alpha \in \mathbb{Z}_{\ge 0}^d} \mathbb{C} \cdot \partial_x^\alpha k(x,\cdot)|_{x=p}. 
\end{align*}
This space $V_p$ is what we call the space of {\em intrinsic observables}.
We note that we can intrinsically define $V_{p,n}$ as the image under the dual map of the composition of continuous  linear maps $H \overset{\iota}{\to} \mathfrak{E}(\Omega) \overset{\mathfrak{j}_{p,n}}{\to} \mathfrak{J}_{p,n}$ (see Section \ref{sec: theory} for details).
Originally, $V_{p,n}$ is introduced in \cite{ISHIKAWA2023109048} as a key notion for connecting the boundedness of Koopman operators with the behavior of dynamical systems.
For special RKHSs, the equivalent notion is considered in \cite{Ikeda_Ishikawa_Sawano-2022} as a machinery for proving the linearity of dynamical systems with bounded Koopman operators.

As in the following proposition, the union of $V_{p,n}$ for $n \ge 0$ constitutes a dense subspace of $H$ and, if $p$ is a fixed point of $f$, $V_{p,n}$ is invariant under the Perron--Frobenius operator $C_f^*$ for all $n \ge 0$ (Theorem \ref{thm: basic theorem for PF operator, discrete} and Proposition \ref{prop: density}):
\begin{proposition}\label{prop:_intro_Vpn}
    The space $V_p$ is dense in $H$.
    Moreover, if $p$ satisfies $f(p) = p$, we have $C_f^*(V_{p,n}) \subset V_{p,n}$ for any $n \ge 0$.
    In addition, if $f(p)=p$ and the dual map $\iota'$ of $\iota$ is injective on the image of $\mathfrak{j}_{p,n}'$, the set of eigenvalues of $C_f^*|_{V_{p,n}}$ coincides with
    \[\left\{ \lambda^\alpha : |\alpha| \le n, \alpha \in \mathbb{Z}_{\ge 0}^d\right\}.\]
\end{proposition}
We note that RKHSs of the Gaussian kernel and the exponential kernel, mainly treated in this paper, satisfies the condition of the injectivity of $\iota'$ on the image of $\mathfrak{j}_{p,n}'$.

Now, we describe the JetEDMD (see Section \ref{sec:NumericalSimulation} for details).
As seen below, Jet DMD provides a refinement of EDMD with the intrinsic observables.
Let $r_n := \mathop{\rm dim}V_{p,n}$.
Let $\{v_i\}_{i=1}^\infty$ be an orthonormal basis of $V_p$ such that $\{v_1,\dots, v_{r_n}\}$ constitutes a basis of $V_{p,n}$. 
Then, for $x \in \Omega$, we define a vector of length $r_n$ by
\begin{align*}
     {\bf v}_{n}(x) &:= \left(\overline{v_i(x)}\right)_{i=1}^{r_n}.
\end{align*}
For $X = (x_1, \dots, x_N) \in {\Omega}^N$, we define a matrix of size $r_n \times N$ by
\begin{align}
    \mathbf{V}_{n}^X :=& \left({\bf v}_n(x_1), \dots, {\bf v}_{n}(x_N)\right).
\end{align}
Let $m \le n$.
Let $y_i = f(x_i)$ for $i=1,\dots,N$ and define $Y:=(y_1,\dots,y_N)$.
Then, we consider the matrix $\widehat{\mathbf{C}}_{m,n,N} \in \mathbb{C}^{r_m \times r_m}$ defined by the leading principal submatrix of order $r_m$ of $\mathbf{V}_n^Y (\mathbf{V}_n^X)^\dagger$, namely, 
\[
\mathbf{V}_n^Y (\mathbf{V}_n^X)^\dagger = 
\begin{pmatrix}
    \widehat{\mathbf{C}}_{m,n,N} & *\\
    * &*
\end{pmatrix},
\]
where $(\cdot)^\dagger$ is the Moore-Penrose pseudo-inverse.
We remark that EDMD corresponds to the case of $m=n$.
According to the following example, EDMD with monomials is recovered in JetEDMD with exponential kernel:
\begin{example}
    Let $\Omega = \mathbb{R}^d$ and let $k(x,y) = e^{x^\top y}$.
    Then, we can take the following set of polynomials of degree up to $n$
    \[\left\{1,~x_1,~\dots,~x_d,~\dots, \frac{x_1^{\alpha_1}\cdots x_d^{\alpha_d}}{\sqrt{\alpha_1!\cdots\alpha_d!}},~\dots,~\frac{x_d^n}{\sqrt{n!}}\right\}\]
    as the orthonormal basis of $V_{0,n}$ (here, we take $p=0$) and we have $r_n = \binom{n+d}{d}$.
    Figure \ref{fig: difference of spectra quadratic map} actually describes the computational result in this setting.
    Using this orthonormal basis, the matrix $\mathbf{V}_n^Y (\mathbf{V}_n^X)^\dagger$ essentially coincides with the one considered in EDMD with monomials.
    In Figure \ref{fig: difference of spectra quadratic map}, we considered $f(x,y) := (x^2 - y^2 + x - y, 2xy + x + y)$, and performed the calculation using $x_1,\dots,x_{100}$ from the uniform distribution on $[-1,1]^2$.
    The the red $+$'s indicates the eigenvalues of $\widehat{\mathbf{C}}_{n,n,N}$ for $n=10$, corresponding to EDMD, and the blue $\times$'s indicates those of $\widehat{\mathbf{C}}_{m,n,N}$ for $(m,n)=(5,10)$, corresponding to JetEDMD.
\end{example}
When we take $N \to \infty$ and $n \to \infty$, the matrix $\widehat{\mathbf{C}}_{m,n,N}$ with JetEDMD actually converges to the correct target $C_f^*|_{V_{p,m}}$, more precisely, we have the following theorem (Theorem \ref{thm: error analysis for PF operator for algorithm, disc} and Theorem \ref{thm: explicit asymptotic, disc} for the general and precise statements):
\begin{theorem}\label{thm:_intro_convergence}
    Let $m \le n$.
    Let $\Omega = \mathbb{R}^d$ and let $p \in \Omega$.
    Let $k(x,y) = e^{(x-p)^\top (y-p)}$ or $k(x,y) = e^{-|x-y|^2/2\sigma^2}$.
    Let $f=(f_1,\dots, f_d): \mathbb{R}^d \to \mathbb{R}^d$ be a map such that $f(p) = p$ and $V_p \subset D(C_f)$. 
    Let $x_1,\dots, x_N$ be i.i.d random variables of the distribution with compactly supported density function $\rho$ such that ${\rm ess.inf}_{x \in U}\rho(x) >0$ for some open subset $U \subset \mathbb{R}^d$.
    Let $y_i = f(x_i)$ for $i=1,\dots, N$.
    Then, there exist $b_m>0$ and $R>0$ such that we have
    \begin{align}
        \lim_{N\to \infty} \left\|\widehat{\mathbf{C}}_{m,n,N} -\mathbf{C}_{f,m}^\star\right\| 
        \le& \left\|C_f|_{V_{m,p}}\right\|_{\rm op} \cdot b_mn^m \cdot \frac{R^n}{\sqrt{n!}}~~~{\rm a.e.} \label{intro: error bound}\\
        &\underset{n \to \infty}{\longrightarrow}~~0~~~{\rm a.e.} \notag
    \end{align}
    where $\mathbf{C}_{f,m}^\star \in \mathbb{C}^{r_m \times r_m}$ is the representation matrix of $C_f^*|_{V_{p,m}}$ with respect to the orthonormal basis $\{v_1,\dots, v_{r_m}\}$.
\end{theorem}
This theorem clearly explains why the step of ``truncating $\mathbf{V}_n^Y (\mathbf{V}_n^X)^\dagger$ to a leading principal submatrix'' is inevitable.
From the proofs of Theorem \ref{thm: error analysis for PF operator for algorithm, disc} and Theorem \ref{thm: explicit asymptotic, disc},  in \eqref{intro: error bound}, the constant $\left\|C_f|_{V_{m,p}}\right\|_{\rm op}$ 
arises from extracting the first $r_m$ rows of the matrix $\mathbf{V}_n^Y (\mathbf{V}_n^X)^\dagger$, while the constant $b_mn^m$ appears as a result of extracting the first $r_m$ columns.
Since $V_p$ is dense in $H$, we easily see that the sequence $\big\{ \left\|C_f|_{V_{p,m}}\right\|_{\rm op} \big\}_{m \ge 0}$ is bounded if and only if the Koopman operator $C_f$ is a bounded linear operator on $H$.
However, according to \cite{CMS} and \cite{Ikeda_Ishikawa_Sawano-2022}, no nonlinear dynamics induces a bounded Koopman operator on $H$.
Thus, we have the divergence $\|C_f|_{V_{p,n}}\|_{\rm op} \to \infty$ as $n \to \infty$ for the nonlinear dynamical system $f$.
This fact implies the difficulty of analyzing EDMD, corresponding to JetEDMD with $m=n$, and it generally suggests that EDMD fails to estimate the target operator since the term $\left\|C_f|_{V_{p,n}}\right\|_{\rm op} \cdot b_nn^n$ will rapidly diverge if we take $n \to \infty$.
We also emphasize that here we prove this theorem under the assumption that $C_f$ is densely defined, but we do not assume its boundedness.

As explained above, the matrix $\widehat{\mathbf{C}}_{m,n,N}$ produced by JetEDMD is capable of approximating the Perron-Frobenius operator $C_f^*|_{V_{p,m}}$ restricted to $V_{p,m}$.
Thus, by considering the adjoint of $\widehat{\mathbf{C}}_{m,n,N}$ in $V_{p,m}$, we can estimate the adjoint linear map $(C_f^*|_{V_{p,m}})^*$ on $V_{p,m}$ (note that the first ``$*$'' means the adjoint in $H$, but the second one means the adjoint in $V_{p,m}$).
While it seems that the limit of $(C_f^*|_{V_{p,m}})^*$ along $m$ becomes the Koopman operator $C_f$ on $H$, it actually converges not to the Koopman operator itself, but to the ``extended'' Koopman operator defined through the Gelfand triple.

Since a rigorous explanation of the Gelfand triple and the extended Koopman operator is provided in Section \ref{sec: generalized spectrum}, here we will explain the ``extended'' Koopman operator in a somewhat rough manner.
Let $\Phi := V_p$.
We equip $\Phi$ with the inductive limit topology.
Then, $(\Phi, H, \Phi')$ determine the Gelfand triple and satisfies the inclusion relation $\Phi \subset H = H' \subset \Phi'$, where $(\cdot)'$ indicates the dual space with the strong topology, and we identify $H$ with its dual via the Riesz representation theorem.
We usually call a Hilbert space equipped with a Gelfand triple the {\em rigged Hilbert space}.
The Gelfand triple is introduced for further investigation of the spectrum of linear operators, and it has been well studied in quantum mechanics.
For more detail of the mathematical formulation and applications, see \cite{BG89, CHIBA2015324, Gel2, Gel} and references therein. 
Since $C_f^*(\Phi) \subset \Phi$ by Proposition \ref{prop:_intro_Vpn}, the dual operator of $C_f^*|_\Phi$ induces a continuous linear operator on $\Phi'$.
We denote this induced continuous linear operator by $C_f^\times:=(C_f^*|_\Phi)': \Phi' \to \Phi'$ and call it the {\em extended Koopman operator}.
It can be shown that $C_f^\times$ is actually the extension of Koopman operator, namely, it satisfies $C_f^\times|_{H} = C_f$ on the domain of $C_f$ (Proposition \ref{prop: extension}) and is regarded as the limit of $(C_f^*|_{V_{p,n}})^*$.
Moreover, we can prove the Jordan--Chevalley decomposition of $C_f^\times$ on $\Phi'$ (Theorem \ref{thm: jordan decomposition of C_f}) and the eigendecomposition of $C_f^\times$ on $\Phi'$ under mild conditions as follows:
\begin{theorem}\label{thm: intro eigendecomposition}
    Assume that  $\iota'$ is injective on the image of $\mathfrak{j}_{p,n}'$ and that $\lambda^\alpha \neq \lambda^\beta$ for $\alpha, \beta \in \mathbb{Z}_{\ge 0}^d$ with $\alpha \neq \beta$.
    Then, there exist families of vectors $\big\{w_{\alpha}\big\}_{\alpha \in \mathbb{Z}_{\ge 0}^d} \subset \Phi$ and $\left\{u_{\alpha}\right\}_{\alpha \in \mathbb{Z}_{\ge 0}^d} \subset \Phi'$ such that
        \begin{align*}
            C_f^\times u_\alpha &= \lambda^\alpha u_\alpha,\\
            \langle u_\alpha \mid w_\beta \rangle &= \delta_{\alpha, \beta}
        \end{align*}
        for all $\alpha, \beta \in \mathbb{Z}_{\ge 0}^d$, where $\delta_{\alpha, \beta}$ is the Kronecker delta, and 
        \begin{align*}
            C_f^\times u &= \sum_{\alpha \in \mathbb{Z}_{\ge 0}^d} \lambda^\alpha \langle w_{\alpha} \mid u \rangle u_{\alpha}.
        \end{align*}
        for $u \in \Phi'$ and $w \in \Phi$.
\end{theorem}
It is worth noting that the eigenvectors of the extended Koopman operator $C_f^\times$ can be interpreted as eigenfunctions of the Koopman operator in a weak sense (see Proposition \ref{prop: eigenfunction in weak sense}), but they are not eigenfunctions in the usual sense.
An important consequence derived from the usage of Gelfand triple is that what is estimated using a family of observables is not the Koopman operator itself, but the extended Koopman operator defined on the space of functional of the intrinsic observables.
Therefore, the estimated eigenvectors with the observables are strictly those of the extended Koopman operator and, unless the Koopman operator preserves the space of the observables, they are generally not eigenfunctions of the Koopman operator.
According to the numerical experiment in Section \ref{sec:NumericalSimulation}, the eigenvectors of the extended Koopman operator capture some important characteristics of the dynamical system.
However, the extended Koopman operator and its eigenvectors remain abstract and are not fully understood.
Their mathematical properties and their relation with the behavior of the dynamical system are crucial research topics for the future.

\section*{Acknowledgement}
I I, M I, Y K acknowledge support from the Japan Science and Technology Agency (JST) under CREST Grant JPMJCR1913. Y K acknowledges support from the Japan Society of Promotion of Science (JSPS) under KAKENHI Grant Number JP22H05106. I I acknowledges support from the JST under ACT-X Grant JPMJAX2004, the JST under CREST Grant JPMJCR24Q6, JPMJCR24Q, and the JSPS under KAKENHI Grant Number JP24K06771, JP24K16950, and JP24K21316. M I aknowledges support from the JSPS KAKENHI Grant Number 25H01453.

\section{Intrinsic observables and jets}\label{sec: theory}
Here, we describe the core notion of this paper, the space of the intrinsic observables. 
We introduce the notion of jet and the space of the intrinsic observables, and show their basic properties.
Let ${\Omega} \subset \mathbb{R}^d$ be an open subset.

\subsection{\revised{Jets and derivatives of Dirac deltas}}

For $\alpha \in \mathbb{Z}_{\ge 0}^d$ and $p \in \Omega$, let $\delta_p^{(\alpha)}$ be the $\alpha$-th derivative of the Dirac delta at $p$, that is defined by
    \[\delta_p^{(\alpha)}(h) := \partial_x^\alpha h(p)\]
for $h \in \mathfrak{E}(\Omega)$.
\color{blue}
\begin{definition}\label{def: alternative definition of Delta}
For any $n \ge 0$, we define
\begin{align*}
    \mathfrak{D}_{p,n} &= \sum_{|\alpha| \le n} \mathbb{C} \delta_p^{(\alpha)}.
\end{align*}
We also define $\mathfrak{D}_{p,-1} := \{0\}$ and $\mathfrak{D}_{p} := \bigcup_{n \ge 0}\mathfrak{D}_{p,n}$.
\end{definition}
\color{black}
According to the following lemma, we see that $\mathop{\rm dim}\mathfrak{D}_{p,n}$ coincides with $\binom{n+d}{d}$.
\begin{lemma}\label{lem: linear independence of Delta p}
    The set
    \[\left\{ \delta_{p}^{(\alpha)} : p \in \Omega, ~\alpha \in \mathbb{Z}_{\ge 0} \right\}\]
    is linearly independent.
\end{lemma}
\begin{proof}
    Let $p_1,p_2,\dots, p_r \in {\Omega}$ be distinct points.
    For each $i=1,\dots,r$, let
    \[D_i := \sum_\alpha c_{i,\alpha}\delta_{p_i}^{(\alpha)}.\]
    Here, we assume that $c_{i, \alpha}=0$ for all but finitely many $i$'s and $\alpha$'s.
    It suffices to show that  $c_{i, \alpha} = 0$ for $i=1,\dots,r$ and $\alpha \in \mathbb{Z}_{\ge 0}^d$ if $\sum_{i=1}^r D_i=0$. 
    Let $\phi_i \in \mathfrak{E}(\Omega)$ such that $\phi(p_i)=1$ and $p_j \notin {\rm supp}(\phi_i)$ for $j \neq i$.
    Then, for any $\alpha$, 
    \[\left(\sum_{k=1}^r D_k\right)\left((x-p_i)^\alpha \phi_i \right) = D_i\left((x-p_i)^\alpha \phi_i\right) =\alpha! c_{i,\alpha} = 0.\]
    Thus, we have $c_{i, \alpha} = 0$.
\end{proof}

For a map $f=(f_1,\dots, f_d) : {\Omega} \to {\Omega}$ of class $C^\infty$, let $f^*: \mathfrak{E}({\Omega}) \to \mathfrak{E}({\Omega})$ be the {\em pull-back} (Koopman operator) on $\mathfrak{E}({\Omega})$, that allocates $h \circ f \in \mathfrak{E}({\Omega})$ to $h \in \mathfrak{E}({\Omega})$, and let $f_*:=(f^*)': \mathfrak{E}({\Omega})' \to \mathfrak{E}({\Omega})'$ be the {\em push-forward}, that is the dual map of $f^*$.
Let $\mathbb{C}[X_1,\dots, X_d]_n$ be the space of homogeneous polynomials of degree equal to $n$:
\[\mathbb{C}[X_1,\dots, X_d]_n := \left\{\sum_{|\alpha| = n} a_\alpha X_1^{\alpha_1}\cdots X_d^{\alpha_d} : a_\alpha \in \mathbb{C}\text{ for }\alpha\text{ with }|\alpha| = n\right\}.\]
Then, by Lemma \ref{lem: linear independence of Delta p}, we have the linear isomorphism
\[\rho_n: \mathfrak{D}_{p,n}/\mathfrak{D}_{p,n-1} \longrightarrow \mathbb{C}[X_1,\dots, X_d]_n\]
defined by $\rho\left(\delta_p^{(\alpha)} + \mathfrak{D}_{p,n-1}\right) := X_1^{\alpha_1}\cdots X_d^{\alpha_d}$ for $\alpha$ with $|\alpha| = n$.
Then, we have the following statements:
\begin{proposition}\label{prop: explicit description of representation matrix of f_*}
    \begin{enumerate}[{\rm (1)}]
        \item  \label{preserving property, discrete}  For each integer $n \ge -1$, we have $f_*(\mathfrak{D}_{p,n}) \subset \mathfrak{D}_{f(p),n}$.
        \item   \label{graded map, discrete} Let ${\rm gr}^n_{f_*}: \mathfrak{D}_{p,n}/\mathfrak{D}_{p,n-1} \to \mathfrak{D}_{f(p),n}/\mathfrak{D}_{f(p),n-1}$ be the linear map induced $f_*$ via \eqref{preserving property, discrete}.
        Then, we have
    \begin{align}
        {\rm gr}^n_{f_*} = \rho_n^{-1} \circ S_n({\rm d}f_p) \circ \rho_n,
    \end{align}
    where $S_n({\rm d}f_p) \in \mathbb{C}^{d \times d}$ is the $n$-th symmetric product of the Jacobian matrix ${\rm d}f_p$, that is the linear map on $\mathbb{C}[X_1,\dots, X_d]_n$ defined by
    \[ S_n({\rm d}f_p)(Q(X_1,\dots, X_d)) := Q\left((X_1,\dots,X_d)^\top \cdot {\rm d}f_p\right)\]
    for $Q(X_1,\dots,X_d) \in \mathbb{C}[X_1,\dots, X_d]_n$.
    \item \label{representation matrix, discrete}Let $\mathbf{C}_{f,n}^\star \in \mathbb{C}^{\binom{n+d}{d} \times \binom{n+d}{d}}$ be the representation matrix of $f_*|_{\mathfrak{D}_{p,n}}: \mathfrak{D}_{p,n} \to \mathfrak{D}_{f(p),n}$ with respect to the basis $\{\delta_p^{(\alpha)} : |\alpha| \le n\}$ and $\{\delta_{f(p)}^{(\alpha)} : |\alpha| \le n\}$.
    Then, $\mathbf{C}_{f,n}^\star$ is in the form of 
    \begin{equation}\label{representation matrix}
        \mathbf{C}_{f,n}^\star=
        \begin{pmatrix}
            1 & * & * & *\\
            0 & \mathbf{J}_{p,1} & * & *\\
            \vdots & \ddots& \ddots & \vdots\\
            0 & \cdots & 0& \mathbf{J}_{p,n} 
        \end{pmatrix},
    \end{equation}
    where $\mathbf{J}_{p,i}$ is the representation matrix of $S_i({\rm d}f_p)$ with respect to the basis $\{X_1^{\alpha_1}\cdots X_d^{\alpha_d} : |\alpha| = i\} \subset \mathbb{C}[X_1,\dots,X_d]$ for $i = 1,\dots,n$.
    \item \label{eigenvalues of pushforward} Assume that $p = f(p)$.  Let $\lambda_1,\dots, \lambda_d$ be the eigenvalues (with multiplicity) of the Jacobian matrix ${\rm d}f_p$ and let $\lambda := (\lambda_1,\dots, \lambda_d)$.
    Then, the set of eigenvalues of $f_*|_{\mathfrak{D}_{p,n}}: \mathfrak{D}_{p,n} \to \mathfrak{D}_{p,n}$ is
    \[\left\{ \lambda^\alpha : |\alpha| \le n\right\}.\]
    \end{enumerate}
\end{proposition}
\begin{proof}
       The statements \eqref{preserving property, discrete} and \eqref{graded map, discrete} are direct consequences of \cite[Lemma 2.2]{ISHIKAWA2023109048}.
       The formula \eqref{representation matrix} follows from \eqref{graded map, discrete}.
       The statement \eqref{eigenvalues of pushforward} immediately follows from \eqref{representation matrix}.
\end{proof}
\revised{
Intuitively, the jet space $\mathfrak D_{p,n}$ collects all derivatives up to order $n$ at $p$.
By the chain rule, the push-forward $f_*=(f^*)'$ sends an $n$-jet at $p$
to an $n$-jet at $f(p)$.
On the homogeneous degree-$n$ layer $\mathfrak D_{p,n}/\mathfrak D_{p,n-1}\cong\mathbb C[X_1,\dots,X_d]_n$, this action is the symmetric power $S_n(df_p)$ of the Jacobian, so the full matrix in a jet basis is block-upper-triangular.
If $p$ is a fixed point, the eigenvalues of $f_*|_{\mathfrak D_{p,n}}$ are the monomials $\lambda^\alpha=\lambda_1^{\alpha_1}\cdots\lambda_d^{\alpha_d}$ with $|\alpha|\le n$, where $\lambda_1,\dots,\lambda_d$ are the eigenvalues of $df_p$.
This fact also compatible with the fact that the products of eigenvalues of Koopman operators are themselves eigenvalues of the Koopman operators (if eigenvalues exist).
}

As a remark, we also provide a canonical (not depending on the coordinate system) definition of $\mathfrak{D}_{p,n}$, using the notion of jets as follows.
Let $\mathfrak{P}_n \subset \mathfrak{E}(\Omega)$ be the set of polynomial functions of degree less than or equal to $n$:
\[ \mathfrak{P}_n := \left\{x \mapsto \sum_{|\alpha| \le n} a_\alpha x^\alpha : a_\alpha \in \mathbb{C}\text{ for }\alpha\text{ with }|\alpha| \le n\right\}.\]
For $p \in \Omega$ and non-negative integer $n \ge 0$, we define 
\begin{align}
    \tau_{p,n}: \mathfrak{E}(\Omega) \to \mathfrak{P}_n;~h \mapsto \sum_{|\alpha| \le n} \frac{\partial_x^\alpha h(p)}{\alpha !} (x-p)^\alpha.
\end{align}
\begin{definition}\label{def: jet}
    We define the set $\mathfrak{J}_{p,n}$ of {\em $n$-jets} from $\Omega$ to $\mathbb{C}$ at $p$ by the quotient space
    \begin{align}
        \mathfrak{J}_{p,n} := \mathfrak{E}(\Omega)/{\rm Ker}(\tau_{p,n}).
    \end{align}
    We denote by $\mathfrak{j}_{p,n}$ the natural map $\mathfrak{j}_{p,n}: \mathfrak{E}(\Omega) \to \mathfrak{J}_{p,n}$.
\end{definition}
By definition, $\tau_{p,n}$ induces the isomorphism from $\mathfrak{J}_{p,n}$ to $\mathfrak{P}_n$.
We note that the notion of the jet is a mathematically canonical object on a smooth manifold (not depending on the choice of coordinate systems).
For details of the theory of jets, see, for example, \cite[Section 4]{KPS_NatOpMan93}.

Then, we have the following proposition:
\color{blue}
\begin{proposition}[An alternative definition of $\mathfrak{D}_{p,n}$]
For each $n \ge 0$, we have $\mathfrak{D}_{p,n} = \mathfrak{j}_{p,n}'\left(\mathfrak{J}_{p,n}'\right)$.
\end{proposition}
\color{black}
\begin{proof}
    Since $\tau_{p,n}'(\mathfrak{P}_n) = \mathfrak{j}_{p,n}'(\mathfrak{J}_{p,n}')$, we prove $\tau_{p,n}'(\mathfrak{P}_n) = \mathfrak{D}_{p,n}$.
    For $\beta \in \mathbb{Z}_{\ge 0}^d$ with $|\beta| \le n$, let $\ell_\beta \in \mathfrak{P}_n'$ be the linear function of $\mathfrak{P}_n$ defined by $\ell_\beta(\sum_\alpha a_\alpha x^\alpha) := \beta!a_\beta$.
    Then, for any $h \in \mathfrak{E}(\Omega)$, we have 
    \[\langle \tau_{p,n}'(\ell_\beta) \mid h\rangle = \partial_x^\beta h|_{x=p} = \langle \delta_p^{(\beta)} \mid h\rangle.\]
    Thus, we conclude that the image of $\tau_{p,n}'$ includes $\mathfrak{D}_{p,n}$.
    Since $\tau_{p,n}$ is surjective, the dual map $\tau_{p,n}'$ is injective, thus the dimension of the image of $\tau_{p,n}'$ coincides with that of $\mathfrak{P}_n$.
    On the other hand, the dimension of $\mathfrak{D}_{p,n}$ is the same as $\mathop{\rm dim}\mathfrak{P}_n$ by Lemma \ref{lem: linear independence of Delta p}.
    Therefore, the image of $\tau_{p,n}'$ coincides with $\mathfrak{D}_{p,n}$.
\end{proof}

\subsection{Construction of the space of intrinsic observables in reproducing kernel Hilbert spaces}\label{subsubsec: canonical invariant subsapces}
Let $H \subset \mathfrak{E}(\Omega)$ be a Hilbert space with inner product $\langle\cdot,\cdot\rangle_H$.
We denote the inclusion map by 
\begin{align}
    \iota: H \hookrightarrow \mathfrak{E}({\Omega}). \label{inclusion}
\end{align}
We always assume the following assumption:
\begin{assumption}\label{asm: continuity of inclusion}
    The inclusion map $\iota$ is continuous.
\end{assumption}

What is important here is that such a Hilbert space is captured in the framework reproducing kernel Hilbert space (RKHS). 
First, we provide the definition of RKHS:
\begin{definition}
    Let $X$ be a set and let $H$ be a Hilbert space composed of maps from $X$ to $\mathbb{C}$ with inner product $\langle\cdot, \cdot\rangle_H$.
    We say that $H$ is a {\it reproducing kernel Hilbert space  (RKHS)} on $X$ if the point evaluation map $h \mapsto h(x)$ from $H$ to $\mathbb{C}$ is a continuous linear map for any $x \in X$.
\end{definition}
Let $H$ be an RKHS on a set $X$.
Then, for each $x \in X$, by the Riesz representation theorem, there uniquely exists $k_x \in H$  satisfying
\[\langle h, k_x\rangle_H =h(x).\]
We define $k(x,y) := \langle k_x, k_y\rangle$ for $x, y$ in $X$.
Then, $k$ is a positive definite kernel on $X$.
Here we say a map $k: X \times X \to \mathbb{C}$ is positive definite kernel on $X$ if the matrix $\left(k(x_i,x_j)\right)_{i,j=1,\dots, r}$ is a Hermitian positive semi-definite matrix for any finite elements $x_1,\dots, x_r \in X$.
Conversely, for a positive definite kernel $k$ on $X$, there uniquely exists an RKHS on $X$ (see \cite[Theorem 2.2]{6fc4e20b-e6b8-3651-8988-bc9975065f31}, for example).
Thus, we have a one-to-one correspondence between the RKHSs on $X$ and the positive definite kernels on $X$.
In the remainder of this paper, we always consider $X = \Omega$.

Under Assumption \ref{asm: continuity of inclusion}, $H$ is an RKHS on $\Omega$ since the point evaluation map at $x \in \Omega$ coincides with $\delta_x\circ \iota$.
We note that the associated positive definite kernel $k$ of $H$ is a function of class $C^\infty$.
Conversely, as in the following proposition, the RKHS on $\Omega$ of a positive definite kernel of class $C^\infty$ is continuously included in $\mathfrak{E}(\Omega)$.
Moreover, if the positive definite kernel $k$ satisfies $k_x \in \mathfrak{H}(\Omega)$ for all $x \in \Omega$, $H$ is continuously included in the space of real analytic functions.
\begin{proposition}\label{prop: continuity of inclusion}
    The correspondence $H \mapsto k$ induces a one-to-one correspondence between the Hilbert spaces $H\subset \mathfrak{E}(\Omega)$ (resp. $\mathfrak{H}(\Omega)$) with continuous inclusions and the continuous positive definite kernels $k$ such that $k_x \in \mathfrak{E}(\Omega)$ (resp. $\mathfrak{H}(\Omega)$) for all $x \in \Omega$.
\end{proposition}
\begin{proof}
    It follows from \cite[Theorems 2.6 and 2.7]{6fc4e20b-e6b8-3651-8988-bc9975065f31}.
\end{proof}

We define
\begin{align}
    \mathfrak{r}: H' \longrightarrow H \label{riesz representation}
\end{align}
be the anti-linear isomorphism such that $ \langle \ell \mid \cdot \rangle = \langle \cdot, \mathfrak{r}(\ell)\rangle_H$ for $\ell \in H'$ by the Riesz representation theorem.
Now, we define the space of the {\rm intrinsic observables} as follows:
\begin{definition}[Intrinsic observables]
For $p \in \Omega$ and $n \ge 0$, we define a finite dimensional subspace of $H$ by
\begin{align*}
    V_{p,n} &:= (\mathfrak{r} \circ \iota')(\mathfrak{D}_{p,n}). \\
\end{align*}
We denote the orthogonal projection to $V_{p,n}$ by
\begin{align}
    \pi_n: H \to V_{p,n} \label{orthogonal projection}
\end{align}
and define
\begin{align*}
    V_p &:= \bigcup_{n\ge0} V_{p,n}
\end{align*}
\end{definition}
\revised{
The space $V_p$ controls the Taylor coefficients of functions in $H$ at that point since the space $\mathfrak{D}_p$ is spanned by the functionals of all partial derivatives at $p$ (see also Corollary \ref{cor: alternative definition of Vpn}).
}
We remark that we can explicitly describe the elements $(\mathfrak{r} \circ \iota')(\delta_p^{(\alpha)}) \in H$ of $V_{p}$ using derivatives of $k$ as follows:
\begin{proposition}\label{prop: explicit formula of iota D delta p}
    For $\alpha \in \mathbb{Z}_{\ge 0}^d$ and $p \in \Omega$, we have
    \begin{align}
        (\mathfrak{r} \circ \iota')(\delta_p^{(\alpha)}) = \partial_x^\alpha k(x,\cdot)|_{x=p}.
    \end{align}
\end{proposition}
\begin{proof}
    Let $y \in \Omega$ be an arbitrary point.
    Then, it follows from the following computation:
    \begin{align*}
        (\mathfrak{r} \circ \iota')(\delta_p^{(\alpha)})(y) &= \langle \mathfrak{r}\iota'\delta_p^{(\alpha)}, k_y\rangle_H
        = \overline{\langle k_y, \mathfrak{r}\iota'\delta_p^{(\alpha)} \rangle_H} = \overline{\partial_x^\alpha k_y(x)|_{x=p}}\\
        &= \partial_x^{\alpha} k(x,y)|_{x=p}.
    \end{align*}
\end{proof}

We introduced the space of intrinsic observables in $H$ via the space of jets $\mathfrak{J}_{p,n}$ and the continuous inclusion $\iota: H \subset \mathfrak{E}(\Omega)$.
This definition is mathematically canonical, but it is a little abstract and not explicit as well.
As an immediate conclusion of Proposition \ref{prop: explicit formula of iota D delta p}, the space $V_{p,n}$ has an alternative explicit expression:
\begin{corollary}\label{cor: alternative definition of Vpn}
For $p \in \Omega$ and $n \ge 0$, we have    
\begin{align}
    V_{p,n} &=  \sum_{|\alpha| \le n} \mathbb{C} \cdot \partial_x^\alpha k(x,\cdot)|_{x=p}
\end{align}
\end{corollary}


Now, let us introduce the Koopman operator on the RKHS and consider the relation between the intrinsic observables and the Koopman operator.
For a map $f=(f_1,\dots, f_d) : {\Omega} \to {\Omega}$ of class $C^\infty$, let $C_f$ be the Koopman operator, that is a linear operator defined by $C_fh := h \circ f$ with domain $D(C_f) = \{h \in H : h\circ f \in H\}$.
When $C_f$ is densely defined, we can define the adjoint operator $C_f^*: H \to H$, that is known as the Perron--Frobenius operator.
For Koopman operators with dense domain, the Perron--Frobenius operator on $H$ is compatible with the push-forward:
\begin{proposition}\label{prop: domain of PF operators}
    Assume that $C_f$ is densely defined.
    For any $\ell \in \mathfrak{E}(\Omega)'$, $(\mathfrak{r} \circ \iota')(\ell) \in D(C_f^*)$ and $C_f^* \circ (\mathfrak{r} \circ \iota')=  (\mathfrak{r} \circ \iota') \circ f_*$ holds.
\end{proposition}
\begin{proof}
    Since $f^*$ is a continuous linear map on $\mathfrak{E}(\Omega)$, we see that $\langle C_fh, \mathfrak{r}\iota'(\ell)\rangle_H = (\ell \circ f^*\circ \iota )(h)$ is also continuous on $h \in D(C_f)$, proving the first statement.
    As for the second statement,  since $\mathfrak{r}(\ell \circ f^*\circ \iota) = \mathfrak{r}\iota'f_*(\ell)$, we have $\langle h, \mathfrak{r}\iota'f_*(\ell) \rangle_H = \langle h, C_f^*\mathfrak{r}\iota'(\ell)\rangle_H$.
    Thus, we have the second statement.
\end{proof}
Then, we have the following theorem:
\begin{theorem}\label{thm: basic theorem for PF operator, discrete}
Let $p$ be a point $p$ such that $f(p)=p$.
Assume that $C_f$ is densely defined.
Then, we have the following statements:
\begin{enumerate}[{\rm (1)}]
    \item  $C_f^*(V_{p,n}) \subset V_{p,n}$ for all $n \ge 0$. \label{preserving property, PF operator}
    \item  Assume that $\iota'|_{\mathfrak{D}_p}: \mathfrak{D}_p \to H'$ is injective. \label{eigenvalues of PF operator}
    Then, the set of the eigenvalues of $C_f^*$ is 
    \[\left\{ \lambda^\alpha : |\alpha| \le n\right\},\]
    where $\lambda := (\lambda_1,\dots, \lambda_d)$ and $\lambda_1,\dots, \lambda_d$ are the eigenvalues (with multiplicity) of the Jacobian matrix ${\rm d}f_p$.
\end{enumerate}
\end{theorem}
\begin{proof}
    It follows from Proposition \ref{prop: explicit description of representation matrix of f_*} and Proposition \ref{prop: domain of PF operators}.
\end{proof}
The assumption of the injectivity of $\iota'|_{\mathfrak{D}_p}: \mathfrak{D}_p \to H$ in \eqref{eigenvalues of PF operator} in Theorem \ref{thm: basic theorem for PF operator, discrete} holds for typical RKHSs used in applications, for example, the exponential kernel $k(x,y)=e^{(x-b)^\top (y-b)/\sigma^2}$ and the Gaussian kernel $e^{-|x-y|^2/2\sigma^2}$.
We provide more theoretical results for these two kernels in Section \ref{sec: special kernels}.

We end this section with several special properties of the space of intrinsic observables when the positive definite kernel $k$ has a real analyticity.
We consider the following assumption:
\begin{assumption}\label{asm: basic}
    The open set $\Omega$ is connected and $k$ is continuous on $\Omega \times \Omega$, satisfying $k_x \in \mathfrak{H}(\Omega)$ for any $x \in \Omega$.
\end{assumption}
Then, we have the following proposition:
\begin{proposition}\label{prop: density}
    Under Assumption \ref{asm: basic}, $V_p$ is dense in $H$.
\end{proposition}
\begin{proof}
    Let $h \in (V_p)^\perp$.
    Then, $\langle h, (\mathfrak{r} \circ \iota')\delta_p^{(\alpha)} \rangle_H = \partial_x^\alpha h(p)=0$ for any $\alpha \in \mathbb{Z}_{\ge 0}^d$.
    Thus, $h$ is zero on a neighborhood of $p$, and thus $h=0$ since $h$ is real analytic and ${\Omega}$ is connected.
    Therefore, we have $(V_p)^\perp=\{0\}$ and $V_p$ is dense in $H$.
\end{proof}
\revised{
Intuitively, if $h\in H$ is orthogonal to $V_p$, meaning that every derivative of $h$ at $p$ is zero by the reproducing property of $H$, all its Taylor coefficients at $p$ vanish.
In our setting functions in $H$ are real analytic, so vanishing of the entire Taylor series at one point forces $h$ to vanish on a neighborhood of $p$, hence everywhere on the connected $\Omega$. 
Thus the only function orthogonal to $V_p$ is $0$, which means $V_p$ is dense in $H$.
}

For $x \in {\Omega}$, we introduce the minimal approximation error for between $k_x$ and the elements of $V_{p,n}$ as follows:
\begin{align}
    \mathcal{E}_{p,n}(x) := \min_{v\in V_{p,n}}\left\|k_x - v \right\|_H = \|k_x-\pi_nk_x\|_H,\label{error}
\end{align}
where $\pi_n: H \to V_{p,n}$ is the orthogonal projection of \eqref{orthogonal projection}.
Proposition \ref{prop: density} implies that using elements of $V_{p,n}$ for sufficiently large $n$, we can approximate the family $\{k_x\}_{x \in K}$ with arbitrary precision for any compact set $K$:
\begin{proposition}\label{prop: uniform convergence of the minimal error}
    Suppose Assumption \ref{asm: basic} holds.
    Let $K \subset {\Omega}$ be a compact set.
    Then, $\sup_{x \in K}\mathcal{E}_{p,n}(x) \to 0$ as $n \to \infty$.
\end{proposition}
\begin{proof}
    By definition, for each $x \in {\Omega}$, we have $\mathcal{E}_{p,n}(x) \ge \mathcal{E}_{p,n+1}(x)$ for all $n \ge 0$ and $\mathcal{E}_{p,n}(x) \to 0$ as $n \to \infty$.
    Since $\mathcal{E}_{p,n}(x)$ is continuous, by Dini's theorem,  $\mathcal{E}_{p,n}(x)$ uniformly converges to $0$ on $K$, namely, $\lim_{n\to\infty}\sup_{x \in K}\mathcal{E}_{p,n}(x) \to 0$.
\end{proof}
Then, we have the following lemma, which will play a main role in the theoretical results for the reconstruction of dynamical systems in Section \ref{sec:NumericalSimulation}:
\begin{lemma} \label{lem: lemma for system identification}
    Assume $C_f$ is densely defined.
    Let $B \subset D(C_f)$ be a subset such that $C_f(B)$ is a bounded subset in $H$.
    Then, for any compact set $K \subset \Omega$, we have
    \[ \sup_{\substack{x \in K \\ h \in B}}|\langle k_{f(x)} - C_f^*\pi_nk_x, h\rangle| \le \sup_{x \in K}\mathcal{E}_{p,n}(x)\cdot \sup_{h \in B}\|C_fh\|_H.\]
    Moreover, if we further assume Assumption \ref{asm: basic}, we have
    \[\sup_{\substack{x \in K \\ h \in B}}\big|\langle k_{f(x)} - C_f^*\pi_nk_x, h\rangle_H \big| \to 0\]
    as $n \to \infty$.
\end{lemma}
\begin{proof}
The first statement follows from $|\langle k_{f(x)} - C_f^*\pi_nk_x, h\rangle_H| \le \mathcal{E}_{p,n}(x)\|C_fh\|_H$ using the Cauchy--Schwarz inequality and the formula $C_f^*k_x = k_{f(x)}$.
The second statement follows from the first statement with Proposition \ref{prop: uniform convergence of the minimal error}.
\end{proof}

\section{Estimation of Perron--Frobenius operators}\label{sec: datadriven estimation of PF operator}
Here, we provide the key error bound for the estimation of the Perron--Frobenius operator using the intrinsic observables.
In this section, we always assume $C_f$ is densely defined and consider the following assumptions:
\begin{assumption}[Existence of a fixed point]\label{asm: existence of a fixed point}
    $p$ is a fixed point of $f$, namely, $f(p) = p$.
\end{assumption}
\begin{assumption}[Domain condition for $C_f$]\label{asm: domain of C_f}
    $V_p \subset D(C_f)$.
\end{assumption}
We note that Assumptions \ref{asm: basic} and \ref{asm: domain of C_f} imply that $C_f$ is densely defined by Proposition \ref{prop: density}.
As for the dynamical system $f$, although we are only assuming that $f$ is $C^\infty$, if we further impose Assumption \ref{asm: domain of C_f} on $f$, then $f$ automatically has a stronger analytic property than mere smoothness.
We also note that we a have more specific sufficient condition of Assumption \ref{asm: domain of C_f} for the exponential kernel and the Gaussian kernel (see Section \ref{subsec: validity of domain assumption of koopman}).

\subsection{Error analysis} \label{subsec: error bound analysis}
Let $r_n := \mathop{\rm dim}V_{p, n}$.
We fix a basis $\mathcal{B}_p = \{ v_n \}_{n \ge 0}$ of $V_p$ such that $\mathcal{B}_{p,n} = \{v_i\}_{i=1, \dots, r_n}$ constitutes a basis of $V_{p,n}$.
Then, we define 
\begin{align}
     {\bf v}_{n}(x) &:= \left(\overline{v_i(x)}\right)_{i=1}^{r_n}, \label{bold v}\\
     \mathbf{G}_{n} &:= \left(\langle v_i, v_j \rangle_H\right)_{i,j = 1,\dots, r_n}. \label{gram matrix}
\end{align}
For $X = (x_1, \dots, x_N) \in {\Omega}^N$, we define 
\begin{align}
    \mathbf{V}_{n}^X :=& \left({\bf v}_{n}(x_1), \dots, {\bf v}_{n}(x_N)\right). \label{bold V}
\end{align}

We have an explicit description of the coefficients of the basis $\mathcal{B}_{p,n}$ for the minimizer that attains the minimization problem \eqref{error}:
\begin{proposition}\label{prop: coefficient of minimizer}
Let $x \in {\Omega}$.
Let $\pi_nk_x = \sum_{i=1}^{r_n} c_i^*(x) v_i$.
Then, $c_i^*(x)$ satisfies
\begin{align}
\mathbf{G}_{n}\cdot \left(c_i^*(x)\right)_{i=1}^{r_n} = {\bf v}_{n}(x).\label{coefficient of minimizer}
\end{align}
\end{proposition}
\begin{proof}
Let $k_x = v' + \sum_{i=1}^{r_n} c_i^*(x) v_i$, where $v' \in V_{p,n}^\perp$.
Then, considering $\langle v_j, k_x\rangle_H$ for each $j$, we have the identity \eqref{coefficient of minimizer}.
\end{proof}
\begin{example}
Assume that $\iota'|_{\mathfrak{D}_p}$ is injective.
We fix a numbering of $\mathbb{Z}_{\ge 0}^d$ such that  $\mathbb{Z}_{\ge 0}^d = \{\alpha^{(i)}\}_{i=1}^\infty$ and $|\alpha^{(i)}| \le |\alpha^{(j)}|$ if $i \le j$.
Then, if we take $\partial_x^{\alpha^{(i)}} k(x,\cdot)|_{x=p}$ as $v_i$, we have
    \[{\bf v}_{n}(x) = (\partial_x^{\alpha^{(i)}} k(x,p))_{i=1}^{r_n},
    \quad \mathbf{G}_{n} = \left(\partial_x^{\alpha^{(i)}} \partial_y^{\alpha^{(j)}} k(p,p) \right)_{i,j=1,\dots,r_n}.\]
\end{example}

Then, we have the following theorem:
\begin{theorem}\label{thm: error analysis for PF operator for algorithm, disc}
    Let $p \in \Omega$.
    Let $m, n \ge 0$ be integers with $m \le n$.
    Let $\mathbf{C}_{f,m}^\star$ be the representation matrix of $C_f^*|_{V_{p,m}}$ with respect to $\mathcal{B}_{p,m}$.
    Let $X_N := (x_1,\dots, x_N) \in {\Omega}^N$ and $Y_N := (y_1, \dots, y_N) \in {\Omega}^N$ such that $f(x_i) = y_i$ for $i=1,\dots,N$.
    We assume Assumption \ref{asm: basic}, $r_n \le N$, ${\rm span}(\{\pi_nk_{x_1},\dots,\pi_nk_{x_N}\}) = V_{p,n}$, and that $p$ satisfies Assumptions \ref{asm: existence of a fixed point} and \ref{asm: domain of C_f}.
    We define the square matrices $\widehat{\mathbf{C}}_{m,n, N}$ and $\mathbf{E}_{m,n}^{X_N}$ of size $r_m$ by the leading principal minor matrix of order $r_m$ such that
    \begin{align}
    \mathbf{G}_m^{-1}\mathbf{V}_m^{Y_N} (\mathbf{V}_n^{X_N})^\dagger\mathbf{G}_n = 
    \begin{pmatrix}
        \widehat{\mathbf{C}}_{m, n, N} & *
    \end{pmatrix}, \label{CnmX}\\
    \mathbf{G}_n\left(N^{-1}\mathbf{V}_n^{X_N}(\mathbf{V}_n^{X_N})^*\right)^{-1} \mathbf{G}_n = 
    \begin{pmatrix}
        \mathbf{E}_{m,n}^{X_N} & *\\
        * &*
    \end{pmatrix}.
    \end{align}
    Then, we have
    \begin{align}
        &\left\| \mathbf{C}_{f, m}^\star - \widehat{\mathbf{C}}_{m, n, N}\right\|
        \le 
        \left\|C_f|_{V_{p,m}}\right\|_{\rm op}\big\|\mathbf{G}_m^{-1/2}\big\|_{\rm op}\big\|\mathbf{E}_{m, n}^{X_N}\big\|_{\rm op}^{1/2}\sqrt{\frac{1}{N}\sum_{i=1}^N \mathcal{E}_{p,n}(x_i)^2}. \label{explicit error}
    \end{align}
    
\end{theorem}
\begin{proof}
Let $\mathbf{P}_m :=\begin{pmatrix} \mathbf{G}_m^{-1} &\mathbf{O}_{r_n-r_m} \end{pmatrix} \mathbf{G}_n$ be the representation matrix of $\pi_m|_{V_{p,n}}\colon V_{p,n} \to V_{p,m}$ with respect to the basis $\mathcal{B}_{p,n}$ and $\mathcal{B}_{p,m}$.
We denote by $\mathbf{P}_m^\diamond := \begin{pmatrix} \mathbf{I}_{r_m} &\mathbf{O}_{r_n-r_m} \end{pmatrix}^\top$ the representation matrix of the adjoint operator of $\pi_m|_{V_{p,n}}$.
Then, we have
\begin{align*}
    \widehat{\mathbf{C}}_{m, n, N}
    &=
    \mathbf{P}_m
    \mathbf{G}_n^{-1}\mathbf{V}_n^{Y_N} (\mathbf{V}_n^{X_N})^\dagger\mathbf{G}_n
    \mathbf{P}_m^\diamond,\\
    \mathbf{C}_{f, m}^\star
    &=
    \mathbf{P}_m
    \mathbf{C}_{f,n}^\star
    \mathbf{P}_m^\diamond.
\end{align*}
By the assumption that $\pi_nk_{x_1},\dots,\pi_nk_{x_N}$ span $V_{p,n}$, we have $\mathbf{V}_n^{X_N} (\mathbf{V}_n^{X_N})^\dagger = \mathbf{I}_{r_n}$.
Therefore, we see that
\begin{align}
    \left\| \mathbf{C}_{f, m}^\star - \widehat{\mathbf{C}}_{m, n, N}\right\|
    \le 
    \|\mathbf{G}_m^{-1/2}\|_{\rm op} 
    \big\|\mathbf{G}_m^{1/2} \mathbf{P}_m \mathbf{C}_{f,n}^\star\mathbf{G}_n^{-1}\mathbf{V}_n^{X_N} -
    \mathbf{G}_m^{1/2}\mathbf{P}_m\mathbf{G}_n^{-1}\mathbf{V}_n^{Y_N}
    \big\| 
    \cdot 
    \left\|(\mathbf{V}_n^{X_N})^\dagger \mathbf{G}_n
    \mathbf{P}_m^\diamond \right\|_{\rm op}. \label{estimation 1}
\end{align}
By \cite{4aa77599-3504-3cfd-9aba-80cca6f578f2}, we have
\[\left((\mathbf{V}_n^{X_N})^\dagger\right)^*(\mathbf{V}_n^{X_N})^\dagger = \left(\mathbf{V}_n^{X_N}(\mathbf{V}_n^{X_N})^*\right)^{-1}.\]
Thus, 
\begin{align}
\left\|(\mathbf{V}_n^{X_N})^\dagger \mathbf{G}_n \mathbf{P}_m^\diamond \right\|_{\rm op} = N^{1/2}\|\mathbf{E}_{m,n}^{X_N}\|_{\rm op}^{1/2}. \label{estimation 2}
\end{align}
     By Proposition \ref{prop: coefficient of minimizer}, we see that
    \begin{align*}
        (\pi_mC_f^* \pi_n k_{x_1},\dots,\pi_mC_f^* \pi_n k_{x_N}) &= (v_i)_{i=1}^{r_m} \mathbf{P}_m\mathbf{C}_{f,n}^\star \mathbf{G}_{n}^{-1}\mathbf{V}_{n}^{X_N},\\
        (\pi_m \pi_n k_{y_1},\dots,\pi_m \pi_n k_{y_N}) &= (v_i)_{i=1}^{r_m} \mathbf{P}_m\mathbf{G}_{n}^{-1}\mathbf{V}_{n}^{Y_N}.
    \end{align*}  
    Thus, we have
    \begin{align*}
       \|\mathbf{G}_m^{1/2} \mathbf{P}_m \mathbf{C}_{f,n}^\star\mathbf{G}_n^{-1}\mathbf{V}_n^{X_N} -
    \mathbf{G}_m^{1/2}\mathbf{P}_m\mathbf{G}_n^{-1}\mathbf{V}_n^{Y_N}\| ^2 = \sum_{i=1}^N \|\pi_m C_f^* \pi_n k_{x_i}  -  \pi_m  \pi_n k_{y_i}\|_H^2.
    \end{align*}
    Since $\pi_m \pi_n k_{y_i} = \pi_mk_{y_i} = \pi_mC_f^*k_{x_i}$, we have
    \begin{align}
   \|\pi_m C_f^* \pi_n k_{x_i}  -  \pi_m  \pi_n k_{y_i}\| 
    & \le  \left\|\pi_mC_f^*( \pi_n k_{x_i} - k_{x_i})\right\|_H \notag\\
    & \le  \left\|\pi_mC_f^*\right\|_{\rm op} \|\pi_n k_{x_i}  - k_{x_i}\|_H \notag\\
    & =  \left\|\pi_mC_f^*\right\|_{\rm op}\mathcal{E}_{p,n}(x_i). \label{estimation 3}
    \end{align}
    By Assumption \ref{asm: domain of C_f}, the linear map $C_f|_{V_{p,m}}: V_{p,m} \to H$ is automatically bounded, and its adjoint coincides with $\pi_mC_f^*$.
    Thus, we have $\big\|\pi_mC_f^*\big\|_{\rm op} =  \left\|C_f|_{V_{p,m}}\right\|_{\rm op} < \infty$.
    Therefore, combining the inequality \eqref{estimation 1} with \eqref{estimation 2} and \eqref{estimation 3}, we have \eqref{explicit error}.
\end{proof}
This theorem describes the essence of the truncating process in JetEDMD as mentioned in the following remark of Theorem \ref{thm:_intro_convergence} in Section \ref{sec: introduction}.
The presence of the constant $\|C_f|_{V_{p,m}}\|_{\rm op}$ arises from extracting the first $r_m$ rows of the matrix, while the the constant $\|\mathbf{E}_{m,n}^{X_N}\|_{\rm op}$ appears as a result of extracting the first $r_m$ columns.
Estimating the growth rate of those constants $\|C_f|_{V_{p,m}}\|_{\rm op}$ and  $\|\mathbf{E}_{m,n}^{X_N}\|_{\rm op}$ as $m$ and $n$ increase is extremely challenging, and it is generally expected that they will grow rapidly with $m$ and $n$.
This indicates the difficulty of analyzing EDMD, corresponding to the case of $m=n$, and generally suggests that the estimation with EDMD is likely to fail.
In contrast, if we fix $m$, then $\|C_f|_{V_{p,m}}\|_{\rm op}$ becomes just a constant and the behavior of $\|\mathbf{E}_{m,n}^{X_N}\|_{\rm op}$ along $n$ gets much tamer and more controllable.
As a result, we have rigorous convergence results in some specific RKHSs (see Section \ref{sec: special kernels}).
We also remark that an alternative method, Analytic EDMD, to approximate $C_f|_{V_{p,m}}^*$ is proposed in \cite{mauroy2024analyticextendeddynamicmode}.

\begin{remark}
It is easy to show that $C_f$ is bounded on $H$ if and only if $\{\|C_f|_{V_{p,n}}\|_{\rm op}\}_{n=1}^\infty$ is a bounded sequence since $V_p = \cup_{n} V_{p,n}$ is dense in $H$.
However, some theoretical results show the dynamical system becomes linear when the Koopman operator is bounded on specific RKHSs corresponding to a certain important kernels such as the exponential kernel \cite{CMS} and the Gaussian kernel \cite{Ikeda_Ishikawa_Sawano-2022} (we will use them in the numerical simulation).
Moreover, it is also shown that the dynamical system is never chaotic if the Koopman operator is bounded \cite{ISHIKAWA2023109048}.
Thus, in the context of nonlinear analysis, we cannot expect that the Koopman operator on an RKHS is bounded, namely, we should assume that $\|C_f|_{V_{p,n}}\|_{\rm op} \to \infty$ as $n \to \infty$, in many cases.
\end{remark}

We will use the following corollary for further investigation of \eqref{estimation 1} when the sample number $N$ goes to infinity:
\begin{corollary}\label{cor: error analysis for PF operator, infinite samples: disc.}
    We use the notation and assume the assumptions in Theorem \ref{thm: error analysis for PF operator for algorithm, disc}.
    Let $\mu$ be a Borel probability measure on $\Omega$.
    Let $\nu$ be an arbitrary $\sigma$-finite Borel measure absolutely continuous with respect to $\mu$ on $\Omega$.
    Suppose that $x_1,\dots, x_N$ be the i.i.d. random variables with respect to $\mu$.
    Assume that the Radon-Nikodym derivative $\partial_\mu\nu$ of $\nu$ is an element of $L^\infty(\mu)$.
    We further assume that $\mathcal{B}_{p,n}$ constitutes an orthogonal system in $H$ and that $V_{p,n} \subset L^2(\mu)$ (more precisely, any element of $V_{p,n}$ is square integrable with respect to $\mu$ and the natural map $V_{p,n} \to L^2(\mu)$ is injective).
    Let $\{u_i\}_{i=1}^{r_n}$ be an orthonormal basis of $V_{p,n}$ in $L^2(\nu)$.
    We define $\{q_{ij}\}_{i,j=1,\dots,r_n}$ by the complex numbers satisfying
    \[u_i = \sum_{j=1}^{r_n} q_{ij} v_j.\]
    For $k,\ell \ge 0$, let
    \[\mathbf{Q}_{k,\ell}(\nu) := (q_{ij})_{i \le r_k, j \le r_\ell}.\]
    Then, we have
       \begin{align}
        &\limsup_{N \to \infty}\left\| \mathbf{C}_{f, m}^\star - \widehat{\mathbf{C}}_{m, n, N}\right\|
        \le 
        L_m\|\mathcal{E}_{p,n}\|_{L^2(\mu)}\|\mathbf{Q}_{n,m}(\nu)\|_{\rm op}\quad{\rm a.e.}, \label{explicit asymptotics, C_f}
    \end{align}
    where 
    \[ L_m := \left\|\partial_\mu\nu \right\|_{L^\infty(\mu)} \|\mathbf{G}_m\|_{\rm op} \big\|\mathbf{G}_m^{-1/2} \big\|_{\rm op} \|C_f|_{V_{p,m}}\|_{\rm op}.\]
\end{corollary}
\begin{proof}
    We put $\mathbf{Q}_{k,\ell} := \mathbf{Q}_{k,\ell}(\nu)$.
    By definition of $\mathbf{Q}_{n,n}$, we have
    \[ \mathbf{I}_{n} = \mathbf{Q}_{n,n} \left(\langle v_i,v_j \rangle_{L^2(\nu)}\right)_{i,j=1,\dots,r_n} \mathbf{Q}_{n,n}^*.\]
    Thus, we have
    \[\mathbf{G}_n\left(\langle v_i,v_j \rangle_{L^2(\nu)}\right)_{i,j=1,\dots,r_n}^{-1}\mathbf{G}_n = \mathbf{G}_n \mathbf{Q}_{n,n}^*\mathbf{Q}_{n,n} \mathbf{G}_n. \]
    By the law of large number, we have
    \begin{align}
        \mathbf{G}_n\left(N^{-1}\mathbf{V}_n^{X_N}(\mathbf{V}_n^{X_N})^*\right)^{-1} \mathbf{G}_n \longrightarrow 
    \left(\langle v_i,v_j \rangle_{L^2(\mu)} \right)_{i,j=1,\dots,r_n}^{-1}~~~\text{a.e.}
    \label{a.e. convergence of matrix}
    \end{align}
    as $N \to \infty$.
    Since $\mathbf{G}_n$ is a diagonal matrix, and the following matrix
    \[ 
    \left(\left\|\partial_\mu\nu \right\|_{L^\infty(\mu)} \langle v_i,v_j \rangle_{L^2(\mu)}
    - 
    \langle v_i,v_j \rangle_{L^2(\nu)}\right)_{i,j=1,\dots,r_n} \]
    is a positive semi-definite matrix by the assumption, we have
    \[\lim_{N\to \infty}\|\mathbf{E}_{m,n}^{X_N}\|_{\rm op}^{1/2}
    \le \left\|\partial_\mu\nu \right\|_{L^\infty(\mu)} \| \mathbf{G}_m \mathbf{Q}_{n,m}^*\mathbf{Q}_{n,m} \mathbf{G}_m \|_{\rm op}^{1/2} 
    \le \|\mathbf{G}_m\|_{\rm op}\left\|\partial_\mu\nu \right\|_{L^\infty(\mu)} \cdot \|\mathbf{Q}_{n,m}\|_{\rm op}\quad{\rm a.e.}\]
    by the convergence \eqref{a.e. convergence of matrix}.
    Therefore, the inequality \eqref{explicit asymptotics, C_f} follows from Theorem \ref{thm: error analysis for PF operator for algorithm, disc}.
\end{proof}
\begin{remark}
    Theorem \ref{thm: error analysis for PF operator for algorithm, disc} and Corollary \ref{cor: error analysis for PF operator, infinite samples: disc.} are equivalent (up to constant).
    In fact, Theorem \ref{thm: error analysis for PF operator for algorithm, disc} corresponds to Corollary \ref{cor: error analysis for PF operator, infinite samples: disc.} in the case of $\mu = \nu = \frac{1}{N} \sum_{i=1}^N \delta_{x_i}$ 
\end{remark}

\subsection{Several analysis with bounded Koopman operators}
Here, we provide several results when $C_f$ is bounded.
Unfortunately, the Koopman operator on an RKHS is not necessarily bounded and only linear dynamical system can induce bounded Koopman operators in some cases, for example, on the RKHS of the Gaussian kernel \cite{Ikeda_Ishikawa_Sawano-2022} and the exponential kernel \cite{CMS}.
We can also show that a dynamical system around the fixed point is stable if the Koopman operator is bounded and any chaotic dynamical system is out of scope in the Koopman analysis on RKHS with boundedness \cite{ISHIKAWA2023109048}.
However, there many examples and studies of bounded Koopman operators in the fields of complex analysis as well (see \cite{Zh07, CM95, SS17} and reference therein).
It will be worth providing several results when the Koopman operator is bounded.

We show that the level sets of an eigenvector of $C_f^*|_{V_{p,n}}$ for sufficiently large $n$ provide invariant subsets, more generally, we have the following theorem:
\begin{proposition}
    We assume that $C_f$ is bounded on $H$ and $\iota'|_{\mathfrak{D}_p}$ is injective.
    Let $\lambda_1,\dots, \lambda_d$ be the eigenvalues (with multiplicity) of the Jacobian matrix ${\rm d}f_p$ and write $\lambda := (\lambda_1,\dots, \lambda_d)$.
    Let $u_{\alpha, n} \in V_{p,n}$ be a unit eigenvector of $C_f^*|_{V_{p,n}}$ associated with $\lambda^\alpha$ for $\alpha \in \mathbb{Z}_{\ge 0}^d$.
    Then, for $x \in {\Omega}$, we have
    \[\big|u_{\alpha,n}(f(x)) - \lambda^\alpha u_{\alpha,n}(x)\big| \le \|C_f - \lambda^\alpha\|_{\rm op} \mathcal{E}_{p,n}(x).\]
    Moreover, for any compact subset $K \subset {\Omega}$,
    \[\sup_{x \in K}\big|u_{\alpha,n}(f(x)) - \lambda^\alpha u_{\alpha,n}(x)\big| \to 0\]
    as $n \to \infty$.
\end{proposition}
\begin{proof}
    Let $x \in {\Omega}$ and let $\tilde{v}_x := \pi_n k_x \in V_{n,p}$.
    We have 
    \begin{align*}
        u_{\alpha,n}(f(x)) 
        &= \langle u_{\alpha, n}, C_f^*k_{x} \rangle_H\\
        &= \langle C_f u_{\alpha, n}, k_{x} \rangle_H\\
        &= \langle C_f u_{\alpha, n}, \tilde{v}_x\rangle_H + \langle C_f u_{\alpha,n}, k_x - \tilde{v}_x\rangle_H.
    \end{align*}
    Since 
    \[\langle C_f u_{\alpha, n}, \tilde{v}_x\rangle_H = \langle C_f^*|_{V_{n,p}}^* u_{\alpha,n}, \tilde{v}_x\rangle_H = \lambda^\alpha \langle u_{\alpha,n}, \tilde{v}_x\rangle_H, \]
    combining with the Cauchy--Schwarz inequality, we have
    \[\big|u_{\alpha,n}(f(x)) - \lambda^\alpha u_{\alpha,n}(x)\big| \le \|(C_f - \lambda^\alpha)u_{\alpha,n}\|_H \mathcal{E}_{p,n}(x) \le \|C_f - \lambda^\alpha\|_H \mathcal{E}_{p,n}(x).\]
    The second statement follows from Proposition \ref{prop: uniform convergence of the minimal error}.
\end{proof}

Let $\mu$ be a probability measure on ${\Omega}$ and assume that $\int_\Omega \sqrt{k(x,x)} {\rm d}\mu < \infty$.
Let $\Phi(\mu) \in H$ be the kernel mean embedding of $\mu$, that is the unique element such that $\langle h, \Phi(\mu)\rangle = \int_{\Omega} h(x)\,{\rm d}\mu(x)$for all $h \in H$.
For details of the theory for kernel mean embedding, see \cite{MAL-060}.
We define
\[\mathcal{E}_{p,n}(\mu) := \min_{v \in V_{p,n}}\| \Phi(\mu) - v \|_H.\]
For an integer $m \ge 0$, we define 
\[\widehat{\Phi}_n^m(\mu) := C_{f^m}^*\pi_n\Phi(\mu).\]
Then, $\widehat{\Phi}_n^m(\mu)$ will provide a prediction of the push-forward measure of $\mu$ by $f^m$ as follows:
\begin{proposition}\label{prop: prediction of distribution: fixed point}
    Assume that $C_f$ is bounded on $H$ and $\int_\Omega \sqrt{k(x,x)} {\rm d}\mu < \infty$.
    Then, we have
    \[ \big\|\Phi(f_*\mu) - \widehat{\Phi}_n^m(\mu)\big\|_H \le \|C_{f^m}\|_{\rm op} \mathcal{E}_{p,n}(\mu).\]
\end{proposition}
\begin{proof}
    Since $\Phi_n^0(\mu)$ is the orthonormal projection of $\Phi(\mu)$ to $V_{p,n}$, we see that
    \[\mathcal{E}_{p,n}(\mu) = \|\Phi_n^0(\mu) - \Phi(\mu)\|_H.\]
    Then, the statement follows since $\Phi_n^m(\mu) = C_{f^m}^*\Phi_n^0(\mu)$ and $C_{f^m}^*\Phi(\mu) = \Phi(f_*^m\mu)$.
\end{proof}

\section{Extended Koopman operators and rigged Hilbert spaces}\label{sec: generalized spectrum}
Here, we explain the theory of rigged Hilbert space with the Gelfand triple.
Then, we introduce the extended Koopman operator and show its Jordan--Chevalley decomposition and eigendecomposition.
In this section, we always assume Assumption \ref{asm: basic}.
We denote by $\mathfrak{r}: H' \to H$ the anti-linear isomorphism of \eqref{riesz representation}.

\subsection{Rigged Hilbert spaces}
First, we introduce the notion of the Gelfand triples.
We explicitly describe some natural identifications appearing in typical description of the theory of Gelfand triple, and our definition is slightly different from the usual one.
\begin{definition}\label{def: Gelfand triple}
Let $H$ be a Hilbert space.
Let $\Phi$ and $\Phi^*$ be locally convex Hausdorff topological spaces over $\mathbb{C}$
equipped with a continuous pairing
\begin{align}
\langle \cdot, \cdot \rangle : \Phi \times \Phi^* \longrightarrow \mathbb{C}
\label{pairing}
\end{align}
such that
\begin{align*}
     \langle f, a\mu + b\xi\rangle &= \overline{a}\langle f, \mu \rangle + \overline{b}\langle f, \xi \rangle \\
    \langle af + bg, \mu \rangle &= a\langle f, \mu\rangle + b\langle g, \xi \rangle
\end{align*}
for $a, b \in \mathbb{C}$, $f,g \in \Phi$, and $\mu, \xi \in \Phi^*$.
Assume that there exist injective continuous linear maps $\mathfrak{i}: \Phi \to H$ and $\mathfrak{j}:H \to \Phi^*$:
    \textcolor{blue}{
    \[ \Phi \overset{\mathfrak{i}}{\hookrightarrow} H  \overset{\mathfrak{j}}{\hookrightarrow} \Phi^* \cong \Phi'\]}
We call the triplet $(\Phi, H, \Phi^*)$ with continuous injections $\mathfrak{i}: \Phi \to H$ and $\mathfrak{j}:H \to \Phi^*$ the {\em Gelfand triple} if the following conditions hold:
\begin{enumerate}[(1)]
    \item $\mathfrak{i}(\Phi)$ is a dense subset of $H$, \label{density of test space}
    \item the natural anti-linear map $\mathfrak{s}: \Phi^* \to \Phi'; \mu \mapsto \langle \cdot, \mu \rangle$ \revised{via the paring \eqref{pairing}} is a bijective homeomorphism, \label{isomorphism of duals}
    \item the paring $\langle \cdot, \cdot \rangle$ is compatible with the bilinear form $\langle \cdot \mid \cdot\rangle$, namely, for $g \in \Phi$ and $\mu \in \Phi^*$,
    \begin{align}
        \langle g, \mu \rangle = \langle \mathfrak{s}(\mu) \mid g \rangle.
    \end{align}
    \item the paring $\langle \cdot, \cdot\rangle$ is compatible with the inner product $\langle\cdot, \cdot\rangle_H$, namely, for any $g \in \Phi$, $h \in H$, 
    \begin{align}
        \langle g, \mathfrak{j}(h)\rangle = \langle \mathfrak{i}(g), h\rangle_H.
    \end{align}

\end{enumerate}
We also call the Hilbert space equipped with the Gelfand triple {\em the rigged Hilbert space}.
\end{definition}

\revised{Usually, Gelfand triple is introduced only using $\Phi$ and its dual space $\Phi'$, and $\Phi'$ is identified with another more explicit topological vector space, so here we explicitly describe the identification by introducing $\Phi^*$ and $\mathfrak{s}$.
Although it depends on the context, we usually consider the weak$*$ topology for the dual space of $\Phi'$ instead of the strong topoology.
}

The Gelfand triplet is introduced for further investigation of the spectrum of linear operators.
We usually refer to the point spectrum computed via the ``analytic continuation'' with the Gelfand triplet as the resonance, the resonance poles, or the generalized eigenvalue, and they have been well studied in quantum mechanics.
For more details of the mathematical formulation and application based on the Gelfand triplet, see \cite{BG89, CHIBA2015324, Gel2, Gel, CII} and references therein. 

Using the Gelfand triple, we can define an extension of a linear operator, that will be used for defining the extended Koopman operator.
\begin{definition}\label{defn: extendion of linear operator}
Let $(\Phi, H, \Phi^*)$ be the Gelfand triple of a Hilbert space $H$.
Let $T: H \to H$ be a densely defined linear operator.
Assume that the adjoint operator $T^*$ satisfies $T^*(\mathfrak{i}(\Phi)) \subset \mathfrak{i}(\Phi)$ and $T^*|_\Phi := \mathfrak{i}^{-1}T\mathfrak{i}: \Phi \to \Phi$ is continuous.
We define the continuous linear map $T^\times: \Phi^* \to \Phi^*$ by the unique linear map satisfying
\[\langle f, T^\times \mu\rangle = \langle T^*|_\Phi f , \mu\rangle\]
for any $\mu \in \Phi^*$ and $f \in \Phi$, equivalently, 
\begin{align}
    T^\times := \mathfrak{s}^{-1} (T^*|_\Phi)'\mathfrak{s}. \label{equiv def of extention}
\end{align}
\end{definition}
As in the following proposition, we may actually regard $T^\times$ as an extension of $T$ to $\Phi^*$.
\begin{proposition}\label{prop: extension}
    Let $T: H \to H$ be a linear map with a dense domain $D(T)$.
    For $h \in D(T)$, we have $\mathfrak{j}Th = T^\times \mathfrak{j}h$.
\end{proposition}
\begin{proof}
It suffices to show that $\mathfrak{s} \mathfrak{j}T = (T^*|_\Phi)'\mathfrak{s} \mathfrak{j} $.
Let $g \in \Phi$ be an arbitrary element.
By direct computation, we have
\[
\langle \mathfrak{s} \mathfrak{j}T(h) \mid g \rangle =  \langle \mathfrak{i}(g), Th\rangle_H = \langle T^*\mathfrak{i}(g), h\rangle_H 
= \langle \mathfrak{s} \mathfrak{j} (h) \mid T^* (g) \rangle =  \langle (T^*|_\Phi)' \mathfrak{s}\mathfrak{j}(h) \mid g\rangle.\]
Thus, $\mathfrak{s} \mathfrak{j}T = (T^*|_\Phi)'\mathfrak{s} \mathfrak{j}$.
\end{proof}

\subsection{Extended Koopman operators} \label{subsec: extended koopman operators}
Here, we specifically define $\Phi$ and $\Phi^*$ in Definition \ref{def: Gelfand triple}  to be the space of observables and a ``limit'' of the space of observables, respectively.
Then, we will define the extended Koopman operator on $\Phi^*$.

Let $f: \Omega \to \Omega$ be a map.
Assume that Assumption \ref{asm: basic} holds, that $C_f: H \to H$ is densely defined, and that there exist $p_1,\dots, p_r \in \Omega$ satisfying Assumption \ref{asm: existence of a fixed point}.
Let 
\[\Lambda := \{p_1,\dots, p_r\} \subset \Omega.\]
We define
\begin{align}
    V_{\Lambda,n} := \sum_{i=1}^r V_{p_i,n}~~ \subset H.    
\end{align}
\revised{
Let
\begin{align}
    \Phi := \bigcup_{n \ge 0} V_{\Lambda, n}
\end{align}
and $\iota_n: V_{\Lambda,n} \hookrightarrow \Phi$ the inclusion map.
We equip the strongest topology such that the inclusion map $\iota_n$ for all $n \ge 0$ with $\Phi$.
We define the injection $\mathfrak{i}: \Phi \to H$ in Definition \ref{def: Gelfand triple} by the inclusion map.
}

\revised{
Let $\pi_{n+1,n}: V_{\Lambda,n+1} \to V_{\Lambda,n}$ be the orthogonal projection and let
\begin{align}
    \Phi^*  := \left\{ (\mu_n)_n \in \prod_{n=0}^\infty V_{\Lambda,n} : \pi_{n+1, n}(\mu_{n+1}) = \mu_n \right\}.
\end{align}
We equip the relative topology of the product topology $\prod_{n=0}^\infty V_{\Lambda,n}$ with $\Phi^*$. 
Recall $\pi_n: H \to V_{p,n}$ is the orthogonal projection \eqref{orthogonal projection}.
Then, we define the injection $\mathfrak{j}: H \to \Phi^*$ in Definition \ref{def: Gelfand triple} by $\mathfrak{j}(h) := \left(\pi_n(h)\right)_{n=0}^\infty$.
}

We define a paring 
\begin{align}
    \langle \cdot, \cdot\rangle: \Phi \times \Phi^* \longrightarrow \mathbb{C}; \left(h, (g_n)_n \right) \mapsto \lim_{n \to \infty} \langle h, g_n\rangle_H.
\end{align}

Then, we have the following proposition:
\begin{proposition}
    The triplet $(\Phi, H, \Phi^*)$ defined as above is a Gelfand triple.
\end{proposition}
\begin{proof}
    The condition \eqref{density of test space} follows from Proposition \ref{prop: density}.
    Regarding the condition \eqref{isomorphism of duals}, by \cite[Theorem 12]{MR0217557}, the natural map $\displaystyle \Phi' \to \Phi^*;~\mu \mapsto (\mu|_{V_{\Lambda,n}})_n$ induces an isomorphism.
    Since the inner product induces an anti-linear isomorphism $V_{\Lambda,n} \cong V_{\Lambda,n}'$, we have the isomorphism $\mathfrak{s}$ in \eqref{isomorphism of duals} as the composition of the above isomorphisms.
    Other two conditions are obvious.
\end{proof}

By \eqref{preserving property, PF operator} of Theorem \ref{thm: basic theorem for PF operator, discrete}, $C_f^*(\Phi) \subset \Phi$ holds.
Thus, we define the {\em extended Koopman operator} 
\begin{align}
    C_f^\times: \Phi^* \to \Phi^*
\end{align}
by Definition \ref{defn: extendion of linear operator}.

\revised{
\begin{definition}\label{def: proj.lim of linear maps}
    Let $\{A_n: V_{\Lambda,n} \to V_{\Lambda,n}\}_{n=0}^\infty$ be a family of line maps satisfying $A_n \pi_{n+1,n} = \pi_{n+1,n} A_{n+1}$
    Then, we define the continuous linear map on $\Phi^*$ by 
\[\lim A_n: \Phi^* \to \Phi^*;~(g_n) \to (A_ng_n).\]
\end{definition}
}

Using the notation in Definition \ref{def: proj.lim of linear maps}, the extended Koopman operator $C_f^\times$ has an explicit description:
\begin{proposition}\label{prop: another description of extension of Cf}
    We regard $C_f^*|_{V_{\Lambda,n}}$ as a linear map on $V_{\Lambda,n}$ and denote its adjoint in $V_{\Lambda,n}$ by $(C_f^*|_{V_{\Lambda,n}})^*$.
    Then, we have
    \begin{align}
        C_f^\times = \lim (C_f^*|_{V_{\Lambda,n}})^*
    \end{align}
\end{proposition}
\begin{proof}
    Since $(C_f^*|_{V_{\Lambda,n}})^* \pi_{n+1,n} = \pi_{n+1,n}(C_f^*|_{V_{\Lambda,n+1}})^*$, it follows from the equivalent definition \eqref{equiv def of extention} of $C_f^\times$.
\end{proof}
Moreover, if $C_f$ satisfies Assumption \ref{asm: domain of C_f}, we have a more explicit description as follows:
\begin{theorem}\label{thm: description of Koopman operator as a limit}
    Assume that $C_f$ satisfies Assumption \ref{asm: domain of C_f}.
    Then, we have
    \begin{align}\label{prop: extension theorem}
        (C_f^*|_{V_{\Lambda,n}})^* = \pi_n C_f \iota_n.
    \end{align}
    In particular, 
    \begin{align}
        C_f^\times = \lim \pi_n C_f \iota_n.
    \end{align}
\end{theorem}
\begin{proof}
    The second statement follows from Proposition \ref{prop: another description of extension of Cf}.
    We prove the first statement.
    It suffices to show that $\langle (C_f^*|_{V_{\Lambda,n}})^*g, h \rangle_H = \langle \pi_n C_f g, h\rangle_H$ for any $g,h \in V_{\Lambda,n}$.
    It is proved via the direct calculation as follows:
    \begin{align*}
        \langle (C_f^*|_{V_{\Lambda,n}})^*g, h \rangle_H = \langle g, C_f^*|_{V_{\Lambda,n}} h \rangle_H
         = \langle C_fg, h\rangle_H = \langle \pi_nC_fg, h \rangle_H.
    \end{align*}
\end{proof}
Each map $ \pi_n C_f \iota_n$ in Theorem \ref{thm: description of Koopman operator as a limit} is known as a finite approximation through projection into a finite dimensional subspace, which always appears when approximating the Koopman operator.
Theorem \ref{thm: description of Koopman operator as a limit} provides a crucial fact that $\pi_n C_f \iota_n$'s form a projective system, meaning that they are commutative with projections:
\[\pi_{n+1,n} ( \pi_{n+1} C_f \iota_{n+1}) = ( \pi_n C_f \iota_n) \pi_{n+1,n}.\]
This fact is important as it gives not only the approximation of the extended Koopman operator, but also the consistent family of eigenvectors.
Moreover, according to Theorem \ref{thm: error analysis for PF operator for algorithm, disc}, $ \pi_n C_f \iota_n$ is the adjoint of a matrix that can be estimated in a data-driven manner.
In this sense, the spaces $V_{\Lambda, n}$ of the intrinsic observables constructed from the jets provide an appropriate series of subspaces that enables the correct finite approximation of the Koopman operators.

\subsection{Eigendecomposition of the extended Koopman operators}
We use the same notation as in Section \ref{subsec: extended koopman operators}.
First, we describe the ``Jordan--Chevalley decompositions'' of the Perron--Frobenius operator $C_f^*|_{\Phi}$ and the extended Koopman operator $C_f^\times$:
\begin{theorem}\label{thm: jordan decomposition of C_f}
Suppose that Assumption \ref{asm: basic} holds and that $C_f$ is densely defined.
Let $r_n := \mathop{\rm dim}V_{\Lambda,n}$.
Then, there exist continuous linear operators $S_f^\times, N_f^\times: \Phi^* \to \Phi^*$ and $S_f, N_f: \Phi \to \Phi$ such that
\begin{align*}
    C_f^\times &= S_f^\times + N_f^\times,\\
    C_f^*|_{\Phi} &= S_f + N_f,
\end{align*}
with the following properties:
\begin{enumerate}[{\rm (1)}]
    \item there exist a family of complex numbers $\{\gamma_i\}_{i=1}^\infty$ and those of vectors $\{w_i\}_{i=1}^\infty \subset \Phi$ and $\{u_i\}_{i=1}^\infty \subset \Phi^*$ such that $\{w_i\}_{i=1}^{r_n}$ constitutes a basis of $V_{\Lambda, n}$, and
    \begin{align*}
        S_f^\times u_i &= \gamma_i u_i,\\
        S_f w_i &= \overline{\gamma_i} w_i,\\
        \langle u_i, w_j \rangle &= \delta_{i, j}
    \end{align*}
    for positive integers  $i,j \ge 1$, and 
    \begin{align*}
        S_f^\times u &= \sum_{i=1}^\infty \gamma_i \langle w_i, u \rangle u_i,\\
        S_fw &= \sum_{i=1}^\infty \overline{\gamma_i}  \langle w, u_i \rangle w_i
    \end{align*}
    hold for $u \in \Phi^*$ and $w \in \Phi$, where the two convergence on the right hand sides are in the topologies of $\Phi^*$ and $\Phi$, respectively, \label{eigendecompositions of S}
    \item for each $u \in \Phi^*$, $(N_f^\times)^nu \to 0$ as $n\to \infty$, \label{topologically nilpotent property of Nx}
    \item for each $w \in \Phi$, there exists $n \ge 0$ such that $N_f^nw = 0$, \label{nilpotent property of N}
    \item $S_fN_f = N_fS_f$ and $S_f^\times N_f^\times  = S_f^\times N_f^\times $. \label{commutativity of S and N}
\end{enumerate}
\end{theorem}
\begin{proof}
Let $C_f^*|_{V_{\Lambda,n}} = S_{f,n} + N_{f,n}$ be the Jordan--Chevalley decomposition (see, for example, \cite[Proposition 4.2]{Borel1991}), namely, $S_{f,n}: V_{\Lambda,n} \to V_{\Lambda,n}$ is diagonalizable, $N_{f,n}: V_{\Lambda,n} \to V_{\Lambda,n}$ is nilpotent, and $S_{f,n}N_{f,n} = N_{f,n}S_{f,n}$.
Then, by the uniqueness of the Jordan--Chevalley decomposition, we have $S_{f,n+1}|_{V_{\Lambda,n}} = S_{f,n}$ and $N_{f,n+1}|_{V_{\Lambda,n}} = N_{f,n}$.
We define $S_f, N_f:\Phi \to \Phi$ by the linear map such that $S_fg := S_{f,n}g$ and $N_fg := N_{f,n}g$ for $g \in V_{p,n}$.
Since $S_{f,n}^*$'s and $N_{f,n}^*$'s constitute projective systems, they induce linear maps $S_f^\times := \lim S_{f,n}^*$ and $N_f^\times := \lim N_{f,n}^*$ on $\Phi^*$.
By definition, the linear maps satisfy \eqref{topologically nilpotent property of Nx}, \eqref{nilpotent property of N}, and \eqref{commutativity of S and N}.

Since each $S_{f,n}$ is diagonalizable and $S_{f,n+1}|_{V_{\Lambda,n}} = S_{f,n}$, there exist $\{\gamma_i\}_{i=1}^\infty \subset \mathbb{C}$ and $\{w_i\}_{i=1}^\infty \subset \Phi$ such that 
the set $W_n := \{w_i\}_{i=1}^{r_n}$ spans $V_{\Lambda,n}$ for any $n\ge0$ and
\begin{align}
    S_f w_i = \overline{\gamma_i} w_i \label{eigendecomposition of S in proof}
\end{align} 
holds.
Let $i \le r_n$ be an arbitrary positive integer.
Since the orthogonal complement of $W_n \setminus \{ w_i \}$ in $V_{\Lambda,n}$ is 1-dimensional, there uniquely exists $u_{i,n} \in V_{\Lambda,n}$ such that
\begin{align*}
    \langle  w_i, u_{i,n}\rangle_H &= 1.
\end{align*}
Since $\langle  w_i, S_{f,n}^* u_{i,n} - \gamma_i u_{i,n}\rangle_{V_{\Lambda,n}} = 0$ for all $j=1,\dots, r_n$, we see that $S_{f,n}^* u_{i,n} - \gamma_i u_{i,n}$ is orthogonal to all the elements of $W_n$, resulting in $S_{f,n}^* u_{i,n} = \gamma_i u_{i,n}$.
Thus, for any $i,j \le r_n$, we have $\langle \pi_{n+1,n}u_{i,n+1}, w_j\rangle_H = \langle 
 u_{i,n}, w_j\rangle_H = \delta_{i,j}$, thus $\pi_{n+1,n}u_{i,n+1} = u_{i,n}$.
 Therefore, $u_i := (u_{i,n} )_{n=1}^\infty$ determines an element of $\Phi^*$ and satisfies
 \begin{align*}
    S_f^\times u_i &= \gamma_i u_i,\\
    \langle w_i, u_j \rangle &= \delta_{i,j}.
 \end{align*}
for all $i,j \in \mathbb{Z}_{\ge 0}$.
By combining this with \eqref{eigendecomposition of S in proof}, we have the first statement of \eqref{eigendecompositions of S}.
The second statement of \eqref{eigendecompositions of S} is obvious.
\end{proof}
\begin{remark}
    Let $\lambda_{i,1},\dots, \lambda_{i,d}$ be the eigenvalues (with multiplicity) of the Jacobian matrix ${\rm d}f_{p_i}$ at the fixed point $p_i$. 
    Let $\lambda_i:= (\lambda_{i,1},\dots, \lambda_{i,d})$. 
    Assume that $\iota'|_{\sum_{i=1}^r\mathfrak{D}_{p_i}}: \sum_{i=1}^r\mathfrak{D}_{p_i} \to H'$ is injective.
    Then,  $\left\{\lambda_i^\alpha : i=1,\dots, r_n,~\alpha \in \mathbb{Z}_{\ge0}^d\right\}$ coincides with the family $\{\gamma_i\}_{i=1}^\infty$ of complex numbers introduced in Theorem \ref{thm: jordan decomposition of C_f}.
\end{remark}
The extended Koopman operator $C_f^\times$ are diagonalizable if the multiplications of the eigenvalues of the Jacobian of $f$ at $p_i$ are distinct for each fixed point $p_i$, and the Perron--Frobenius operator $C_f^*|_{\Phi}$ is also diagonalizable as in the following corollary:
\begin{corollary}\label{cor: eigendecomposition: fixed point}
    Let $\lambda_{i,1},\dots, \lambda_{i,d}$ be the eigenvalues (with multiplicity) of the Jacobian matrix ${\rm d}f_{p_i}$ at the fixed point $p_i$. 
    Let $\lambda_i:= (\lambda_{i,1},\dots, \lambda_{i,d})$.
    Assume that $\iota'|_{\sum_{i=1}^r\mathfrak{D}_{p_i}}: \sum_{i=1}^r\mathfrak{D}_{p_i} \to H'$ is injective and that $\lambda_i^\alpha \neq \lambda_i^\beta$ for $\alpha, \beta \in \mathbb{Z}_{\ge 0}^d$ with $\alpha \neq \beta$.
    Then, there exist families of vectors $\big\{w_{i,\alpha}\big\}_{i \in \{1,\dots,r\},  \alpha \in \mathbb{Z}_{\ge 0}^d} \subset \Phi$ and $\left\{u_{i,\alpha}\right\}_{i \in \{1,\dots,r\},  \alpha \in \mathbb{Z}_{\ge 0}^d} \subset \Phi^*$ such that
        \begin{align*}
            C_f^\times u_{i, \alpha} &= \lambda_i^\alpha u_{i,\alpha},\\
            C_f^* w_{i, \alpha} &= \overline{\lambda_i^\alpha} w_{i, \alpha},\\
            \langle u_{i,\alpha}, w_{j,\beta} \rangle &= \delta_{(i,\alpha), (j,\beta)}
        \end{align*}
        for all $i,j \in \{1,\dots, r\}$, $\alpha, \beta \in \mathbb{Z}_{\ge 0}^d$, and 
        \begin{align*}
            C_f^\times u &= \sum_{i \in \{1,\dots,r\}, \alpha \in \mathbb{Z}_{\ge 0}^d} \lambda_i^\alpha \langle w_{i,\alpha}, u \rangle u_{i,\alpha},\\
            C_f^*w &= \sum_{i \in \{1,\dots,r\}, \alpha \in \mathbb{Z}_{\ge 0}^d} \overline{\lambda_i^\alpha}  \langle w, u_{i,\alpha} \rangle w_{i,\alpha}
        \end{align*}
        for $u \in \Phi^*$ and $w \in \Phi$, where the two convergence on the right hand sides are in $\Phi^*$ and $\Phi$, respectively.
\end{corollary}
\begin{proof}
It follows from Theorem \ref{thm: jordan decomposition of C_f}.
\end{proof}

According to Corollary \ref{cor: eigendecomposition: fixed point}, the extended Koopman operator $C_f^\times$ has eigenvectors under fairly mild conditions.
However, it is important to emphasize that these eigenvectors are not the eigenfunctions of the Koopman operator $C_f$ itself.
Nevertheless, as shown in the following proposition, they can be considered as approximate eigenvectors.
\begin{proposition}\label{prop: eigenfunction in weak sense}
    Assume that $C_f$ satisfies Assumption \ref{asm: domain of C_f}.
    Let $(g_n)_n \in \Phi^*$ be an eigenvector of the eigenvalue $\lambda$ of $C_f^\times$.
    Then, $C_fg_n - \lambda g_n \in V_{\Lambda,n}^\perp$.
\end{proposition}
\begin{proof}
    Let $h \in V_{\Lambda,n}$.
    By Proposition \ref{prop: extension theorem}, we have $\pi_n C_f g_n = \lambda g_n$.
    Thus, we have
    \begin{align*}
        \langle C_fg_n - \lambda g_n, h \rangle_H = \langle \pi_nC_fg_n - \lambda g_n, h \rangle_H = 0.
    \end{align*}
    Therefore, $C_fg_n - \lambda g_n \in V_{\Lambda,n}^\perp$.
\end{proof}

%

\section{Corresponding results to continuous dynamical systems} \label{sec: theory for continuous dynamical systems}
Here, we explain the corresponding theory for continuous dynamical systems and derive the results on the generator of Koopman operators.
As the Koopman operators for the flow maps of a continuous dynamical system constitute a semigroup, we can consider the derivative of the Koopman operators at time $0$.
We define the generator of the Koopman operator as the derivative of the Koopman operators at time $0$.
The behavior of the generator of the Koopman operators is generally tamer and more controllable than that of the Koopman operators.
Moreover, we may reconstruct the Koopman operators as the image of the generator under the exponential map.
Therefore, the elucidating the mathematical properties of the generators is a significant issue as well as those of the Koopman operators in the context of continuous dynamical system.

We use the notation introduced in Section \ref{sec: theory}.
Let $F=(F_1,\dots, F_d): {\Omega} \to \mathbb{R}^d$ be a map of class $C^\infty$.
We consider the ordinary differential equation:
\begin{align}
    \begin{split}
    \frac{{\rm d}z}{{\rm d}t} &= F(z).
    \end{split}\label{ODE}
\end{align}
For $x \in \Omega$ and an element $t$ of an open interval containing $0$, we denote $z(t)$ with $z(0) = x$ by $\phi^t(x)$.
We will assume the following assumption for some $p \in \Omega$:
\begin{assumption} \label{asm: existence of an equilibrium point}
    $p \in \Omega$ is an equilibrium point of $F$, namely $F(p) = 0$.
\end{assumption}

\subsection{The generator of Koopman operators on the canonical invariant subspaces}\label{subsec: invariant subspace of generators}

We define the continuous linear map $\mathcal{A}_F: \mathfrak{E}({\Omega}) \to \mathfrak{E}({\Omega})$ by
\begin{align}\label{generator}
    \mathcal{A}_F(h) := \lim_{t \to 0} \frac{h\circ \phi^t - h}{t} = \sum_{i=1}^d F_i\frac{\partial h}{\partial x_i} = F^\top \nabla h.
\end{align}
Then, we have the corresponding statements to Proposition \ref{prop: explicit description of representation matrix of f_*}:

\begin{proposition}\label{cor: explicit description of representation matrix of generator}
   \begin{enumerate}[{\rm (1)}]
        \item  \label{preserving property, continuous}  For each $n \ge 0$, $\mathcal{A}_F'(\mathfrak{D}_{p,n}) \subset \mathfrak{D}_{p,n}$.
        \item   \label{graded map, continuous} 
        Let ${\rm gr}^n_{\mathcal{A}_F'}: \mathfrak{D}_{p,n}/\mathfrak{D}_{p,n-1} \to \mathfrak{D}_{p,n}/\mathfrak{D}_{p,n-1}$ be the linear map induced by $\mathcal{A}_F'$ via \eqref{preserving property, continuous}.
        Then, we have
    \begin{align}
        {\rm gr}^n_{\mathcal{A}_F'} = \rho_n^{-1} \circ T_n({\rm d}F_p) \circ \rho_n,
    \end{align}
    where $T_n({\rm d}F_p)$ is the linear map on $\mathbb{C}[X_1,\dots, X_d]_n$ defined by
    \[ T_n({\rm d}F_p)(Q(X_1,\dots, X_d)) := (X_1,\dots,X_d)\cdot {\rm d}F_p \cdot \nabla Q\left(X_1,\dots,X_d\right)\]
    \item \label{representation matrix, continuous}
    Let $\mathbf{A}_{F,n}^\star \in \mathbb{C}^{\binom{n+d}{d} \times \binom{n+d}{d}}$ be the representation matrix of $\mathcal{A}_F'|_{\mathfrak{D}_{p,n}}: \mathfrak{D}_{p,n} \to \mathfrak{D}_{p,n}$ with respect to the basis $\{\delta_p^{(\alpha)} : |\alpha| \le n\}$.
    Then, $\mathbf{A}_{F,n}^\star$ is in the form of 
    \begin{align}\label{representation matrix of A_F}
        \mathbf{A}_{F,n}^\star=
        \begin{pmatrix}
            1 & * & * & *\\
            0 & \mathbf{K}_{p,1} & * & *\\
            \vdots & \ddots& \ddots & \vdots\\
            0 & \cdots & 0& \mathbf{K}_{p,n} 
        \end{pmatrix},
    \end{align}
    where $\mathbf{K}_{p,i}$ is the representation matrix of $T_i({\rm d}F_p)$ for the basis $\{X_1^{\alpha_1}\cdots X_d^{\alpha_d} : |\alpha| = i\} \subset \mathbb{C}[X_1,\dots,X_d]$.
    \item \label{eigenvalues of generator} Let $\mu_1,\dots, \mu_d$ be the eigenvalues (with multiplicity) of the Jacobian matrix ${\rm d}F_p$ and let $\mu := (\mu_1,\dots, \mu_d)$.
    Then, the set of eigenvalues of $\mathcal{A}_F^*|_{\mathfrak{D}_{p,n}}$ is
    \[\left\{ \alpha^\top \mu : |\alpha| \le n\right\}.\]
    \end{enumerate}
\end{proposition}
\begin{proof}
    We only give the proof of \eqref{preserving property, continuous} as the other statements are proved in the same way as in the proof of Proposition \ref{prop: explicit description of representation matrix of f_*}.
    It suffices to show that $\mathcal{A}_F'\delta_p^{(\alpha)}\in \mathfrak{D}_{p,n}$.
    Let $h \in \mathfrak{E}(\Omega)$.
    Then, by direct calculation, we have 
    \begin{align*}
        \langle \mathcal{A}_F' \delta_p^{(\alpha)} \mid h \rangle &= \langle \delta_p^{(\alpha)} \mid F\cdot \nabla h\rangle = F(p)^\top \nabla(\partial_x^\alpha h)(p) + \langle D \mid h \rangle
    \end{align*}
    for some $D \in \mathfrak{D}_{p,n}$.
    Since $F(p)=0$, we have $D=\mathcal{A}_F'\delta_p^{(\alpha)} \in  \mathfrak{D}_{p,n}$.
\end{proof}

\subsection{The generators of the Koopman operators on reproducing kernel Hilbert spaces}\label{subsec: generator on RKHS}
Here, we use the same notation as those in Sections \ref{sec: datadriven estimation of PF operator}, \ref{subsec: invariant subspace of generators}, and \ref{subsec: generator on RKHS}.
Let $H \subset \mathfrak{E}(\Omega)$ be a Hilbert space and assume that Assumption \ref{asm: continuity of inclusion} holds.
We denote the corresponding positive definite kernel by $k \in \mathfrak{E}(\Omega \times \Omega)$.
We define the linear operator $A_F: H \to H$ by
\begin{align}\label{generator on H}
    A_F(h) := \sum_{i=1}^d F_i\frac{\partial h}{\partial x_i} = F^\top \nabla h.
\end{align}

First, we state the corresponding results to Proposition \ref{prop: domain of PF operators} and  Theorem \ref{thm: basic theorem for PF operator, discrete}.
We omit the proofs of these statements as they are the same as the corresponding ones.
\begin{proposition}\label{prop: domain of PF operators, conti}
    Assume that $A_F$ is densely defined.
    Then, for any $\ell \in \mathfrak{E}(\Omega)$, we have $(\mathfrak{r} \circ \iota')(\ell) \in D(A_F^*)$.
    In particular, $V_p \subset D(A_F^*)$.
\end{proposition}
\begin{theorem}\label{thm: basic theorem for PF operator, continuous}
Assume that Assumption \ref{asm: existence of an equilibrium point} for some $p \in \Omega$ holds and $A_F$ is densely defined.
Then, we have the following statement:
\begin{enumerate}[{\rm (1)}]
    \item  $A_F^*(V_{p,n}) \subset V_{p,n}$ for all $n \ge 0$. \label{preserving property, generator}
    \item  Assume that $(\mathfrak{r} \circ \iota')|_{\mathfrak{D}_p}: \mathfrak{D}_p \to H$ is injective. \label{eigenvalues of generators}
    Then, the set of the eigenvalues of $A_F^*$ is 
    \[\left\{ \alpha^\top \mu : |\alpha| \le n\right\},\]
    where $\mu_1,\dots, \mu_d$ are the eigenvalues (with multiplicity) of the Jacobian matrix ${\rm d}F_p$ and let $\mu := (\mu,\dots, \mu_d)$.
\end{enumerate}
\end{theorem}

We note the relation between $A_F$ and $C_{\phi^t}$ as in the following proposition:
\begin{proposition}\label{prop: expC is equal to A}
    Let $t \ge 0$.
    Assume that $\phi^t(x)$ exists for any $x$ (meaning the ODE \eqref{ODE} can be solved until the time $t$ for any initial point).
    Assume that $A_F$ is densely defined.
    Then, for any integer $n \ge 0$, we have
    \begin{align}
        C_{\phi^t}^*|_{V_{p,n}} = {\rm exp}(tA_F^*|_{V_{p,n}}). \label{relation between A and C}
    \end{align}
\end{proposition}
\begin{proof}
    Since $p$ is the fixed point of $\phi^s$ for all $0 \le s \le t$, 
    $C_{\phi^t}^*$ and $A_F^*$ induce the linear map on the finite dimensional subspace $V_{p,n}$.
    The statement follows from the following ordinary differential equation: 
    \[ \frac{{\rm d}}{{\rm d}t} C_{\phi^t}^*|_{V_{p,n}} = A_F^*|_{V_{p,n}} C_{\phi^t}^*|_{V_{p,n}} \]
\end{proof}
\subsection{Estimation of the generators}\label{subsec: Data driven estimation of the generators}
We use the same notation and assumptions as those in Section \ref{sec: datadriven estimation of PF operator} and the previous sections.
We will further assume the following condition:
\begin{assumption}[Domain condition for $A_F$]\label{asm: domain of A_F}
    $V_p \subset D(A_F)$. 
\end{assumption}
We provide a useful sufficient condition for Assumption \ref{asm: domain of A_F} as in the following proposition.
\begin{proposition}\label{prop: sufficient condition for (2)}
    Assume that $V_p$ is closed under differential operators and $F_ih \in H$ for any $h \in V_p$ for $i=1,\dots, d$.
    Then, the domain of $A_F$ contains $V_p$.
\end{proposition}
\begin{proof}
    Since each component of $F_i\partial_{x_i} h$ is in $H$ for any $h \in V_p$ and $i=1,\dots,d$ by the assumption, it follows from \eqref{generator on H}.
\end{proof}
We have a more detailed sufficient condition for Assumption \ref{asm: domain of A_F} for special positive definite kernels in Section \ref{subsec: validity of domain assumption of koopman}. 
\begin{remark}
    Since the multiplication map $h \mapsto F_ih$ is a closed operator, the condition $F_ih \in H$ for all $h \in H$ is equivalent to the boundedness of $h \mapsto F_ih$.
\end{remark}

As in Section \ref{subsec: error bound analysis}, let $r_n := \mathop{\rm dim}V_{p, n}$.
We fix a basis $\mathcal{B}_p = \{ v_n \}_{n \ge 0}$ of $V_p$ such that $\mathcal{B}_{p,n} = \{v_i\}_{i=1, \dots, r_n}$ constitutes a basis of $V_{p,n}$.
Then, for $x \in \Omega$ and $X = (x_1, \dots, x_N) \in {\Omega}^N$, we define 
\begin{align*}
     {\bf v}_{n}(x) := \left(\overline{v_i(x)}\right)_{i=1}^{r_n},~~\mathbf{G}_{n}:= \left(\langle v_i, v_j \rangle_H\right)_{i,j = 1,\dots, r_n}, ~~\mathbf{V}_{n}^X :=& \left({\bf v}_{n}(x_1), \dots, {\bf v}_{n}(x_N)\right). 
\end{align*}
We also define 
\begin{align}
    \mathbf{W}_m^{X_N, F(X_N)} :=\sum_{i=1}^d \left(F_i(x_1) \partial_{x_i}{\bf v}_{n}(x_1), \dots, F_i(x_N)\partial_{x_i}{\bf v}_{n}(x_N)\right).
    \label{bold partial V}
\end{align}

Then, we have the corresponding results to Theorem \ref{thm: error analysis for PF operator for algorithm, disc} and Corollary \ref{cor: error analysis for PF operator, infinite samples: disc.} (we omit the proofs of these results as they are proved in the same way):
\begin{theorem}\label{thm: error analysis for PF operator for algorithm: conti.}
    Let $p \in \Omega$.
    Let $m, n \ge 0$ be integers with $m \le n$.
    Let $\mathbf{A}_{F,m}^\star$ be the representation matrix of $A_F^*|_{V_{p,m}}$ with respect to $\mathcal{B}_{p,m}$.
    Let $X_N := (x_1,\dots, x_N) \in {\Omega}^N$.
    We assume Assumption \ref{asm: basic}, $r_n \le N$, ${\rm span}(\{\pi_nk_{x_1},\dots,\pi_nk_{x_N}\}) = V_{p,n}$, and that $p$ satisfies Assumptions \ref{asm: existence of an equilibrium point} and \ref{asm: domain of A_F}.
    We define the matrices $\widehat{\mathbf{A}}_{m,n, N}$ and $\mathbf{E}_{m,n}^{X_N}$ of size $r_m$ by the leading principal minor matrices of order $r_m$ as follows:
    \begin{align*}
    \mathbf{G}_m^{-1}\mathbf{W}_m^{X_N, F(X_N)} (\mathbf{V}_n^{X_N})^\dagger\mathbf{G}_n = 
    \begin{pmatrix}
        \widehat{\mathbf{A}}_{m, n, N} & *
    \end{pmatrix},\\
    \mathbf{G}_n\left(N^{-1}\mathbf{V}_n^{X_N}(\mathbf{V}_n^{X_N})^*\right)^{-1} \mathbf{G}_n = 
    \begin{pmatrix}
        \mathbf{E}_{m,n}^{X_N} & *\\
        * &*
    \end{pmatrix}.
    \end{align*}
    Then, we have
    \begin{align}
        &\left\| \mathbf{A}_{F, m}^\star - \widehat{\mathbf{A}}_{m, n, N}\right\|
        \le 
        \left\|A_F|_{V_{p,m}}\right\|_{\rm op} \big\|\mathbf{G}_m^{-1/2}\big\|_{\rm op}
        \cdot\big\|\mathbf{E}_{m,n}^{X_N}\big\|_{\rm op}^{1/2} \sqrt{\frac{1}{N}\sum_{i=1}^N \mathcal{E}_{p,n}(x_i)^2}.
    \end{align}
\end{theorem}

\begin{corollary}\label{cor: error analysis for PF operator, infinite samples: conti.}
    We use the same notation and assume the assumptions in Theorem \ref{thm: error analysis for PF operator for algorithm: conti.}.
    Let $\mu$ be a Borel probability measure on $\Omega$.
    Let $\nu$ be an arbitrary $\sigma$-finite Borel measure absolutely continuous with respect to $\mu$ on $\Omega$.
    Suppose that $x_1,\dots, x_N$ be the i.i.d. random variables with respect to $\mu$.
    Assume that the Radon-Nikodym derivative $\partial_\mu\nu$ of $\nu$ is an element of $L^\infty(\mu)$.
    We further assume that $\mathcal{B}_{p,n}$ constitutes an orthogonal system in $H$ and that $V_{p,n} \subset L^2(\mu)$ (more precisely, any element of $V_{p,n}$ is square integrable with respect to $\mu$ and the natural map $V_{p,n} \to L^2(\mu)$ is injective).
    We define $\{q_{ij}\}_{i,j=1,\dots,r_n}$ by the complex numbers satisfying
    \[u_i = \sum_{j=1}^{r_n} q_{ij} v_j.\]
    For $k,\ell \ge 0$, let
    \[\mathbf{Q}_{k,\ell}(\nu) := (q_{ij})_{i \le r_k, j \le r_\ell}.\]
    Then, we have
       \begin{align}
        &\limsup_{N \to \infty}\left\| \mathbf{A}_{F, m}^\star - \widehat{\mathbf{A}}_{m, n, N}\right\|
        \le 
        L_m\|\mathcal{E}_{p,n}\|_{L^2(\mu)}\|\mathbf{Q}_{n,m}(\nu)\|_{\rm op}\quad{\rm a.e.}, \label{explicit asymptotics, A_F}
    \end{align}
    where 
    \[ L_m := \left\|\partial_\mu\nu \right\|_{L^\infty(\mu)} \|\mathbf{G}_m\|_{\rm op} \big\|\mathbf{G}_m^{-1/2} \big\|_{\rm op} \|A_F|_{V_{p,m}}\|_{\rm op}.\]
\end{corollary}

\subsection{Eigendecompostions of the generator of Koopman operators}
Here, we use the notation introduced in Section \ref{sec: generalized spectrum}.
We assume that Assumption \ref{asm: basic} holds and that there exist $p_1,\dots, p_r \in \Omega$ satisfying Assumption \ref{asm: existence of a fixed point} and define $\Lambda := \{p_1,\dots, p_r\} \subset \Omega$.
The corresponding results to Theorem \ref{thm: jordan decomposition of C_f} and Corollary \ref{cor: eigendecomposition: fixed point} are as follows:

\begin{theorem}\label{thm: jordan decomposition of A_F}
Suppose that Assumption \ref{asm: basic} holds and that $A_F$ is densely defined.
Let $r_n := \mathop{\rm dim}V_{\Lambda,n}$.
Then, there exist continuous linear operators $S_F^\times, N_F^\times: \Phi^* \to \Phi^*$ and $S_F, N_F: \Phi \to \Phi$ such that
\begin{align*}
    A_F^\times &= S_F^\times + N_F^\times,\\
    A_F^*|_{\Phi} &= S_F + N_F,
\end{align*}
with the following properties:
\begin{enumerate}[{\rm (1)}]
    \item there exist a family of complex numbers $\{\gamma_i\}_{i=1}^\infty$ and those of vectors $\{w_i\}_{i=1}^\infty \subset \Phi$ and $\{u_i\}_{i=1}^\infty \subset \Phi^*$ such that $\{w_i\}_{i=1}^{r_n}$ constitutes a basis of $V_{\Lambda, n}$, and
    \begin{align*}
        S_F^\times u_i &= \gamma_i u_i,\\
        S_F w_i &= \overline{\gamma_i} w_i,\\
        \langle u_i, w_j \rangle &= \delta_{i, j}
    \end{align*}
    for positive integers  $i,j \ge 1$, and 
    \begin{align*}
        S_F^\times u &= \sum_{i=1}^\infty \gamma_i \langle w_i, u \rangle u_i,\\
        S_Fw &= \sum_{i=1}^\infty \overline{\gamma_i}  \langle w, u_i \rangle w_i
    \end{align*}
    for $u \in \Phi^*$ and $w \in \Phi$, where the convergences on the right hand sides are in the topologies of $\Phi^*$ and $\Phi$, respectively, 
    \item for each $u \in \Phi^*$, $(N_F^\times)^nu \to 0$ as $n\to \infty$, 
    \item for each $w \in \Phi$, there exists $n \ge 0$ such that $N_F^nw = 0$,
    \item $S_FN_F = N_FS_F$ and $S_F^\times N_F^\times  = S_F^\times N_F^\times $. 
\end{enumerate}
\end{theorem}

\begin{corollary}\label{cor: eigendecomposition: equilibrium point}
    Let $\mu_{i,1},\dots, \mu_{i,d}$ be the eigenvalues (with multiplicity) of the Jacobian matrix ${\rm d}F_{p_i}$ at the equilibrium point $p_i$. 
    Let $\mu_i:= (\mu_{i,1},\dots, \mu_{i,d})$.
    Assume that $\mu_i^\top \alpha \neq \mu_i^\top\beta$ for $\alpha, \beta \in \mathbb{Z}_{\ge 0}^d$ satisfying $\alpha \neq \beta$.
    Then, there exists a family of vectors $\big\{w_{i,\alpha}\big\}_{i \in \{1,\dots,r\},  \alpha \in \mathbb{Z}_{\ge 0}^d} \subset \Phi$ and $\left\{u_{i,\alpha}\right\}_{i \in \{1,\dots,r\},  \alpha \in \mathbb{Z}_{\ge 0}^d} \subset \Phi^*$ such that
    \begin{align*}
        A_F^\times u_\alpha &= \mu_i^\top\alpha u_{i,\alpha},\\
        A_F^* w_\alpha &= \overline{\mu_i^\top\alpha} w_{i, \alpha},\\
        \langle u_{i,\alpha}, w_{j,\beta} \rangle &= \delta_{(i,\alpha), (j,\beta)}
    \end{align*}
    for all $i,j \in \{1,\dots, r\}$, $\alpha, \beta \in \mathbb{Z}_{\ge 0}^d$, and 
    \begin{align*}
        A_F^\times u &= \sum_{i \in \{1,\dots,r\}, \alpha \in \mathbb{Z}_{\ge 0}^d} \mu_i^\top\alpha \langle w_{i,\alpha}, u \rangle u_{i,\alpha},\\
        A_F^*w &= \sum_{i \in \{1,\dots,r\}, \alpha \in \mathbb{Z}_{\ge 0}^d} \overline{\mu_i^\top\alpha}  \langle w, u_{i,\alpha} \rangle w_{i,\alpha}
    \end{align*}
    for $u \in \Phi^*$ and $w \in \Phi$, where the convergences on the right hand sides are in $\Phi^*$ and $\Phi$, respectively.
\end{corollary}

\section{Estimations with the exponential kernels and Gaussian kernels} \label{sec: special kernels}
Here, we introduce the Gaussian kernel and the exponential kernel, where Assumption \ref{asm: basic} trivially holds, and discuss their properties.
We precisely estimate the right hand sides of Corollaries \ref{cor: error analysis for PF operator, infinite samples: disc.} and \ref{cor: error analysis for PF operator, infinite samples: conti.} and finally prove the explicit convergence rate for JetEDMD.
In this section, we will use the notation introduced in Section \ref{sec: datadriven estimation of PF operator} and set $\Omega = \mathbb{R}^d$.
We denote by $L^2(\mathbb{R}^d)$ the $L^2$-space with respect to the Lebesgue measure on $\mathbb{R}^d$ and define the Fourier transform $\mathcal{F}[h]$ for $h \in  L^2(\mathbb{R}^d)$ by
\begin{align}
    \mathcal{F}[h](\xi) := \frac{1}{(2\pi)^{d/2}}\int_{\mathbb{R}^d} h(x) e^{-\mathrm{i}x^\top \xi} {\rm d}x.
\end{align}

\subsection{Exponential kernels and Gaussian kernels}
For $\sigma>0$ and $b \in \mathbb{R}^d$, we denote by $H^{\rm e}(\sigma, b)$ the RKHS associated with the {\em exponential kernel} defined by
\begin{align}
    k^{\rm e}(x,y) = e^{\frac{(x-b)^\top (y-b)}{\sigma^2}}. \label{exponential kernel}
\end{align}
For $\sigma>0$, we denote by $H^{\rm g}(\sigma)$ the RKHS associated with the {\em Gaussian kernel} defined by 
\begin{align}
    k^{\rm g}(x,y) =  e^{-|x-y|^2/2\sigma^2}. \label{gaussian kernel}
\end{align}

First, we provide the explicit description of the RKHS for each kernel as in the following propositions:
\begin{proposition}[RKHS of the exponential kernel]\label{prop: RKHS w.r.t exp kernel}
    Let $\sigma>0$ and $b \in \mathbb{R}^d$.
    We define the Hilbert spaces $H_0$ and $H_1$ by
    \begin{align*}
       H_0 &:= \left\{h(x) = \sum_{\alpha \in \mathbb{Z}_{\ge0}^d} a_\alpha (x-b)^\alpha : \sum_{\alpha \in \mathbb{Z}_{\ge 0}^d} \sigma^{2|\alpha|}\alpha! |a_\alpha|^2 < \infty \right\}, \\
        H_1 &:= \left\{ h|_{\mathbb{R}^d} : h\text{ is holomorphic on }\mathbb{C}^d\text{ and }\int_{\mathbb{R}^d \times \mathbb{R}^d} |h(x + y \mathrm{i})|^2 e^{-(\|x-b\|^2 + \|y\|^2)/\sigma^2}\,{\rm d}x {\rm d}y < \infty\right\},
    \end{align*}
    equipped with the inner products
    \begin{align*}
       \langle h, g\rangle_{H_0}  &:= \sum_{\alpha \in \mathbb{Z}_{\ge0}^d} \sigma^{2|\alpha|}\alpha! a_\alpha \overline{b_\alpha}, \\
        \langle h,\,g\rangle_{H_1} &:= \frac{1}{(\pi\sigma^2)^d}\int_{\mathbb{R}^d \times \mathbb{R}^d} h(x + y \mathrm{i})\overline{g(x+y\mathrm{i})} e^{-(\|x-b\|^2 + \|y\|^2)/\sigma^2}\,{\rm d}x {\rm d}y,
    \end{align*}
    where we put $h(x) = \sum_\alpha a_\alpha (x-b)^\alpha$, $g(x) = \sum_\alpha b_\alpha (x-b)^\alpha$ for the inner product of $H_0$.
    Then, we have $H^{\rm e}(\sigma, b) = H_0 = H_1$.
\end{proposition}
\begin{proof}
    As for $H^{\rm e}(\sigma, b)=H_0$, it follows from the fact that $k_x \in H_0$ and $\langle h, k_x\rangle_{H_0} = h(x)$ for any $x \in \mathbb{R}^d$, and that the subspace generated by $k_x$'s for $x \in \Omega$ is dense in $H_0$.
    As for $H^{\rm e}(\sigma, b) = H_1$, see for example, \cite[Chapter 2]{Zhu12}.
\end{proof}
\begin{proposition}[RKHS of the Gaussian kernel] \label{prop: RKHS w.r.t. gaussian kernel}
    Let $\sigma > 0$.
    We define the Hilbert spaces $H_0$ and $H_1$ by
    \begin{align*}
    H_0 &= 
    \left\{
    h \in L^2(\mathbb{R}^d) : \int_{\mathbb{R}^d} |\mathcal{F}[h](\xi)|^2  e^{\sigma^2|\xi|^2/2} {\rm d}\xi < \infty
    \right\},\\
    H_1 &= 
    \left\{
    h|_{\mathbb{R}^d} :  h\text{ is holomorphic on }\mathbb{C}^d\text{ and }\int_{\mathbb{R}^d\times\mathbb{R}^d} |h(x + y\mathrm{i})|^2 e^{-2\|y\|^2/\sigma^2} {\rm d}x {\rm d}y < \infty
    \right\},
\end{align*}
    equipped with the inner products
    \begin{align*}
        \langle g,\,h\rangle_{H_0} &:= \frac{\sigma^d}{(2\pi)^{d/2}}\int_{\mathbb{R}^d} \mathcal{F}[g](\xi)\overline{\mathcal{F}[h](\xi)} e^{\sigma^2|\xi|^2/2} {\rm d}\xi,\\
        \langle g,\,h\rangle_{H_1} &:= \frac{1}{(\pi\sigma^2)^d}\int_{\mathbb{R}^d\times\mathbb{R}^d} g(x + y \mathrm{i})\overline{h(x + y \mathrm{i})} e^{-2\|y\|^2/\sigma^2} {\rm d}x {\rm d}y.
    \end{align*}
    Then, we have $H^{\rm g}(\sigma)=H_0=H_1$
\end{proposition}
\begin{proof}
We show $H^{\rm g}(\sigma)=H_0$.
It suffices to show that $\langle h, k_x\rangle_{H_1} = h(x)$ for $h \in H_1$ and $x \in \mathbb{R}^d$.
Since 
\[\mathcal{F}[k_x](\xi) = \frac{1}{\sigma^d} e^{-\mathrm{i}x^\top \xi} e^{-\sigma^2|\xi|^2/2},\]
we have
\[\langle \mathcal{F}[h], \mathcal{F}[k_x] \rangle_{H_1} = \frac{1}{\sigma^d} \mathcal{F}^{-1}[\mathcal{F}[h]](x) = h(x).\]
As for $H^{\rm g}(\sigma)=H_1$, see \cite[Theorem 1.17]{MR3560890_SaitohSawano}.
\end{proof}

We define 
\[\mathfrak{m}_{\sigma, b}: H^{\rm e}(\sigma, b) \to H^{\rm g}(\sigma)\]
by $(\mathfrak{m}_{\sigma, b}h)(x) := h(x)e^{-\|x-b\|^2/2\sigma^2}$.
Then, $\mathfrak{m}_{\sigma, b}$ induces an isomorphism from $H^{\rm e}(\sigma, b)$ to $H^{\rm g}(\sigma)$ as in the following proposition:
\begin{proposition}\label{prop: isom of exp and gauss}
    Let $\sigma>0$ and $b \in \mathbb{R}^d$.
    Then, $\mathfrak{m}_{\sigma, b}$ induces the isomorphism of Hilbert spaces such that $\mathfrak{m}_{\sigma, b}k^{\rm e}_p = e^{\|p-b\|^2/2\sigma^2}k^{\rm g}_p$ for all $p \in \mathbb{R}^d$.
\end{proposition}
\begin{proof}
    It follows from the the equalities $H = H_1$ for both kernels in Propositions \ref{prop: RKHS w.r.t exp kernel} and \ref{prop: RKHS w.r.t. gaussian kernel}.
\end{proof}

\subsection{Explicit description of the intrinsic observables}
\label{subsec: explicit intinsic observables}
In this section, we denote $r_n := \binom{n+d}{d}$ (eventually, it coincides with the dimension of the space of the intrinsic observables).
Let $\mathfrak{P}_n$ be the set of polynomial functions on $\mathbb{R}^d$ of degree less than or equal to $n$.
Then, we can determine the space of intrinsic observables $V_{p,n}$ as follows:
\begin{proposition}\label{prop: explicit form of Vpn}
    Let $k$ be the exponential kernel \eqref{exponential kernel} or the Gaussian kernel \eqref{gaussian kernel}.
    Then, $V_{p,n} = \{q k_p : q \in \mathfrak{P}_n\}$.
\end{proposition}
\begin{proof}
    It follows from Corollary \ref{cor: alternative definition of Vpn}.
\end{proof}
\begin{proposition}\label{prop: ONB for exp kernel}
    Let $\sigma > 0$ and $b \in \mathbb{R}^d$.
    We fix a numbering of $\mathbb{Z}_{\ge 0}^d$ such that  $\mathbb{Z}_{\ge 0}^d = \{\alpha^{(i)}\}_{i=1}^\infty$ and $|\alpha^{(i)}| \le |\alpha^{(j)}|$ if $i \le j$.
    Let 
    \[v_{p, i}^{\rm e}(x) := v_{p, i}^{\rm e}(x; \sigma, b) := \frac{(x-p)^{\alpha^{(i)}}}{\sigma^{|\alpha^{(i)}|}} \exp\left(\frac{(p-b)^\top(2x-p-b)}{2\sigma^2}\right).\]
    Then, $\{v_{p,i}^{\rm e}\}_{i = 1}^{r_n}$ is an orthogonal basis of $V_{p,n}$ and the matrix $\mathbf{G}_n$ of \eqref{gram matrix} is a diagonal matrix whose $i$-th diagonal component is $\alpha^{(i)}!$.
    In particular, $\iota'$, the dual map of $\iota$ of \eqref{inclusion} is injective on $\sum_{p \in \mathbb{R}^d} \mathfrak{D}_p$.
\end{proposition}
\begin{proof}
    By direct calculation, we have
    \[v_{p,i}^{\rm e}(z) \overline{v_{p,j}^{\rm e}(z)} e^{-\|z-b\|^2/\sigma^2} = \sigma^{-|\alpha^{(i)}|-|\alpha^{(j)}|}(z-p)^{\alpha^{(i)}}(\overline{z} - p)^{\alpha^{(j)}} e^{-\|z-p\|^2/2\sigma^2}. \]
    The first statement follows from the combination of the above equality with the identity $H=H_1$ in Proposition \ref{prop: RKHS w.r.t exp kernel} and \cite[Proposition 2.1]{Zhu12}.
    As for the injectivity of $\iota'$, by Lemma \ref{lem: linear independence of Delta p}, it suffices to show the linear independence of $\cup_{p \in \mathbb{R}^d} \{v_{p,i}^{\rm e}\}_{i = 1}^{r_n}$ for all $n \ge 0$.
    We omit the proof as it is easily proved via the induction on $n$ using the differential operator.
\end{proof}
We denote $\mathbf{v}_{p,n}(x)$ and $\mathbf{G}_n$ defined by the orthogonal system $\{v_{p, i}^{\rm e}\}_{i=1}^\infty$ by $\mathbf{v}_{p,n}^{\rm e}(x;\sigma)$ and $\mathbf{G}_n^{\rm e}$, respectively.

\begin{proposition}\label{prop: ONB for gaussian kernel}
    Let $\sigma > 0$.
    We fix a numbering of $\mathbb{Z}_{\ge 0}^d$ such that  $\mathbb{Z}_{\ge 0}^d = \{\alpha^{(i)}\}_{i=1}^\infty$ and $|\alpha^{(i)}| \le |\alpha^{(j)}|$ if $i \le j$.
    Let 
    \[v_{p,i}^{\rm g}(x) := v_{p,i}^{\rm g}(x; \sigma) := \frac{(x-p)^{\alpha^{(i)}}}{\sigma^{|\alpha^{(i)}|}} e^{\frac{-\|x-p\|^2}{2\sigma^2}}.\]
    Then, $\{v_{p,i}^{\rm g}\}_{i = 1}^{r_n}$ is an orthogonal basis of $V_{p,n}$ and $\mathbf{G}_n$ is a diagonal matrix whose $i$-th diagonal component is $\alpha^{(i)}!$.
    In particular, $\iota'$ is injective on $\sum_{p \in \mathbb{R}^d} \mathfrak{D}_p$.
\end{proposition}
\begin{proof}
    Since
    \[\mathfrak{m}_{\sigma,0}(v_{p,i}^{\rm e}) = v_{p,i}^{\rm g},\]
    it follows from Proposition \ref{prop: isom of exp and gauss}.
\end{proof}
We denote $\mathbf{v}_{p,n}(x)$ and $\mathbf{G}_n$ defined by the orthogonal system $\{v_{p, i}^{\rm g}\}_{i=1}^\infty$ by $\mathbf{v}_{p,n}^{\rm g}(x;\sigma)$ and $\mathbf{G}_n^{\rm g}$, respectively.

\subsection{Validity of Assumptions \ref{asm: domain of C_f} and \ref{asm: domain of A_F}}\label{subsec: validity of domain assumption of koopman}
Here, we provide several sufficient conditions for Assumptions \ref{asm: domain of C_f} and \ref{asm: domain of A_F} to be satisfied.

First, we discuss the discrete case.
Let $f: \mathbb{R}^d \to \mathbb{R}^d$ be a map of class $C^\infty$.
At present, we know the nontrivial results only for the RKHS with respect to the exponential kernel.
\begin{proposition}\label{prop: sufficient condition for domain condition, discrete, exp, p=b}
    Let $\sigma > 0$ and $p \in \mathbb{R}^d$.
    Assume $f(\cdot)^\alpha \in H^{\rm e}(\sigma, p)$ for all $\alpha \in \mathbb{Z}_{\ge 0}^d$.
    Then, Assumption \ref{asm: domain of C_f} for $p \in \mathbb{R}^d$ holds for $H^{\rm e}(\sigma, p)$.
\end{proposition}
\begin{proof}
    Since $V_p$ in $H^{\rm e}(\sigma, p)$ coincides with the space of polynomial functions, it follows from Proposition \ref{prop: RKHS w.r.t exp kernel}.
\end{proof}
For example, if we take $f$ as the polynomial map, $f(\cdot)^\alpha \in H^{\rm e}(\sigma, p)$ always holds and thus Assumption \ref{asm: domain of C_f} is satisfied.
Here, we note that we assume $b=p$ in Proposition \ref{prop: sufficient condition for domain condition, discrete, exp, p=b}, and this assumption is essential.
We do not assume $b = p$ although the class of the dynamical systems satisfying Assumption \ref{asm: domain of C_f} is more restricted, in the following result:

\begin{proposition}\label{prop: sufficient condition for domain condition, discrete, exp, quadratic}
    Let $\sigma > 0$.
    Let $p \in \mathbb{R}^d$, and $b \in \mathbb{R}$ with $p \neq b$.
    Assume that there exists $\widetilde{f}: \mathbb{C}^d \to \mathbb{C}^d$ and $0 \le \varepsilon < 1/2$ such that $\widetilde{f}|_{\mathbb{R}^d} = f$ and
    \begin{align}
        \|\widetilde{f}(z)\| \le \frac{\varepsilon }{\|p - b\|}  \|z\|^2
    \end{align}
    for all $z \in \mathbb{C}^d$.
    Then, Assumption \ref{asm: domain of C_f} for $p \in \mathbb{R}^d$ holds for $H^{\rm e}(\sigma, b)$.
\end{proposition}
\begin{proof}
    Since any element of $V_p$ is in the for of product of a polynomial and $e^{(p-b)^\top z/\sigma^2}$, it follows from definition of $H_1$ in Proposition \ref{prop: RKHS w.r.t exp kernel}.
\end{proof}

Next, we discuss the continuous case.
In contrast to the discrete case, a quite general class of dynamical systems satisfies Assumption \ref{asm: domain of A_F}.
Let $F: \mathbb{R}^d \to \mathbb{R}^d$ be a $C^\infty$ map and regard it as a vector field on $\mathbb{R}^d$.
Regarding the continuous dynamical system of \eqref{ODE}, we have the following propositions:
\begin{proposition}\label{prop: sufficient condition for domain condition, continuous, exp}
    Let $\sigma > 0$ and $b \in \mathbb{R}^d$.
    Let $k^{\rm e}$ be the positive definite kernel for $H^{\rm e}(\sigma, b)$.
    Let $p \in \mathbb{R}^d$.
    Assume that there exists $\widetilde{F} = (\widetilde{F}_1,\dots, \widetilde{F}_d) : \mathbb{C}^d \to \mathbb{C}^d$ such that $\widetilde{F}|_{\mathbb{R}^d} = F$ and that
    \begin{align}
        \int_{\mathbb{R}^d \times \mathbb{R}^d}|(\widetilde{F}_iPk_p^{\rm e})(x + y \mathrm{i})|^2 e^{-(\|x-b\|^2 + \|y\|^2)/\sigma^2}\,{\rm d}x {\rm d}y < \infty
    \end{align}
    for $i=1,\dots,d$ and any polynomial function $P$.
    Then, Assumption \ref{asm: domain of A_F} for $p \in \mathbb{R}^d$ holds for $H^{\rm e}(\sigma, b)$.
\end{proposition}
\begin{proof}
    By Proposition \ref{prop: explicit form of Vpn}, it follows from the definition of $A_F$.
\end{proof}
\begin{proposition}\label{prop: sufficient condition for domain condition, continuous, gauss}
    Let $\sigma > 0$.
    Let $k^{\rm g}$ be the positive definite kernel for $H^{\rm g}(\sigma)$.
    Let $p \in \mathbb{R}^d$.
    Assume that there exists $\widetilde{F}: \mathbb{C}^d \to \mathbb{C}^d$ such that $\widetilde{F}|_{\mathbb{R}^d} = F$ and that
    \begin{align}
        \int_{\mathbb{R}^d \times \mathbb{R}^d}|(F_iPk_p^{\rm e})(x + y \mathrm{i})|^2 e^{-\|y\|^2/\sigma^2}\,{\rm d}x {\rm d}y < \infty
    \end{align}
    for $i=1,\dots,d$ and any polynomial function $P$.
    Then, Assumption \ref{asm: domain of A_F} for $p \in \mathbb{R}^d$ holds for $H^{\rm g}(\sigma)$.
\end{proposition}
\begin{proof}
    By Proposition \ref{prop: explicit form of Vpn}, it follows from the definition of $A_F$.
\end{proof}

\subsection{Explicit convergence rates} \label{subsec: explicit convergence rate}

\begin{proposition}\label{prop: explicit asymptotic of E, exp kernel}
    Let $\sigma>0$ and $b \in \mathbb{R}^d$.
    We consider $\mathcal{E}_{p,n}$  of \eqref{error} for $H^{\rm e}(\sigma, b)$.
    For any $x \in \mathbb{R}^d$, we have
    \begin{align}
    \mathcal{E}_{p,n}(x) 
    = e^{\frac{\|x-b\|^2 - \|x-p\|^2}{2\sigma^2}} \sqrt{\sum_{m={n+1}}^\infty \frac{1}{m!} \left(\frac{\|x-p\|^2}{\sigma^2} \right)^{m} } \le \frac{1}{\sqrt{(n+1)!}}\left(\frac{\|x-p\| }{\sigma} \right)^{n+1} e^{\|x-b\|^2/2\sigma^2}. \label{asymptotic rate of error; exp kernel}
   \end{align}
\end{proposition}
\begin{proof}
Let 
    \[v_\alpha(x) := \frac{(x-p)^{\alpha}}{\sqrt{\alpha !}\sigma^{|\alpha|}} e^{\frac{(p-b)^\top(2x-p-b)}{2\sigma^2}}.\]
Since $v_\alpha$'s constitutes an orthonormal basis of $V_p$ by Proposition \ref{prop: ONB for exp kernel}, we have
\begin{align*}
    k(x,y) =  \sum_{\alpha \in \mathbb{Z}_{\ge 0}^d}v_\alpha(x)v_\alpha(y).
\end{align*}
Thus, we have
\begin{align*}
    \mathcal{E}_{p,n}(x)^2 = e^{\frac{(p-b)^\top (2x-p-b)}{\sigma^2}}\sum_{m=n+1}^\infty \frac{1}{m!} \left( \frac{\|x-p\|^2}{\sigma^2} \right)^m.
\end{align*}
Since $(p-b)^\top (2x-p-b) = \|x-b\|^2 - \|x-p\|^2$, we have the first equality.
The second follows from the inequality:
\begin{align*}
    \sum_{m=n+1}^\infty \frac{1}{m!} \left( \frac{\|x-p\|^2}{\sigma^2} \right)^m 
    \le \frac{1}{(n+1)!} e^{ \frac{\|x-p\|^2}{\sigma^2}}\left(\frac{\|x-p\|^2 }{\sigma^2} \right)^{n+1}.
\end{align*}
\end{proof}
\begin{proposition}\label{prop: explicit asymptotic of E, gaussian}
    Let $\sigma>0$.
    We consider $\mathcal{E}_{p,n}$ of \eqref{error} for $H^{\rm g}(\sigma)$.
    For any $x \in \mathbb{R}^d$, we have
      \begin{align}
       \mathcal{E}_{p,n}(x) = e^{-|x-p|^2/\sigma^2} \sqrt{\sum_{m={n+1}}^\infty \frac{1}{m!} \left(\frac{\|x-p\|^2}{\sigma^2} \right)^{m} } \le \frac{1}{\sqrt{(n+1)!}}\left(\frac{\|x-p\| }{\sigma} \right)^{n+1}. \label{asymptotic rate of error}
   \end{align}
\end{proposition}
\begin{proof}
    It follows from the combination of Proposition \ref{prop: explicit asymptotic of E, exp kernel} with Proposition \ref{prop: isom of exp and gauss}.
\end{proof}

Then, we have a detailed result of Corollary \ref{cor: error analysis for PF operator, infinite samples: disc.} as follows:

\begin{theorem}\label{thm: explicit asymptotic, disc}
    We use the same notation and assumptions as those in Corollary \ref{cor: error analysis for PF operator, infinite samples: disc.}.
    Assume that $H = H^{\rm e}(\sigma, b)$ or $H^{\rm g}(\sigma)$.
    Assume that $\mu$ is compactly supported and absolutely continuous  with respect to the Lebesgue measure, namely there exists a compactly supported non-negative measurable function $\rho: \mathbb{R}^d \to \mathbb{R}_{\ge 0}$ such that $\mu = \rho(x) {\rm d}x$.
    We assume that there exists a rectangle $R(a,r):=\prod_{i=1}^d[a_i-r_i,a_i+r_i]$ for $a_i \in \mathbb{R}$ and $r_i \in \mathbb{R}_{>0}$ ($i=1,\dots,d$) such that $\mathrm{ess.inf}_{x\in R(a,r)}\rho(x) > 0$.
    Let
    \begin{align*}
        B_1 &:= \sup_{x \in {\rm supp}(\rho) } \|x - p\|,\\
        B_2 &:= \sup_{i=1,\dots,d}\left(1+ \frac{|a_i-p_i|}{r_i}\right).
    \end{align*}
    Then, there exists a constant $C > 0$ independent of $n$ such that
    \begin{align}
        \limsup_{N \to \infty}\left\| \mathbf{C}_{f, m}^\star - \widehat{\mathbf{C}}_{m, n, N}\right\|
        \le \frac{C n^{m+d}}{\sqrt{(n+1)!}}
        \left(\frac{2B_1B_2}{\sigma}\right)^n~~~{\rm a.e.}
        \label{explicit asymptotic rate}
    \end{align}
\end{theorem}
\begin{proof}
    First, we show the case of $H = H^{\rm e}(\sigma, b)$.
    We note that $\pi_nk_{x_1}, \dots, \pi_nk_{x_N}$ are linearly independent with probability $1$ by the assumption $\mathrm{ess.inf}_{x\in R(a,r)}\rho(x) > 0$.
    We define 
    \begin{align}
        \nu := e^{-\frac{(p-b)^\top(2x-p-b)}{\sigma^2}}\mathbf{1}_{R(a,r)}(x) {\rm d}x.
    \end{align}
    Now, we estimate $\|\mathcal{E}_{p,n}\|_{L^2(\mu)}$, and $\|\mathbf{Q}_{n,m}(\nu)\|$ in Corollary \ref{cor: error analysis for PF operator, infinite samples: disc.}.
    As for $\|\mathcal{E}_{p,n}\|_{L^2(\mu)}$, by Proposition \ref{prop: explicit asymptotic of E, exp kernel} , we immediately have
    \begin{align}
        \|\mathcal{E}_{p,n}\|_{L^2(\mu)} \lesssim \frac{B_1^n}{\sigma^n\sqrt{(n+1)!}}. \label{estimation of E}
    \end{align}
    Next, we estimate $\|\mathbf{Q}_{n,m}(\nu)\|$.
    Let $\{P_n(t)\}_{n\ge0}$ be the orthonormal polynomials in the $L^2$-space on $\mathbb{R}$ associated with the weight function $\mathbf{1}_{[-1,1]}(x){\rm d}x$ such that the degree of $P_n$ coincides with $n$.
    Let $\gamma_{n}$ be the leading coefficient of $P_{n}(t)$ and let $\omega_{n,1},\dots,\omega_{n,n}$ be the roots (with multiplicity) of $P_{n}(t)$.
    Then, by \cite[Theorem 12.7.1]{MR0372517}, the following inequality holds:
    \begin{align}
        |\gamma_n| \lesssim 2^{n}. \label{leading coefficient of legendre polynomials}
    \end{align}
    We define
    \[Q_\alpha(x) := e^{\frac{(p-b)^\top(2x-p-b)}{2\sigma^2}}\prod_{i=1}^d \frac{1}{\sqrt{r_i}} P_{\alpha_i}\left(\frac{x_i-a_i}{r_i}\right) .\]
    Then, $\{Q_\alpha(x)\}_{|\alpha| \le n}$ constitutes an orthonormal basis of $V_{p,n}$ in $L^2(\nu)$.
    Let
    \begin{align*}
        \frac{1}{\sqrt{r_i}} P_{\alpha_i}\left(\frac{x_i-a_i}{r_i}\right) 
        &= \sum_{j=0}^{\alpha_i} c_{\alpha_i,j}(x_i-p_i)^j
    \end{align*}
    for some real numbers $c_{\alpha_i,j} \in \mathbb{R}^d$.
    By \cite[Theorem 3.3.1]{MR0372517}, we have $|\omega_{n,i}| \le 1$
    for $i=1,\dots,n$.
    Thus, by
    \[\frac{1}{\sqrt{r_i}} P_{\alpha_i}\left(\frac{x_i-a_i}{r_i}\right) 
    = \frac{\gamma_{\alpha_i}}{r_i^{\alpha_i+1/2}}\prod_{j=1}^{\alpha_i} \left(x_i -p_i - (r_i\omega_{\alpha_i,j}  + a_i - p_i)\right), \]
    we have
    \begin{align}
        |c_{\alpha_i,j}| 
        &\le \sum_{\substack{ S \subset \{1,\dots, \alpha_i\} \\ |S|=\alpha_i - j}} 
    \frac{\gamma_{\alpha_i}}{r_i^{\alpha_i+1/2}} \prod_{\ell \in S}|r_i\omega_{\alpha_i,\ell}  + a_i - p_i| \nonumber\\
        &\le \binom{\alpha_i}{j} \frac{\gamma_{\alpha_i}}{r_i^{\alpha_i+1/2}} (r_i + |a_i - p_i|)^{\alpha_i - j} \nonumber\\
        &\le \binom{\alpha_i}{j} \frac{\gamma_{\alpha_i}}{r_i^{j+1/2}} \left(1 + \frac{|a_i - p_i|}{r_i}\right)^{\alpha_i}. \label{estimation of c}
    \end{align}
    Let
    \[q_{\alpha,\beta} := \prod_{i=1}^dc_{\alpha_i,\beta_i}.\]
    Thus, combining this with \eqref{leading coefficient of legendre polynomials}, we have
    \begin{align}
        |q_{\alpha,\beta}| \lesssim \frac{2^{|\alpha|}}{\min(r_1,\dots,r_d)^{|\beta| + d/2}}  \left(1 + \sup_{i=1,\dots,d}\frac{|a_i - p_i|}{r_i}\right)^{|\alpha|}\prod_{i=1}^d \alpha_i^{\beta_i}. 
    \end{align}
    Since
    \begin{align*}
        \|\mathbf{Q}_{n,m}(\nu)\|_{\rm op}^2 
        \le \|\mathbf{Q}_{n,m}(\nu)\|^2 
        = \sum_{|\alpha| \le n, |\beta| \le m} |q_{\alpha,\beta}|^2
        \lesssim (nm)^d \sup_{|\alpha| \le n, |\beta| \le m}|q_{\alpha,\beta}|^2,
    \end{align*}
    we have
    \begin{align}
        \|\mathbf{Q}_{m,n}(\nu)\|_{\rm op} \lesssim C_m n^{m+d} 2^n \left(1 + \sup_{i=1,\dots,d}\frac{|a_i - p_i|}{r_i}\right)^n.
    \label{estimation of Q}
    \end{align}
    for some constant $C_m>0$ only depending on $m$.
    Therefore, by \eqref{estimation of E} and \eqref{estimation of Q}, we have \eqref{explicit asymptotic rate} in the case of $H = H^{\rm e}(\sigma, b)$.
    Regarding the case of $H^{\rm g}(\sigma)$, the argument is completely the same as above but we use $\nu := e^{|x-p|^2}{\sigma^2}\mathbf{1}_{R(a,r)}(x) {\rm d}x$ and Proposition \ref{prop: explicit asymptotic of E, exp kernel} in the corresponding parts.
\end{proof}
\begin{remark}
    For example, if the density function $\rho$ of $\mu$ is continuous with compact support, then such a rectangle $R(a,r)$ always exists.
\end{remark}

We also have the corresponding theorem to the continuous case.
As the proof is completely the same as that of Theorem \ref{thm: explicit asymptotic, disc}, we omit it.
\begin{theorem}\label{thm: explicit asymptotic, conti}
    We use the same notation and assumptions as in those Corollary \ref{cor: error analysis for PF operator, infinite samples: disc.}.
    Assume that $H = H^{\rm e}(\sigma, b)$ or $H^{\rm g}(\sigma)$.
    Assume that $\mu$ is compactly supported and absolutely continuous with respect to the Lebesgue measure, namely there exists a compactly supported $\rho: \mathbb{R}^d \to \mathbb{R}_{\ge 0}$ such that $\mu = \rho(x) {\rm d}x$.
    We assume that there exists a rectangle $R(a,r):=\prod_{i=1}^d[a_i-r_i,a_i+r_i]$ for $a_i \in \mathbb{R}$ and $r_i \in \mathbb{R}_{>0}$ ($i=1,\dots,d$) such that $\mathrm{ess.inf}_{x\in R(a,r)}\rho(x) > 0$.
    Let
    \begin{align*}
        B_1 &:= \sup_{x \in {\rm supp}(\rho)} \|x - p\|,\\
        B_2 &:= \sup_{i=1,\dots,d}\left(1+ \frac{|a_i-p_i|}{r_i}\right).
    \end{align*}
    Then, there exists a constant $C > 0$ independent of $n$ such that
    \begin{align}
        \limsup_{N \to \infty}\left\| \mathbf{A}_{F, m}^\star - \widehat{\mathbf{A}}_{m, n, N}\right\|
        \le \frac{Cn^{m+d}}{\sqrt{(n+1)!}}
        \left(\frac{2B_1B_2}{\sigma}\right)^n~~~{\rm a.e.}
    \end{align}
\end{theorem}

\subsection{Explicit reconstruction error for dynamical systems}
As a crucial application of JetEDMD, we can reconstruct the dynamical system from only a set of discrete data on the trajectories.
We actually provide a theoretical guarantee for our reconstruction method as well as the algorithms and numerical results in Section \ref{sec: system identification}.
We will use the results in this section to prove the theoretical guarantee.

First, we provide an estimate of the reconstruction error for discrete dynamical systems.
\begin{theorem}\label{thm: reconstruction, discrete}
    Let $H = H^{\rm e}(\sigma, b)$ for $\sigma > 0$ and $b \in \mathbb{R}^d$.
    Let $p \in \mathbb{R}^d$.
    Let $f=(f_1,\dots, f_d): \mathbb{R}^d\to \mathbb{R}^d$ be a map of class $C^\infty$.
    Assume that Assumptions \ref{asm: existence of a fixed point} and \ref{asm: domain of C_f} for $p$ hold.
    Let $m \ge 1$ be an integer and let $\{v_i\}_{i=1}^{r_m}$ be a basis of $V_{p,m}$.
    Let $\mathbf{C}_{f,m}^\star$ be the representation matrix of $C_f^*|_{V_{p,m}}$ with respect to the basis.
    For $i=1,\dots, d$, we define 
    \begin{align}
        \widehat{f}_{m,i}(x) := (\partial_{x_i}\mathbf{v}_{p,m}(b))^* \sigma^2 \mathbf{C}_{f,m}^\star \mathbf{G}_m^{-1} \mathbf{v}_{p,m}(x) + b_i.
        \label{reconstructed dynamical system, discrete, exp}
    \end{align}
    Then, for any compact subset $K \subset \mathbb{R}^d$ and $i=1,\dots,d$, we have
    \begin{align}
        \sup_{y \in K}\left|f_i(y) - \widehat{f}_{m,i}(y)\right| 
        \le \|f_i-b_i\|_H 
        \frac{\sup_{y \in K}\|y-p\|^{m+1}e^{\|y-b\|^2/2\sigma^2}}{\sigma^{m+1}\sqrt{(m+1)!}}.  \label{reconstruction error rate, discrete, exp}
    \end{align}
\end{theorem}
\begin{proof}
    Since
    \[ C_f^*k_y(z) = k(f(y),z) = e^{(f(y)-b)^\top (z-b)/\sigma^2},\]
    we have
    \[f_i(y)-b_i = \sigma^2 \partial_{x_i}C_f^*k_y(x)|_{x=b} 
    = \sigma^2 \langle C_f^* k_y, \partial_{x_i}k(x,\cdot)|_{x=b}\rangle_H.\]
    On the other hand, by Proposition \ref{prop: coefficient of minimizer}, we have
    \[C_f^*\pi_nk_y = (v_\alpha)_{|\alpha| \le m}^\top \mathbf{C}_{f,m}^\star \mathbf{G}_m^{-1} \mathbf{v}_{p,m}(y). \]
    Thus, we have
    \[\sigma^2 \langle C_f^*\pi_nk_y, \partial_{x_i}k(x,\cdot)|_{x=b}\rangle_H = \widehat{f}_{m,i}(y) - b_i.\]
    Since $C_f^* \partial_{x_i}k(x,\cdot)|_{x=b} = \sigma^{-2}(f_i - b)$,  we have the formula \eqref{reconstruction error rate, discrete, exp} by Lemma \ref{lem: lemma for system identification} and Proposition \ref{prop: explicit asymptotic of E, exp kernel}.
\end{proof}

Next, we provide a corresponding result for continuous dynamical systems as follows.
We omit the proof as it is the same as the discrete case.
\begin{theorem}\label{thm: reconstruction continuous, exp}
    Let $H = H^{\rm e}(\sigma, b)$ for $\sigma > 0$ and $b \in \mathbb{R}^d$.
    Let $p \in \mathbb{R}^d$.
    Let $F=(F_1,\dots, F_d): \Omega \to \mathbb{R}^d$ be a map of class $C^\infty$.
    Assume that Assumptions \ref{asm: existence of an equilibrium point} and \ref{asm: domain of A_F} hold.
    Let $m \ge 1$ be an integer and let $\{v_i\}_{i=1}^{r_m}$ be a basis of $V_{p,m}$.
    Let $\mathbf{A}_{F,m}^\star$ be the representation matrix of $A_F^*$ on $V_{p,m}$.
    For $i=1,\dots, d$, we define 
    \begin{align}
        \widehat{F}_{m,i}(x) := (\partial_{x_i}\mathbf{v}_m(p))^* \sigma^2 \mathbf{A}_{F,m}^\star \mathbf{G}_n^{-1} \mathbf{v}_m(x)
        \label{reconstructed dynamical system, continuous, exp}
    \end{align}
    Then, for any compact subset $K \subset \mathbb{R}^d$ and $i=1,\dots,d$, we have
    \begin{align}
        \sup_{y \in K}\left|F_i(y) - \widehat{F}_{m,i}(y)\right| \le \|F_i\|_H \frac{\sup_{y \in K}\|y-p\|^{m+1}}{\sigma^{m+1}\sqrt{(m+1)!}}.  \label{reconstruction error rate, continuous, exp}
    \end{align}
\end{theorem}

Under Assumption \ref{asm: domain of C_f} or \ref{asm: domain of A_F}, for example, under the assumption for Proposition \ref{prop: sufficient condition for domain condition, discrete, exp, p=b}, \ref{prop: sufficient condition for domain condition, discrete, exp, quadratic}, or \ref{prop: sufficient condition for domain condition, continuous, exp}, the right hand side of \eqref{reconstructed dynamical system, discrete, exp} can be estimated from data as in Theorem \ref{thm: explicit asymptotic, disc} or \ref{thm: explicit asymptotic, conti}.

We have another result for the reconstruction error for continuous dynamical systems using the Gaussian kernel.
\begin{theorem}\label{thm: reconstruction continuous, gaussian}
    Let $H = H^{\rm g}(\sigma)$ for $\sigma > 0$.
    Let $p \in \mathbb{R}^d$.
    Let $F=(F_1,\dots, F_d): \Omega \to \mathbb{R}^d$ be a map of class $C^\infty$.
    Assume that Assumptions \ref{asm: existence of an equilibrium point} and \ref{asm: domain of A_F} for $p \in \mathbb{R}^d$ hold.
    Let $m \ge 1$ be an integer and let $\{v_i\}_{i=1}^{r_m}$ be a basis of $V_{p,m}$.
    Let $\mathbf{A}_{F,m}^\star$ be the representation matrix of $A_F^*$ on $V_{p,m}$.
    For $i=1,\dots, d$, we define 
    \begin{align}
        \widehat{F}_{m,i}(x) := (\partial_{x_i}\mathbf{v}_m(x))^* \sigma^2 \mathbf{A}_{F,m}^\star \mathbf{G}_n^{-1} \mathbf{v}_m(x)
        \label{reconstructed dynamical system, gaussian}
    \end{align}
    Then, for any compact subset $K \subset \mathbb{R}^d$ and $i=1,\dots,d$, we have
    \begin{align}
        \sup_{y \in K}\left|F_i(y) - \widehat{F}_{m,i}(y)\right| \le \sup_{y \in K}\|A_F\partial_{x_i}k(x,\cdot)|_{x=y}\|_H \frac{\sup_{y \in K}\|y-p\|^{m+1}}{\sigma^{m-1}\sqrt{(m+1)!}}.  \label{reconstruction error rate, gaussian}
    \end{align}
\end{theorem}
\begin{proof}
    Since
    \[ A_F^*k_y(z) = \sum_{i=1}^d F_i(y)\partial_xk(x,z)|_{x=y} = \sum_{i=1}^d F_i(y)\frac{z_i - y_i}{\sigma^2}e^{-|y-z|^2/2\sigma^2},\]
    we have
    \[F_i(y) = \sigma^2 \partial_{x_i}A_F^*k_y(x)|_{x=y} 
    = \sigma^2 \langle A_F^* k_y, \partial_{x_i}k(x,\cdot)|_{x=y}\rangle_H.\]
    On the other hand, by Proposition \ref{prop: coefficient of minimizer}, we have
    \[A_F^*\pi_nk_y = (v_\alpha)_{|\alpha| \le m}^\top \mathbf{A}_{F,m}^\star \mathbf{G}_n^{-1} \mathbf{v}_m(y). \]
    Thus, we have
    \[\sigma^2 \langle A_F^*\pi_nk_y, \partial_{x_i}k(x,\cdot)|_{x=y}\rangle_H = \widehat{F}_{m,i}(y).\]
    Therefore, by Lemma \ref{lem: lemma for system identification} and Proposition \ref{prop: explicit asymptotic of E, gaussian}, we have the formula \eqref{reconstruction error rate, gaussian}.
\end{proof}
Since the right hand side of \eqref{reconstructed dynamical system, gaussian} can be estimated from data as in Theorem \ref{thm: explicit asymptotic, conti}, our framework can effectively estimate the unknown system using only data.


\section{Jet Extended Dynamic Mode Decomposition}\label{sec:NumericalSimulation}
Here, we describe the details of the algorithms of JetEDMD and provide several computational results.
The code for the numerical simulation is available at \url{https://github.com/1sa014kawa/JetEDMD}.

\subsection{Eigenvalues and eigenvectors of Perron--Frobenius operators and extended Koopman operators}\label{subsec: data-driven estimation}
Here, we introduce the details of the estimation algorithms for eigenvalues and eigenvectors of the Perron-Frobenius operator and the extended Koopman operator.
We also provide several numerical results using the van der Pol oscillator, Duffing oscillator, and the H\'enon map.
We describe the algorithms in Algorithms \ref{algorithm: data-driven PF operator estimation, discrete}, \ref{algorithm: data-driven PF operator estimation, continuous}, \ref{algorithm: eigenfunction, discrete}, and \ref{algorithm: eigenfunction, continuous}.
We provide several remarks for these algorithms and accurate convergence theorems for them as follows:
\begin{remark}\label{rmk: matrix log}
    When $f = \phi^{T_s}$ for some $T_s>0$ in Algorithm \ref{algorithm: data-driven PF operator estimation, discrete}, using the output $\widehat{\mathbf{C}}$, we may define
    \begin{align}
        \widehat{\mathbf{A}} := \frac{1}{T_s}\log \widehat{\mathbf{C}}.
    \end{align}
    Then, $\widehat{\mathbf{A}}$ can provide an estimation of the generator $A_F|_{V_{p,m}}$ by Proposition \ref{prop: expC is equal to A} under suitable conditions.
    However, in general, we need to carefully choose a branch of $\log \widehat{\mathbf{C}}$ to get a correct approximation of the generator (see also Figure \ref{fig: vdP_eigenvalues_exp} below).
\end{remark}

\begin{remark}
    The Gaussian-Hermite quadrature (see, for example, \cite[Section 3]{MR2401585}) is a more effective and faster way to compute $\mathbf{G}_{m}^{ij}$ than a usual numerical integration in Algorithm \ref{algorithm: eigenfunction, discrete}.
\end{remark}

Using Theorem \ref{thm: explicit asymptotic, disc} and Proposition \ref{prop: sufficient condition for domain condition, discrete, exp, p=b}, we immediately have the following theorems:
\begin{theorem}[Convergence of Algorithm \ref{algorithm: data-driven PF operator estimation, discrete}]\label{thm: convergence of algorithm, discrete}
    Assume that $\Omega = \mathbb{R}^d$.
    Let $p \in \mathbb{R}^d$ be a fixed point of $f$ and let $k^{\rm e}(x,y) = e^{-(x-p)^\top(y-p)/\sigma^2}$ for some $\sigma > 0$.
    Let $x_1,\dots, x_N$ be i.i.d random variables of the distribution $\rho(x){\rm d}x$ with compactly supported density function $\rho$ such that ${\rm ess.inf}_{x \in U}\rho(x) >0$ for some open subset $U \subset \mathbb{R}^d$.
    Assume that $f^\alpha \in H^{\rm e}(\sigma,p)$ for all $\alpha \in \mathbb{Z}_{\ge 0}^d$.
    Let $\widehat{\mathbf{C}}(m,n,N)$ be the output of Algorithm \ref{algorithm: data-driven PF operator estimation, discrete} with input $m$, $n$, $p$, $X=(x_1,\dots,x_N)$, and $Y=(y_1,\dots,y_N)$.
    Then, we have
    \begin{align}
        \lim_{n\to \infty}\lim_{N\to \infty} \left\|\widehat{\mathbf{C}}(m,n,N) -\mathbf{G}_m\mathbf{C}_{f,m}^\star\mathbf{G}_m^{-1}\right\| = 0~~{\rm a.e.},
    \end{align}
    where $\mathbf{C}_{f,m}^\star$ is the representation matrix of $C_f^*|_{V_{p,m}}$ with respect to $\mathcal{B}_{p,m}$.
\end{theorem}
We note that the condition $f^\alpha \in H^{\rm e}(\sigma,p)$ for all $\alpha \in \mathbb{Z}_{\ge 0}^d$ holds for any holomorphic map on $\mathbb{C}^d$ with exponentially growth, in particular, any polynomial map.
\begin{theorem}[Convergence of Algorithm \ref{algorithm: data-driven PF operator estimation, continuous}]\label{thm: convergence of algorithm, continuous}
    Assume $\Omega = \mathbb{R}^d$ and let $k$ be the Gaussian kernel or the exponential kernel.
    Let $p \in \mathbb{R}^d$ be an equilibrium point of $F$.
    Let $x_1,\dots, x_N$ be i.i.d random variables of the distribution $\rho(x){\rm d}x$ with compactly supported density function $\rho$ such that ${\rm ess.inf}_{x \in U}\rho(x) >0$ for some open subset $U \subset \mathbb{R}^d$.
    Assume that $p$ satisfies Assumptions \ref{asm: existence of an equilibrium point} and \ref{asm: domain of A_F} (see Propositions \ref{prop: sufficient condition for domain condition, continuous, exp} and \ref{prop: sufficient condition for domain condition, continuous, gauss}).
    Define $\widehat{\mathbf{A}}(m,n,N)$ as the output of Algorithm \ref{algorithm: data-driven PF operator estimation, continuous} with input $m$, $n$, $p$, $X=(x_1,\dots,x_N)$, and $Y=(y_1,\dots,y_N)$.
    Then, we have
    \begin{align}
        \lim_{n\to \infty}\lim_{N\to \infty} 
        \left\|\widehat{\mathbf{A}}(m,n,N) - \mathbf{G}_m\mathbf{A}_{F,m}^\star\mathbf{G}_m^{-1}\right\| = 0~~~{\rm a.e.},
    \end{align}
    where $\mathbf{A}_{F,m}^\star$ is the representation matrix of $A_F^*|_{V_{p,m}}$ with respect to $\mathcal{B}_{p,m}$.
\end{theorem}

\begin{algorithm}
\caption{Estimation of Perron--Frobenius operators for discrete dynamical systems}
\label{algorithm: data-driven PF operator estimation, discrete}
\begin{algorithmic}[1]
\Require Positive integers $m$ and $n$ with $m \le n$, a fixed point $p \in {\Omega}$ of $f$, $X := (x_1,\dots, x_N) \in \Omega^N$, and $Y := (y_1 , \dots, y_N) \in \Omega^N$ with $N \ge r_n$, where $y_i = f(x_i)$.
\State Construct an orthogonal basis $\{v_{p, i}\}_{i=1}^{r_n}$ of $V_{p,n}$ such that $\{v_{p,i}\}_{i=1}^{r_m}$ constitutes a basis of $V_{p,m}$.
\State Define ${\bf v}_{p,m} \in V_{p,n}^{r_m}$ with variable $x$ by ${\bf v}_{p,m}(x) := (\overline{v_{p,i}(x)})_{i=1}^{r_m}$.
\State Define ${\bf v}_{p,n} \in V_{p,n}^{r_n}$ with variable $x$ by ${\bf v}_{p,n}(x) := (\overline{v_{p,i}(x)})_{i=1}^{r_n}$.
\State Construct the matrix $\mathbf{V}_{m}^Y := \left({\bf v}_{p,m}(y_1), \dots, {\bf v}_{p,m}(y_N)\right)$ of size $r_m \times N$.
\State Construct the matrix $\mathbf{V}_{n}^X := \left({\bf v}_{p,n}(x_1), \dots, {\bf v}_{p,n}(x_N)\right)$ of size $r_n \times N$.
\State Compute $\mathbf{C}_0 := \mathbf{V}_{m}^Y(\mathbf{V}_{n}^X)^\dagger$, where $(\cdot)^\dagger$ indicates the Moore--Penrose pseudo inverse.
\State Extract the leading principal minor matrix $\widehat{\mathbf{C}} $ of $\mathbf{C}_0$  of order $r_m$:
\begin{align}
    \mathbf{C}_0 = \mathbf{V}_m^Y (\mathbf{V}_n^X)^\dagger = 
    \begin{pmatrix}
        \widehat{\mathbf{C}} & *
    \end{pmatrix}. \label{computation of hatCnm}
\end{align}
\Ensure 
    $\widehat{\mathbf{C}}$.
\end{algorithmic}
\end{algorithm}

\begin{algorithm}
\caption{Estimation of the generators of Perron--Frobenius operators for continuous dynamical systems}\label{algorithm: data-driven PF operator estimation, continuous}
\begin{algorithmic}[1]
    \Require 
    Positive integers $m$ and $n$ with $m \le n$, a fixed point $p \in {\Omega}$ of $f$, 
    $X := (x_1,\dots, x_N) \in \Omega^N$, and $Y := (y_{ij})_{i,j} \in \mathbb{R}^{d\times N}$ with $N \ge r_n$, where $y_{ij}=F_i(x_j)$.
    \State Construct an orthogonal basis $\{v_{p, i}\}_{i=1}^{r_n}$ of $V_{p,n}$ such that $\{v_{p,i}\}_{i=1}^{r_m}$ constitutes a basis of $V_{p,m}$.
    \State Define ${\bf v}_{p,m} \in V_{p,n}^{r_m}$ with variable $x$ by ${\bf v}_{p,m}(x) := (\overline{v_{p,i}(x)})_{i=1}^{r_m}$.
    \State Define ${\bf v}_{p,n} \in V_{p,n}^{r_n}$ with variable $x$ by ${\bf v}_{p,n}(x) := (\overline{v_{p,i}(x)})_{i=1}^{r_n}$.
    \State $\mathbf{W}_{m}^{X,Y} := \mathbf{O} \in \mathbb{C}^{r_m \times N}$.
    \For{$i=1,\dots,d$}
        \For{$j = 1,\dots, r_m$}
            \State Compute the derivative $\partial_{x_i}v_{p,j}$.
        \EndFor
        \State Construct the matrix $\mathbf{W}' := \big(y_{i1} \partial_{x_i}{\bf v}_{p,n}(x_1), \dots, y_{iN}\partial_{x_i}{\bf v}_{p,n}(x_N)\big) \in \mathbb{C}^{r_m \times N}$.
        \State Add $\mathbf{W}'$ to $\mathbf{W}_{m}^{X,Y}$.
    \EndFor
    \State Compute $\mathbf{A}_0 := \mathbf{W}_{m}^{X,Y} (\mathbf{V}_{n}^X)^\dagger$, where $(\cdot)^\dagger$ indicates the Moore--Penrose pseudo inverse.
    \State Extract the the leading principal minor matrix $\widehat{\mathbf{A}} $ of $\mathbf{A}_0$ of order $r_m$:
    \begin{align}
    \mathbf{A}_0 = 
    \mathbf{W}_{m}^{X,Y} (\mathbf{V}_n^X)^\dagger = 
    \begin{pmatrix}
        \widehat{\mathbf{A}} & *
    \end{pmatrix}.
\end{align}
\Ensure 
    $\widehat{\mathbf{A}}$.
\end{algorithmic}
\end{algorithm}


\begin{algorithm}
\caption{A computational framework for estimating eigenvectors of Perron--Frobenius operators and extended Koopman operators for discrete dynamical systems}\label{algorithm: eigenfunction, discrete}
\begin{algorithmic}[1]
    \Require positive integers $m_1,\dots,m_r$ and $n_1,\dots,n_r$ with $m_i \le n_i$ for $i=1,\dots,r$, fixed points $p_1, \dots, p_r\in {\Omega}$ of $f$, $X_i := (x_1^i,\dots, x_{N_i}^i) \in \Omega^{N_i}$ and , $Y_i := (y_1^i , \dots, y_{N_i}^i) \in \Omega^{N_i}$ with $N_i \ge r_{p_i,n_i}$ and $f(x_j^i) = y_j^i$ for $j=1,\dots,N_i$.
    \For{$i=1,\dots, r$}
        \State \label{line2} Execute Algorithm \ref{algorithm: data-driven PF operator estimation, discrete} with inputs $m_i$, $n_i$, $p_i$, $X_i$, and $Y_i$, and assign the output to $\widehat{\mathbf{C}}_i$.
        \State  Compute the eigenvalues $(\gamma_i^{(1)}, \dots, \gamma_i^{(r_{m_i})})$ and eigenvectors $(\hat{\bf w}_i^{(1)}, \dots, \hat{\bf w}_i^{(r_{m_i})})$ of $\widehat{\mathbf{C}}_i$.
        \For{$j=1,\dots,r$}
        \State Compute the matrix $\mathbf{G}^{ij} := \left( \langle v_{p_i,s}, v_{p_j,t} \rangle \right)_{s=1,\dots, r_{m_i}, t=1,\dots, r_{m_j}}$ of size $r_{m_i} \times r_{m_j}$, where $v_{p_i, s}$'s are the elements of the orthogonal basis constructed in Algorithm \ref{algorithm: data-driven PF operator estimation, discrete}.
        \EndFor
        \For{$\ell=1,\dots,r_{m_i}$}
            \State  Define $\widehat{w}_i^{(\ell)}$ with variable $x$ by $\widehat{w}_i^{(\ell)}(x) := \mathbf{v}_{p_i,m_i}(x)^* (\mathbf{G}^{ii})^{-1} \mathbf{w}_i^{(\ell)}$.
        \EndFor
            \State Compute the eigenvector  $(\hat{\bf u}_i^{(1)}, \dots, \hat{\bf u}_i^{(r_{m_i})})$ of $\widehat{\mathbf{C}}_i^*$ corresponding to $(\overline{\gamma_i}^{(1)}, \dots, \overline{\gamma_i}^{(r_m)})$.

    \EndFor
    \State Compute $\mathbf{H}^{ij} \in \mathbb{C}^{r_{m_i} \times r_{m_j}}$ defined as
    \begin{align}
        \begin{pmatrix}
            \mathbf{H}^{11} &\cdots& \mathbf{H}^{1r}\\
            \vdots & \ddots & \vdots\\
            \mathbf{H}^{r1} & \cdots & \mathbf{H}^{rr}
        \end{pmatrix} &:=
        \begin{pmatrix}
            \mathbf{G}^{11} &\cdots& \mathbf{G}^{1r}\\
            \vdots & \ddots & \vdots\\
            \mathbf{G}^{r1} & \cdots & \mathbf{G}^{rr}
        \end{pmatrix}^{-1}
        \begin{pmatrix}
            \mathbf{G}^{11} && \mathbf{O}\\
            & \ddots & \\
            \mathbf{O} & & \mathbf{G}^{rr}
        \end{pmatrix}. \label{computation of H}
    \end{align}
    \For{$i=1,\dots, r$}
        \For{$\ell=1,\dots,r_{m_i}$}
            \State Define $\widehat{u}_i^{(\ell)}$ with variable $x$ by $\widehat{u}_i^{(\ell)}(x) := \sum_{j=1}^r\mathbf{v}_{p_j,m}(x)^* \mathbf{H}_m^{ji} \mathbf{u}_i^{(\ell)}$.
        \EndFor
    \EndFor
    \Ensure $(\gamma_i^{(\ell)}, \widehat{w}_i^{(\ell)})$ as the pairs of the eigenvalue and the eigenvector of the Perron--Frobenius operator, and $(\overline{\gamma_i^{(\ell)}}, \widehat{v}_i^{(\ell)})$ as those of the Koopman operator for $i=1,\dots, r$ and $\ell = 1,\dots, r_{m_i}$.
\end{algorithmic}
\end{algorithm}

\begin{algorithm}
\caption{A computational framework for estimating eigenvectors of Perron--Frobenius operators and extended Koopman operators for continuous dynamical systems}\label{algorithm: eigenfunction, continuous}
\begin{algorithmic}[1]
    \Require positive integers $m_1,\dots,m_r$ and $n_1,\dots,n_r$ with $m_i \le n_i$ for $i=1,\dots,r$, equilibrium points $p_1, \dots, p_r\in {\Omega}$ of $F$, $X_i := (x_1^i,\dots, x_{N_i}^i) \in \Omega^{N_i}$ and , $Y_i := (y_1^i , \dots, y_{N_i}^i) \in \mathbb{R}^{d \times N_i}$ with $N_i \ge r_{p_i,n_i}$ and $F(x_j^i) = y_j^i$ for $j=1,\dots,N_i$.
    \Ensure the outputs of Algorithm \ref{algorithm: eigenfunction, discrete} using Algorithm \ref{algorithm: data-driven PF operator estimation, continuous} instead of Algorithm \ref{algorithm: data-driven PF operator estimation, discrete} in Line \ref{line2}.
\end{algorithmic}
\end{algorithm}

\subsubsection{Eigenvalues for Van der Pol oscillators} \label{subsubsec: vdP eigenvalues}
Let $\mu>0$ be a positive number, and we consider the van der Pol oscillator:
\begin{align}
    \begin{split}
        x' &= y, \\
        y' &= \mu(1-x^2)y - x.
    \end{split}\label{vdP}
\end{align}
This dynamical system has one equilibrium point at the origin and a single limit cycle.
The eigenvalues of the Jacobian matrix of the vector field are $(\mu \pm \sqrt{\mu^2-4})/2$.

Figure \ref{fig: vdP_eigenvalues_exp} describes the estimation of eigenvalues of the Perron--Frobenius operators $A_F^*|_{V_{p,m}}$ of the van der Pol oscillator \eqref{vdP} for $\mu=1$ using Algorithm \ref{algorithm: data-driven PF operator estimation, continuous} and one with Algorithm \ref{algorithm: data-driven PF operator estimation, discrete} with the method as in Remark \ref{rmk: matrix log}, using the exponential kernel $k^{\rm e}(x,y) = e^{(x-b)^\top (y-b)/\sigma^2}$ with $\sigma=2$ and $b=0$.
On the left panel, we use $m=5$, $n=7$, $N=36$ samples from the uniform distribution on $[-1,1]^2$, and the exact velocities at the samples to compute $\widehat{\mathbf{A}}$.
On the middle and right panels, we use the method as in Remark \ref{rmk: matrix log} with the matrix logarithm.
On the middle panels, we take $T_s = 0.5$, and it provides fairly good approximation of the Perron--Frobenius operator.
Moreover, it seems to work well although it has not been proved that $\phi^{T_s}$ satisfies Assumption \ref{asm: domain of C_f}.
The reason is either that $\phi^{T_s}$ meets Assumption \ref{asm: domain of C_f} or that Assumption \ref{asm: domain of C_f} is too strong.
On the other hand, the estimation fails if $T_s$ is relatively large as in the right panel. 
The difference of the two green circles in the right panel is approximately $2\pi$, indicating the algorithm of the numerical computation of the matrix logarithm chooses the principal logarithm \cite{MR2396439}, that is an incorrect branch for our purpose.
We also note that the method in Remark \ref{rmk: matrix log} needs larger $n$ and $N$.  It reflects that the the flow map is much more complicated than its original vector fields.

\begin{figure}
\begin{minipage}[c]{1.0\linewidth}
    \begin{minipage}[c]{0.33\linewidth}
      \includegraphics[keepaspectratio, width=\linewidth]{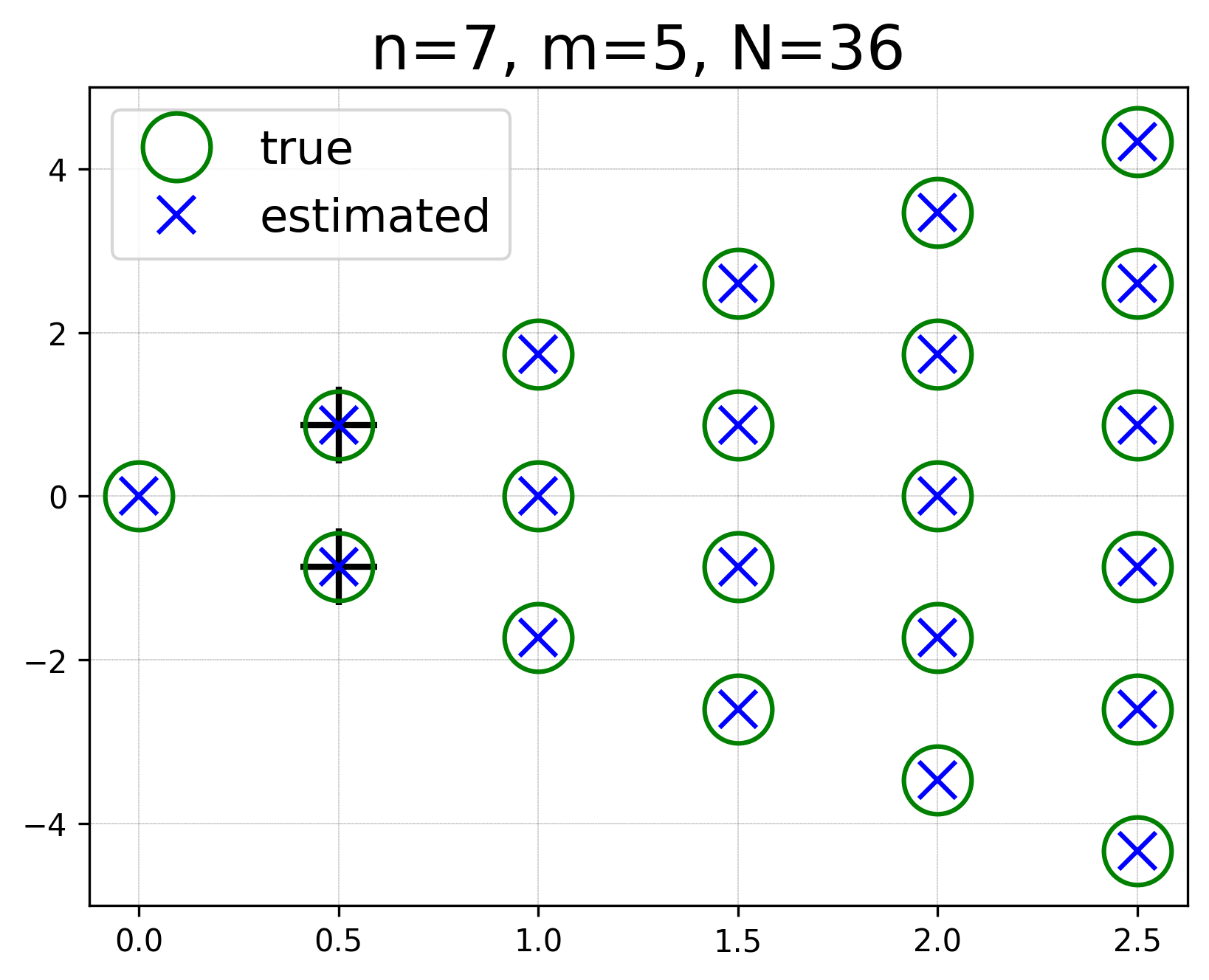}
    \end{minipage}\hfill 
    \begin{minipage}[c]{0.33\linewidth}
      \includegraphics[keepaspectratio, width=\linewidth]{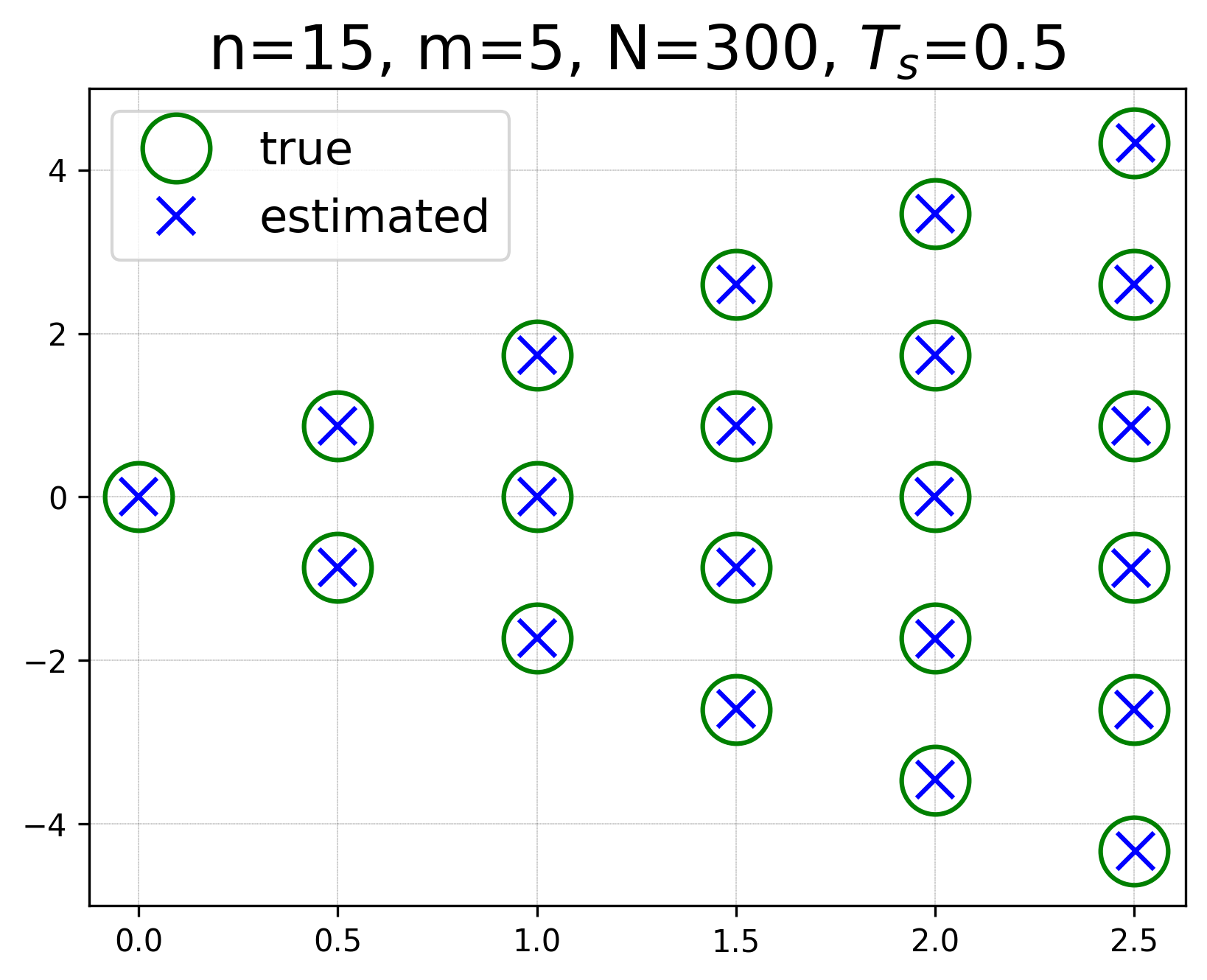}
    \end{minipage}\hfill 
    \begin{minipage}[c]{0.33\linewidth}
      \includegraphics[keepaspectratio, width=\linewidth]{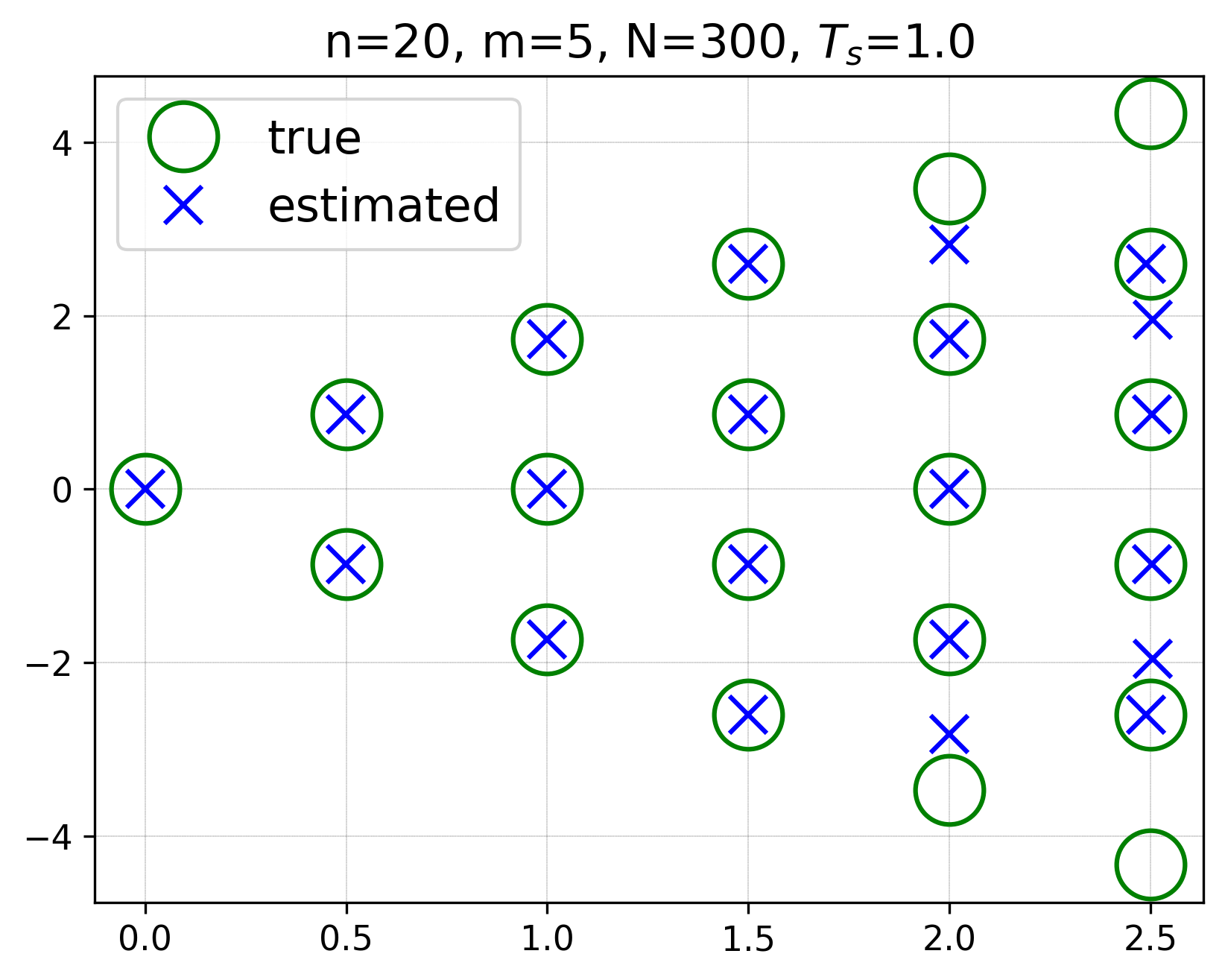}
    \end{minipage}
\end{minipage}
\caption{
Eigenvalue estimation for the van der Pol oscillator ($\mu=1$) using Algorithm \ref{algorithm: data-driven PF operator estimation, continuous} (left) and Algorithm \ref{algorithm: data-driven PF operator estimation, discrete} with the matrix-log method in Remark \ref{rmk: matrix log} (middle, right). We use the exponential kernel $k^{\rm e}(x,y)=e^{x^\top y/4}$. Blue $\times$ marks show the eigenvalues of the estimated Perron--Frobenius operator $\widehat{\mathbf{A}}$; green circles show those of $A_F|{V{p,m}}$; plus signs show the eigenvalues of the Jacobian of $F(x,y)=(y,,\mu(1-x^2)y-x)$ at $p=(0,0)$. {\bf Left:} continuous setting with $m=5$, $n=7$, $p=(0,0)$, $N=36$ samples drawn uniformly from $[-1,1]^2$, and their exact velocities. {\bf Middle:} discrete setting with $m=5$, $n=15$, $T_s=0.5$, $p=(0,0)$, and $N=300$ input-output pairs from $[-1,1]^2$ and their images under $\phi^{T_s}$.  
{\bf Right:} discrete setting with $m=5$, $n=20$, $T_s=1.0$, $p=(0,0)$, and $N=300$ input-output pairs as above.
}
\label{fig: vdP_eigenvalues_exp}
\end{figure}

\subsubsection{Eigenfunctions for Duffing oscillators}
Here, we consider the Duffing oscillator (without external force):
\begin{align}
    \begin{split}
        x' &= y, \\
        y' &= -\delta y - \alpha x - \beta x^3
    \end{split}\label{duffing}
\end{align}
only in the case of $(\alpha, \beta, \delta) = (-1,1,0.5)$.
This dynamical system has two stable equilibrium points, $(-1,0)$, $(1,0)$, and an unstable one $(0,0)$.
There exist two domains of attraction and trajectories that start from the points within the same colored region all converge to the same equilibrium points (left panel in Figure \ref{fig: duffing_eigenfunctions_two_equilibrium_points}). 

Figure \ref{fig: duffing_eigenfunctions_two_equilibrium_points} describes the approximated eigenfunctions for $-1$ using the exponential kernel $e^{x^\top y/\sigma^2}$ with $\sigma = 1$.
We take two equilibrium points, $p_1=(-1,0)$ and $p_2=(1,0)$.
Then, we set $m_1=m_2=10$, $n_1=n_2=16$, and take $N_1=N_2=7000$ samples from the uniform distribution on $[-1.5,1.5] \times [-0.5,0.5]$, and the exact velocities on them as input.
Here, we draw graphs of even and odd eigenfunctions constructed as in the following procedure:
first, we note that the linear operator $C_{-1}: H \to H; h \mapsto h((-1)\times\cdot)$ induces a Hermitian unitary operator.
Moreover, $C_{-1}$ is commutative with $A_F^*$ and satisfies $C_{-1}(V_{p_1,n}) = V_{p_2,n}$.
Thus, for the eigenfunction $v$ of the extended Koopman operator $A_F^\times$ in $V_{p_1,n}$, the image $C_{-1}v$ corresponds to an eigenvector for the same eigenvalue in $V_{p_2,n}$.
Therefore, we canonically construct even and odd eigenfunctions $v \pm C_{-1}v$ of the extended Koopman operator $A_F^\times$.
The right panel is the heat map of the odd eigenfunctions.
We observe that the eigenfunction seems to capture the domain of attractions.
\begin{figure}
\begin{minipage}[c]{1.0\linewidth}
    \begin{minipage}[c]{0.38\linewidth}
      \includegraphics[keepaspectratio, width=\linewidth]{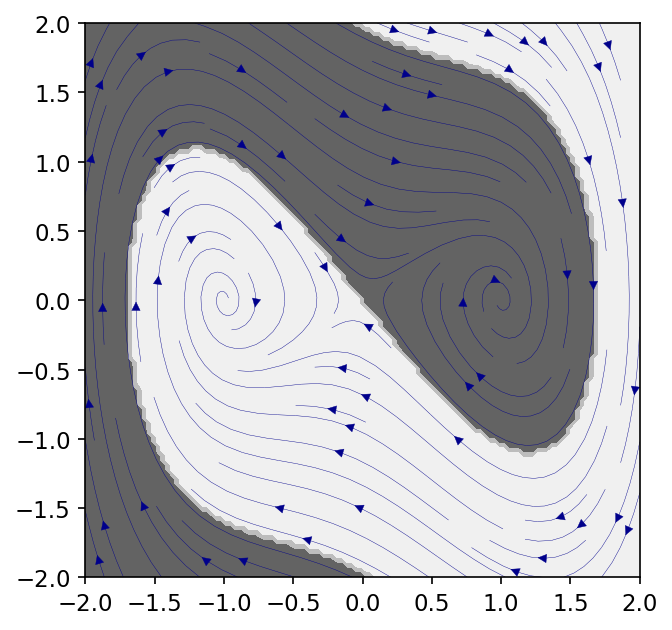}
    \end{minipage}\hfill 
    \begin{minipage}[c]{0.45\linewidth}
      \includegraphics[keepaspectratio, width=\linewidth]{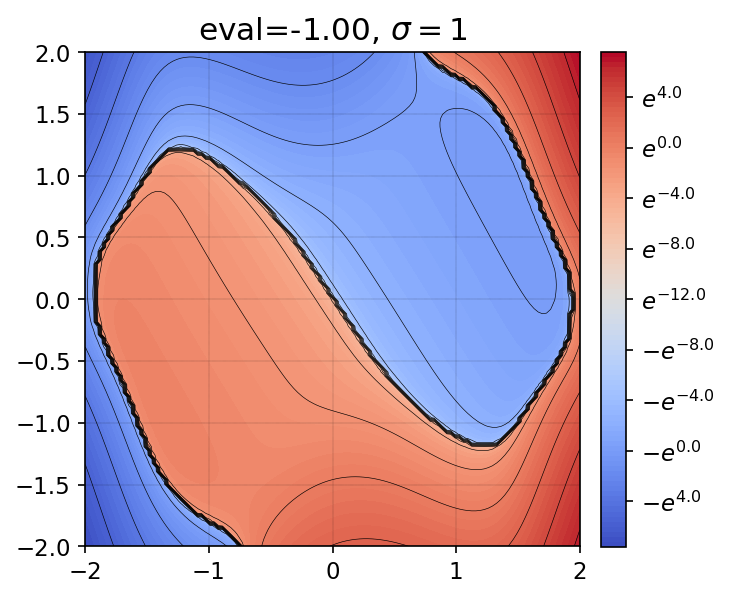}
    \end{minipage}
\end{minipage}
\caption{
Two domains of attraction of the Duffing oscillator \eqref{duffing} (left).
{\bf Left:} two domains of attraction of the Duffing oscillator \eqref{duffing}. {\bf Right:} an odd estimated eigenfunction for eigenvalue $-1$ of the extended Koopman operator $A_F^\times$ for the Duffing oscillator \eqref{duffing}, computed via Algorithm \ref{algorithm: eigenfunction, continuous} using the exponential kernel $e^{x^\top y}$ with input $m_1=m_2=10$, $n_1=n_2=16$, $p_1=(-1,0)$, $p_2=(1,0)$, and $N_1=N_2=7000$ samples drawn uniformly from $[-1.5,1.5]\times[-0.5,0.5]$, together with their exact velocities.
}
\label{fig: duffing_eigenfunctions_two_equilibrium_points}
\end{figure}


\subsubsection{H\'enon maps}
Here, we consider the H\'enon map, the discrete dynamical system $f: \mathbb{R}^2 \to \mathbb{R}^2$, defined by
\begin{align}
    f(x,y) = (y + 1 - {a} x^2, {b} x) \label{henon}
\end{align}
in the case of $a=1.4$ and $b=0.3$.
This dynamical system has two fixed points 
\begin{align}
    p = \left(\frac{{b}-1 \pm \sqrt{({b}-1)^2 + 4{a}}}{2{a}}, \frac{{b}({b}-1)\pm {b}\sqrt{({b}-1)^2 + 4{a}}}{2{a}} \right).
\end{align}
We describe some characteristic features on the left panel in Figure~\ref{fig: henon_eigenfunctions}.

The middle and right pictures in Figure \ref{fig: henon_eigenfunctions} describes the approximated eigenfunctions using the exponential kernel $e^{(x-b)^\top (y-b)/\sigma^2}$ with $\sigma = 0.6$ and $b=0$ with input $m_1=6$, $n_1=30$.
We take $N_1=3000$ pairs of samples from the uniform distribution from $p_1 + [-0.25,0.25]^2$ and their images under the dynamical system $f$, where $p_1 \approx (-1.13135448, -0.33940634)$ is a fixed point.

Basically, the vanishing region of the eigenfunctions captures the several characteristic features of the dynamical system, for example, some parts of invariant manifold and the attractor.
Our framework shows that the estimated eigenfunctions actually approximate those of $C_f^\times$.
The study of the mathematical properties of these eigenfunctions of $C_f^\times$ is a crucial research task for the future.

\begin{figure}
\begin{minipage}[c]{1.0\linewidth}
    \begin{minipage}[c]{0.33\linewidth}
      \includegraphics[keepaspectratio, width=\linewidth]{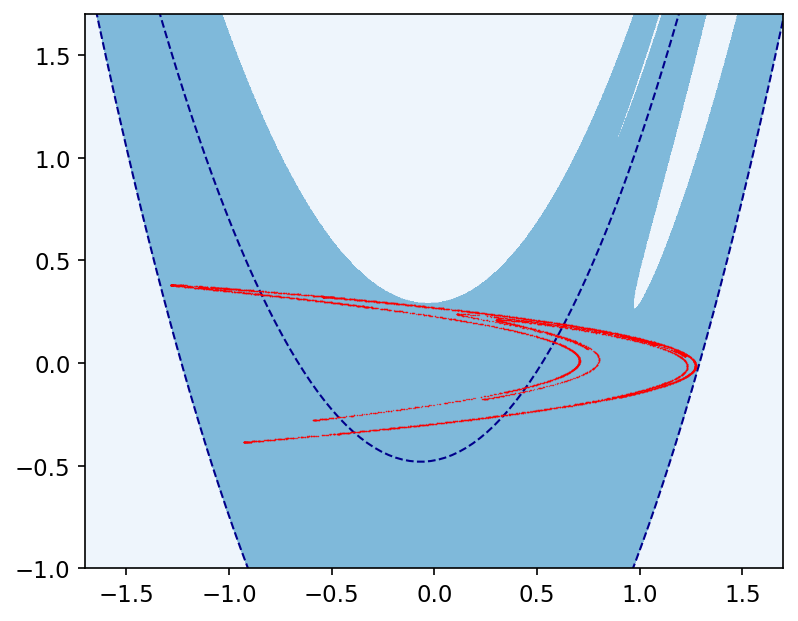}
    \end{minipage}\hfill 
    \begin{minipage}[c]{0.34\linewidth}
      \includegraphics[keepaspectratio, width=\linewidth]{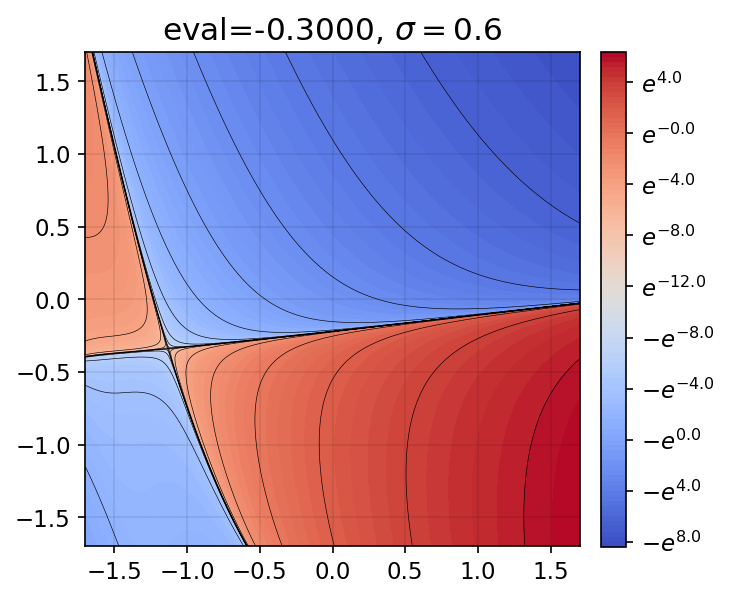}
    \end{minipage}\hfill 
    \begin{minipage}[c]{0.28\linewidth}
      \includegraphics[keepaspectratio, width=\linewidth]{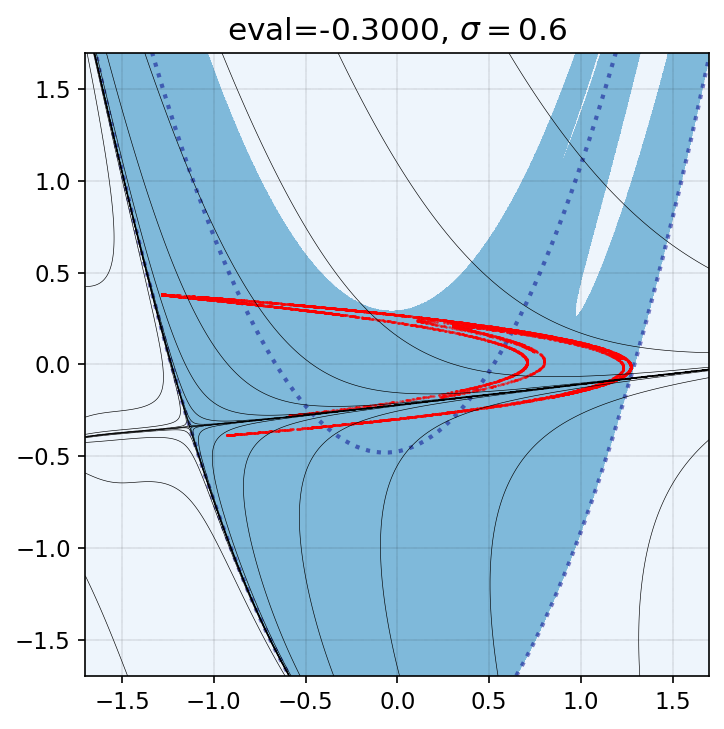}
    \end{minipage}
\end{minipage}
\caption{
{\bf Left:} the attractor (red scatter), the domain of attraction (blue region), and the invariant manifold (dark blue dashed line) of the fixed points of the H'enon map \eqref{henon}. {\bf Middle:} the heat map of the estimated eigenfunction with eigenvalue $0.3$ of the extended Koopman operator $C_f^\times$ for the H\'enon map \eqref{henon}, computed via Algorithm \ref{algorithm: eigenfunction, discrete} using the exponential kernel $e^{x^\top y/0.6}$ with input $m_1=6$, $n_1=30$, $p_1\approx(-1.131,-0.339)$, and $N_1=3000$ input--output pairs sampled uniformly from $p_1+[-0.25,0.25]^2$ and their images under $f$. {\bf Right:} the contours of the eigenfunction overlaid on the left panel.
}
\label{fig: henon_eigenfunctions} 
\end{figure}

\subsection{Data-driven reconstruction of dynamical systems}\label{sec: system identification}
As an application of JetEDMD, we propose a method to reconstruct the original dynamical system from finite discrete data on trajectories of the dynamical system.
We show the performance of the reconstruction using the Lorenz system.

We describe the data-driven reconstruction algorithms in Algorithms \ref{algorithm: reconstruction, discrete}, \ref{algorithm: reconstruction, continuous}, and \ref{algorithm: reconstruction, conti. with disc.}.
Moreover, combining Theorems \ref{thm: reconstruction, discrete} and \ref{thm: convergence of algorithm, discrete}, we have accurate convergence results for these algorithms as follows:
\begin{theorem}[Convergence of Algorithm \ref{algorithm: reconstruction, discrete}]\label{convergence of algorithm, reconstruction, discrete}
    Let $\sigma>0$.
    Let $k^{\rm e}(x,y) = e^{-(x-p)^\top(y-p)/\sigma^2}$.
    Let $x_1,\dots, x_N$ be i.i.d random variables of the distribution with compactly supported density function $\rho$ such that ${\rm ess.inf}_{x \in U}\rho(x) >0$ for some open subset $U \subset \mathbb{R}^d$.
    Let $x_0, y_0 \in \mathbb{R}^d$ such that $y_0 = f(x_0)$. Assume that $f^\alpha \in H^{\rm e}(\sigma, p)$ for all $\alpha \in \mathbb{Z}_{\ge 0}^d$.
    Define $\widehat{f}_{m,n,N}$ as the output of Algorithm \ref{algorithm: reconstruction, discrete} with input $k=k^{\rm e}$, $\sigma$, $m$, $n$, $(x_0, y_0)$, $X=(x_1,\dots, x_N)$, and $Y=(y_1,\dots, y_N)$.
    Then, for any compact set $K \subset \mathbb{R}^d$, we have
    \begin{align*}
            \lim_{m\to \infty}\lim_{n\to \infty}\lim_{N\to\infty}
            \sup_{y \in K}\left\|f(y)-\widehat{f}_{m,n,N}(y) \right\|
            = 0~~~{\rm a.e.}
    \end{align*}
\end{theorem}
Combining  Theorem \ref{thm: convergence of algorithm, continuous} with Theorems \ref{thm: reconstruction continuous, exp} and \ref{thm: reconstruction continuous, gaussian}, we have
\begin{theorem}[Convergence of Algorithm \ref{algorithm: reconstruction, continuous}]\label{convergence of algorithm, reconstruction, continuous}
    Let $\sigma > 0$.
    Let $k$ be the exponential kernel or the Gaussian kernel.
    Let $x_1,\dots, x_N$ are i.i.d random variables of the distribution with compactly supported density function $\rho$ such that ${\rm ess.inf}_{x \in U}\rho(x) >0$ for some open subset $U \subset \mathbb{R}^d$.
    Let $x_0, y_0 \in \mathbb{R}^d$ such that $y_0 = F(x_0)$ and assume that $x_0$ satisfies Assumptions \ref{asm: existence of an equilibrium point} and \ref{asm: domain of A_F} (see Propositions \ref{prop: sufficient condition for domain condition, continuous, exp} and \ref{prop: sufficient condition for domain condition, continuous, gauss}).
    Define $\widehat{F}_{m,n,N}$ as the output of Algorithm \ref{algorithm: reconstruction, continuous} with input $k$, $\sigma$, $m$, $n$, $(x_0, y_0)$, $X=(x_1,\dots, x_N)$, and $Y=(y_1,\dots, y_N)$.
    Then, for any compact set $K \subset \mathbb{R}^d$, we have
    \begin{align*}
           \lim_{m\to \infty}\lim_{n\to \infty}\lim_{N\to\infty}
            \sup_{y \in K}\left\|F(y)-\widehat{F}_{m,n,N}(y) \right\|  = 0~~~{\rm a.e.}
    \end{align*}
\end{theorem}
\begin{theorem}[Convergence of Algorithm \ref{algorithm: reconstruction, conti. with disc.}]\label{convergence of algorithm, reconstruction, discrete continuous}
    Let $\sigma>0$.
    Let $T_s>0$ and $f := \phi^{T_s}$.
    Let $x_1,\dots, x_N$ are i.i.d random variables of the distribution with compactly supported density function $\rho$ such that ${\rm ess.inf}_{x \in U}\rho(x) >0$ for some open subset $U \subset \mathbb{R}^d$.
    Let $x_0, y_0 \in \mathbb{R}^d$ such that $y_0 = f(x_0)$ and let $k^{\rm e}(x,y) = e^{-(x-x_0)^\top(y-x_0)/\sigma^2}$.
    Assume $f^\alpha \in H^{\rm e}(\sigma, p)$ for all $\alpha \in \mathbb{Z}_{\ge 0}^d$.
    Define $\widehat{F}_{m,n,N}$ as the output of Algorithm \ref{algorithm: reconstruction, conti. with disc.} with input $k=k^{\rm e}$, $\sigma$, $m$, $n$, $(x_0, y_0)$, $X=(x_1,\dots, x_N)$, and $Y=(y_1,\dots, y_N)$.
    Then, for any compact set $K \subset \mathbb{R}^d$, we have
    \begin{align*}
           \lim_{m\to \infty}\lim_{n\to \infty}\lim_{N\to\infty}
            \sup_{y \in K}\left\|F(y)-\widehat{F}_{m,n,N}(y) \right\| = 0~~~{\rm a.e.}
    \end{align*}
\end{theorem}

\begin{remark}
Theorem \ref{convergence of algorithm, reconstruction, discrete continuous} guarantees that, under Assumption \ref{asm: domain of C_f}, namely, analytically favorable properties of the dynamical system, the original vector field can be reconstructed from finite data even if the sampling period remains long.
This significantly improves \cite[Theorem 1]{MauroyGoncalves2020}.
\end{remark}

\begin{algorithm}
\caption{
A computational framework for data-driven reconstruction of discrete dynamical systems
}\label{algorithm: reconstruction, discrete}
\begin{algorithmic}[1]
    \Require
    a positive number $\sigma > 0$,
    positive integers $m$ and $n$ with $m \le n$,
    a pair of points $(x_0, y_0) \in \Omega^2$ with $y_0 = f(x_0)$,
    and $X := (x_1,\dots, x_N), Y := (y_1 , \dots, y_N) \in \Omega^N$ such that $y_i = f(x_i)$ with $N \ge r_n$.
    \State Set $\widetilde{Y} := (y_1 - y_0 + x_0, \dots, y_N - y_0 + x_0) \in \Omega^N$.
    \State Compute $\widehat{\mathbf{C}}$ by Algorithm \ref{algorithm: data-driven PF operator estimation, discrete} using positive definite kernel $k(x,y) = e^{(x-x_0)^\top (y-x_0)/\sigma^2}$ with inputs $m$, $n$, $p=x_0$, $X$, and $\widetilde{Y}$  (c.f. Section \ref{subsec: explicit intinsic observables}).
    \For{$i=1,\dots, d$}
        \State Define $\widehat{f}_i$ with variable $x$ by the $(i+1)$-th component of
        $\sigma \widehat{\mathbf{C}}\mathbf{v}_{x_0,m}^{\rm e}(x) + y_0$.
    \EndFor
    \Ensure $\widehat{f} := (\widehat{f}_1, \dots, \widehat{f}_d)$.
\end{algorithmic}
\end{algorithm}

\begin{algorithm}
\caption{
A computational framework for data-driven reconstruction of continuous dynamical systems
}\label{algorithm: reconstruction, continuous}
\begin{algorithmic}[1]
    \Require 
    a positive number $\sigma > 0$,
    positive integers $m$ and $n$ with $m \le n$,
    a pair of points $(x_0, y_0) \in \Omega \times \mathbb{R}^d$ with $y_0 = F(x_0)$,
    and $X := (x_1,\dots, x_N) \in \Omega$ and $Y := (y_1 , \dots, y_N) \in \mathbb{R}^{d \times N}$ with $y_i=F(x_i)$.
    \State Set $\widetilde{Y} := (y_1 - y_0, \dots, y_N - y_0) \in \Omega^N$.
    \State Compute $\widehat{\mathbf{A}}$ by Algorithm \ref{algorithm: data-driven PF operator estimation, continuous} using either of the positive definite kernels $k(x,y) = e^{(x-x_0)^\top (y-x_0)/\sigma^2}$ or $e^{-\|x-y\|^2/2\sigma^2}$ with input $m$, $n$, $p=x_0$, $X$, and $Y$ (c.f. Section \ref{subsec: explicit intinsic observables}).
    \For{$i=1,\dots, d$} 
        \If {$k(x,y) = e^{(x-x_0)^\top (y-x_0)/\sigma^2}$}
        \State Define $\widehat{F}_i$ with variable $x$ as the $(i+1)$-th component of  
        $\sigma \widehat{\mathbf{A}}\mathbf{v}_{x_0,m}^{\rm e}(x) + y_0$  
        \ElsIf{$k(x,y) = e^{-\|x-y\|^2/2\sigma^2}$}
        \State Define $\widehat{F}_i$ with variable $x$ as $(\partial_{x_i}\mathbf{v}_{x_0,m}^{\rm g}(x))^\top \sigma^2 (\mathbf{G}_m^{\rm g})^{-1}\widehat{\mathbf{A}}\mathbf{v}_{x_0,m}^{\rm g}(x) + y_0$.
        \EndIf
    \EndFor
    \Ensure 
        $\widehat{F} := (\widehat{F}_1,\dots, \widehat{F}_d)$.
\end{algorithmic}

\end{algorithm}

\begin{algorithm}
\caption{A computational framework for data-driven reconstruction of continuous dynamical systems with knowledge of the location of equilibrium points}
\label{algorithm: reconstruction, conti. with disc.}
\begin{algorithmic}[1]
    \State  
    a positive number $\sigma > 0$,
    positive integers $m$ and $n$ with $m \le n$,
    an equilibrium point $p\in \Omega$ of $F$,
    and $X := (x_1,\dots, x_N), Y := (y_1 , \dots, y_N) \in \Omega^N$ such that $y_i=\phi^{T_s}(x_i)$ for a fixed $T_s > 0$ for $i=1,\dots N$.
    \State Compute $\widehat{\mathbf{C}}$ by Algorithm \ref{algorithm: data-driven PF operator estimation, discrete} using $k(x,y) := e^{(x-p)^\top (y-p)/\sigma^2}$ with input $m$, $n$, $p$, $X$, and $Y$ (c.f. Section \ref{subsec: explicit intinsic observables}).
    \State Choosing an appropriate brunch of matrix logarithm, compute $\widehat{\mathbf{A}}$ by $\widehat{\mathbf{A}} := \frac{1}{T_s}\log \widehat{\mathbf{C}}$.
    \For{$i=1,\dots, d$} 
      \State Define $\widehat{F}_i$ with variable $x$ by the $(i+1)$-th component of $\sigma (\mathbf{G}_m^{\rm e})^{-1}\widehat{\mathbf{A}}\mathbf{v}_{p,m}^{\rm e}(x) + y_0$.
    \EndFor
    \Ensure $\widehat{F} := (\widehat{F}_1,\dots, \widehat{F}_d)$.
\end{algorithmic}
\end{algorithm}

Now, we consider the Lorenz equation: 
\begin{align}
    \begin{split}
        x' &= 10(y-x), \\
        y' &=x(28-z) - y, \\
        z' &=xy - \frac{8}{3}z.
    \end{split}\label{lorenz}
\end{align}

Figure \ref{fig: reconstruction_lorenz} describes the performance of Algorithm \ref{algorithm: reconstruction, continuous} for the Lorenz system \eqref{lorenz} using different kernels; the exponential kernel and the Gaussian kernel.
We take $m=2$, $n=4$, and $M=50$ samples $x_1^{(0)},\dots, x_M^{(0)}$ from the uniform distribution on $[-10,10]^3$, and define
\[x_i^{(j)} := \phi^{0.01j}(x_i^{(0)}),\]
for $i=1,2,\dots,M$ and $j=1,\dots,9$.
Then, using the finite difference method \cite{d8dce9b2-3375-3c07-b2d5-4fe136d49ddb} of order $10$, there exist rational numbers $c_k^{ij}$ ($k=1,\dots,9)$ such that
\begin{align*}
    y_i^{(j)} &:= \frac{1}{0.01}\sum_{k=0}^9 c_k^{ij}x_i^{(j)}
\end{align*}
provides suitable approximations of the velocities $F(x_i^{(j)})$.
Then, we set $X := (x_i^{(j)})_{i=1,\dots,M, j=0,\dots,9}$ and $Y := (y_i^{(j)})_{i=1,\dots,M, j=0,\dots,9}$ (thus $N=500$).
We set $x_0$ as an element of $\{x_i^{(j)}\}_{i=1,\dots,M, j=0,\dots,9}$ closest to the mean of $\{x_i^{(j)}\}_{i=1,\dots,M, j=0,\dots,9}$, and $y_0$ as an element of $\{y_i^{(j)}\}_{i=1,\dots,M, j=0,\dots,9}$ corresponding to $x_0$.
On the left panel, we use the exponential kernel with $\sigma=30$, while on the right panel, we use the Gaussian kernel with $\sigma=300$.
In the Gaussian case, the output $\widehat{F}_{m,n,N}$ decays as fast as the Gaussian function, and we need to take a relatively large $\sigma$ depending on the region we need to approximate the target.

\begin{figure}
\begin{minipage}[c]{1.0\linewidth}
    \begin{center}
    \begin{minipage}[c]{0.45\linewidth}
      \includegraphics[keepaspectratio, width=\linewidth]{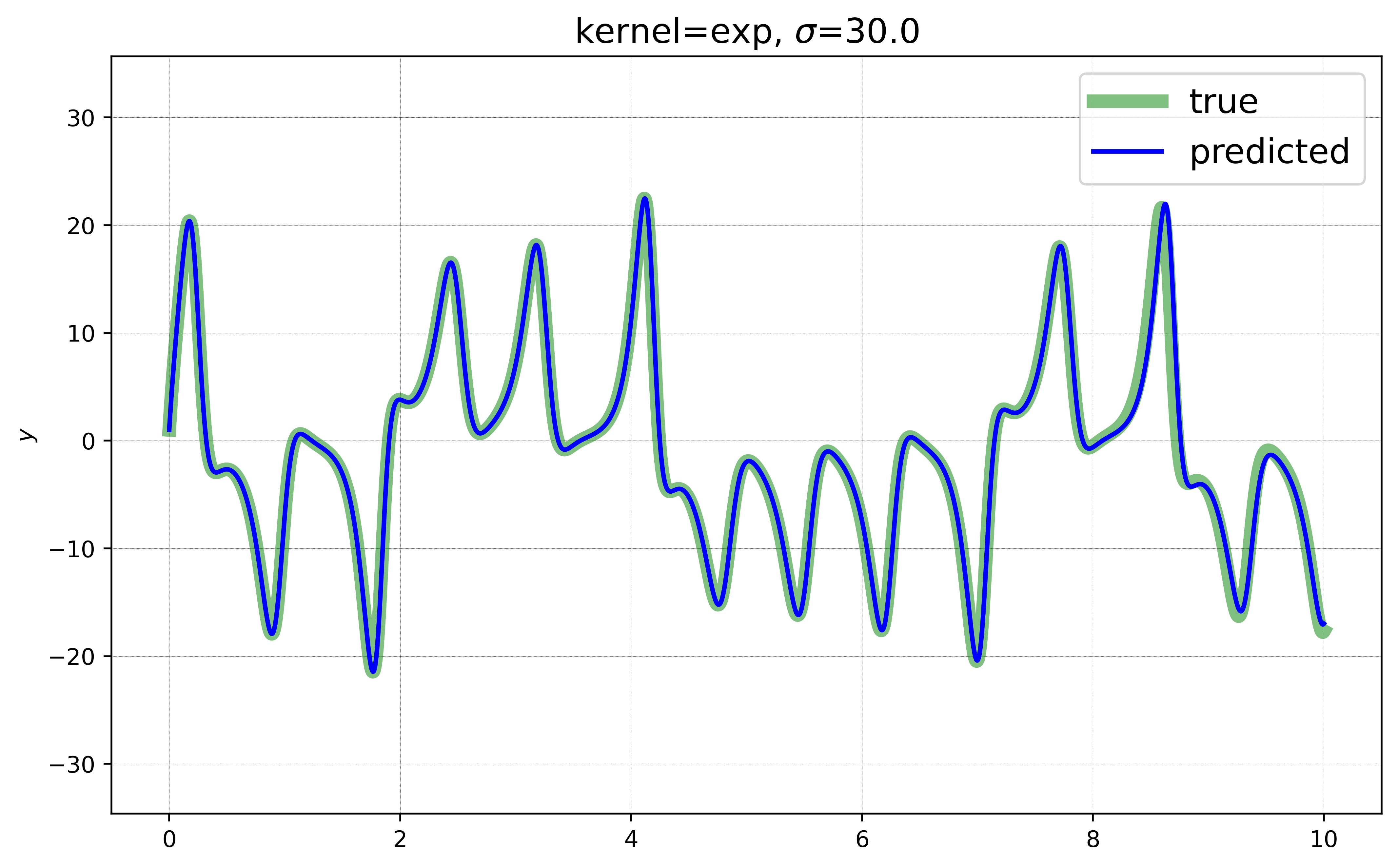}
    \end{minipage}
    \begin{minipage}[c]{0.45\linewidth}
      \includegraphics[keepaspectratio, width=\linewidth]{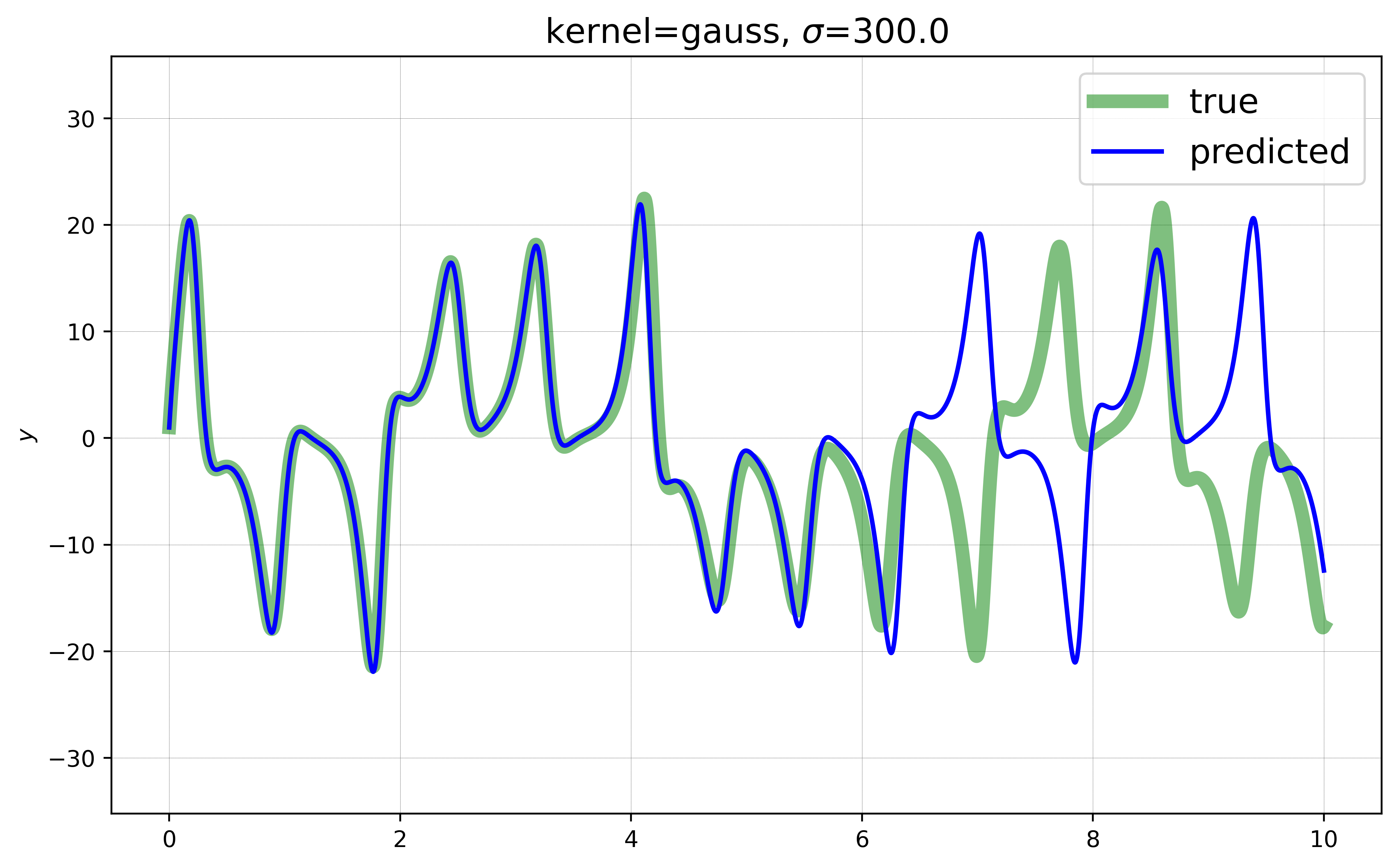}
    \end{minipage}
    \end{center}
\end{minipage}
\caption{
Data-driven reconstruction of the Lorenz attractor \eqref{lorenz} using Algorithm \ref{algorithm: reconstruction, continuous} with the exponential kernel $(\sigma=30)$ (left) and the Gaussian kernel $(\sigma=300)$ (right), with input $(m,n)=(2,4)$. The data consist of $50$ collections of $10$ snapshots at times $0, 0.01, \dots, 0.09$ of trajectories with initial points sampled from the uniform distribution on $[-20,20]^3$, together with the velocities at these points, computed using a $10$th-order finite-difference method. {\bf Left:} results with the exponential kernel $(\sigma=30)$; the blue thin curves show the $y$-coordinate of the predicted trajectories, and the green thick curves show those of the true trajectories. {\bf Right:} results with the Gaussian kernel $(\sigma=300)$; the same plotting convention is used. The initial point for each curve is $(10,1,10)\in\mathbb{R}^3$.
}
\label{fig: reconstruction_lorenz}
\end{figure}

Figure \ref{fig: reconstruction_lorenz, discrete} describes the performance of Algorithm \ref{algorithm: reconstruction, conti. with disc.} for the Lorenz system \eqref{lorenz} for different sampling periods $T_s$.
We set $p=(0,0,0)$ as the equilibrium point.
On the left panel, we set $\sigma=80$, $m=3$, $n=12$, and $T_s=0.033$, and take $20$ collections $\{(x_i^{(0)}, \dots, x_i^{(15)})\}_{i=1}^{20}$ of snapshots at times $0, T_s, \dots, 15T_s$ of trajectories with initial points sampled from the uniform distribution on $[-20,20]^3$, namely $x_i^{(j)} = \phi^{jT_s}(x_i^{(0)})$ holds.
We set $X=(x_1^{(0)},\dots,x_{20}^{(0)}, x_1^{(1)}, \dots, x_{20}^{(14)})$ and $Y=(x_1^{(1)},\dots,x_{20}^{(1)}, x_1^{(2)}, \dots, x_{20}^{(15)})$.
On the right panel, we set $T_s=0.06$. 
Even in a highly complex and chaotic system like the Lorenz system, we see that the original system can be reconstructed from data with a small number of samples with relatively long sampling periods.

\begin{remark}
    On the left panel of Figure \ref{fig: reconstruction_lorenz, discrete}, we use the same data as in \cite{MauroyGoncalves2020}, and our method also has the same or better performance than the lifting method proposed in \cite{MauroyGoncalves2020}.
    Our approach approximates the operators using the space of the intrinsic observables in an RKHS defined via the jets, independent of data, and this constitutes an essential difference from their lifting method.
\end{remark}

\begin{figure}
\begin{minipage}[c]{1.0\linewidth}
    \begin{center}
    \begin{minipage}[c]{0.45\linewidth}
      \includegraphics[keepaspectratio, width=\linewidth]{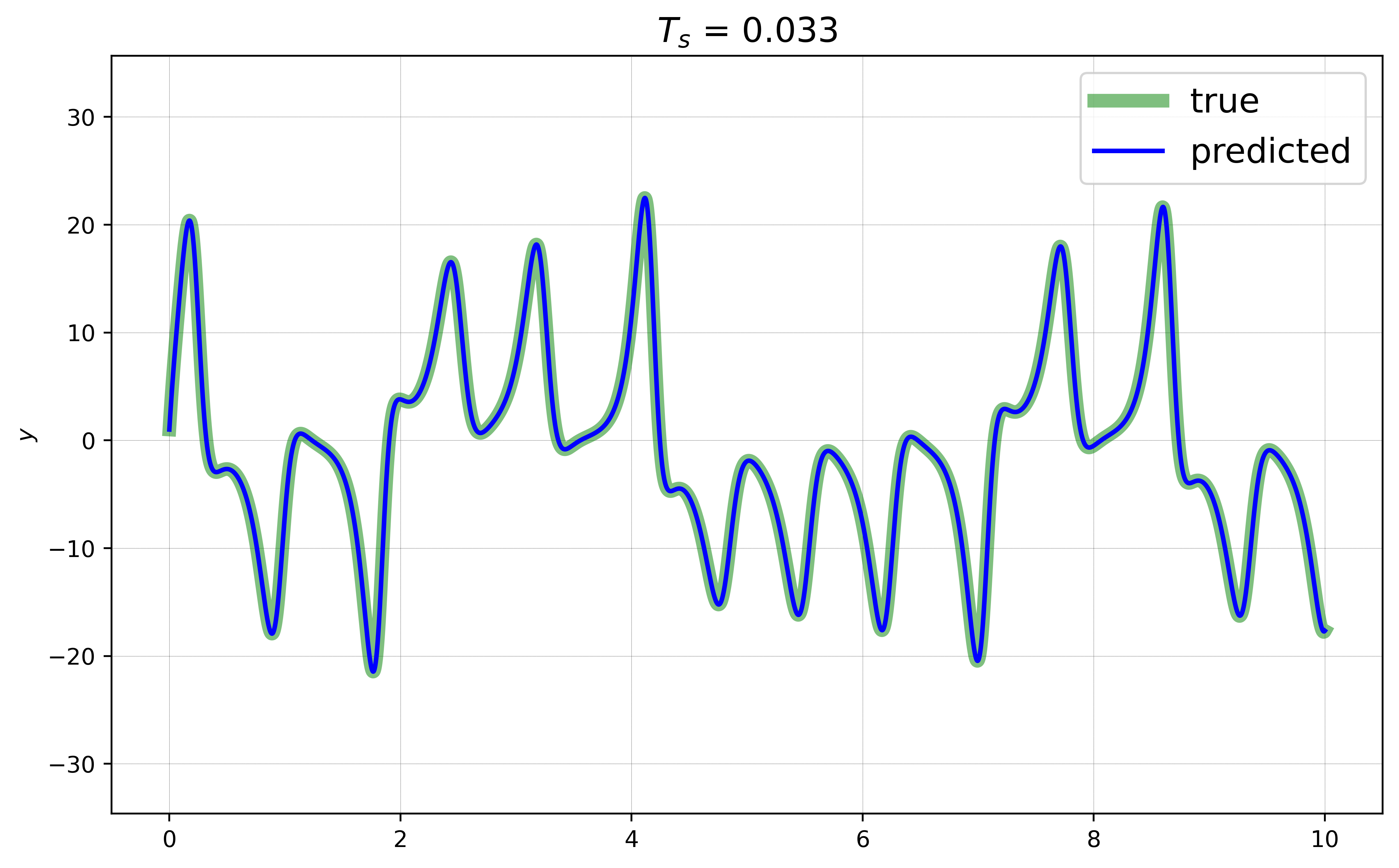}
    \end{minipage}
    \begin{minipage}[c]{0.45\linewidth}
      \includegraphics[keepaspectratio, width=\linewidth]{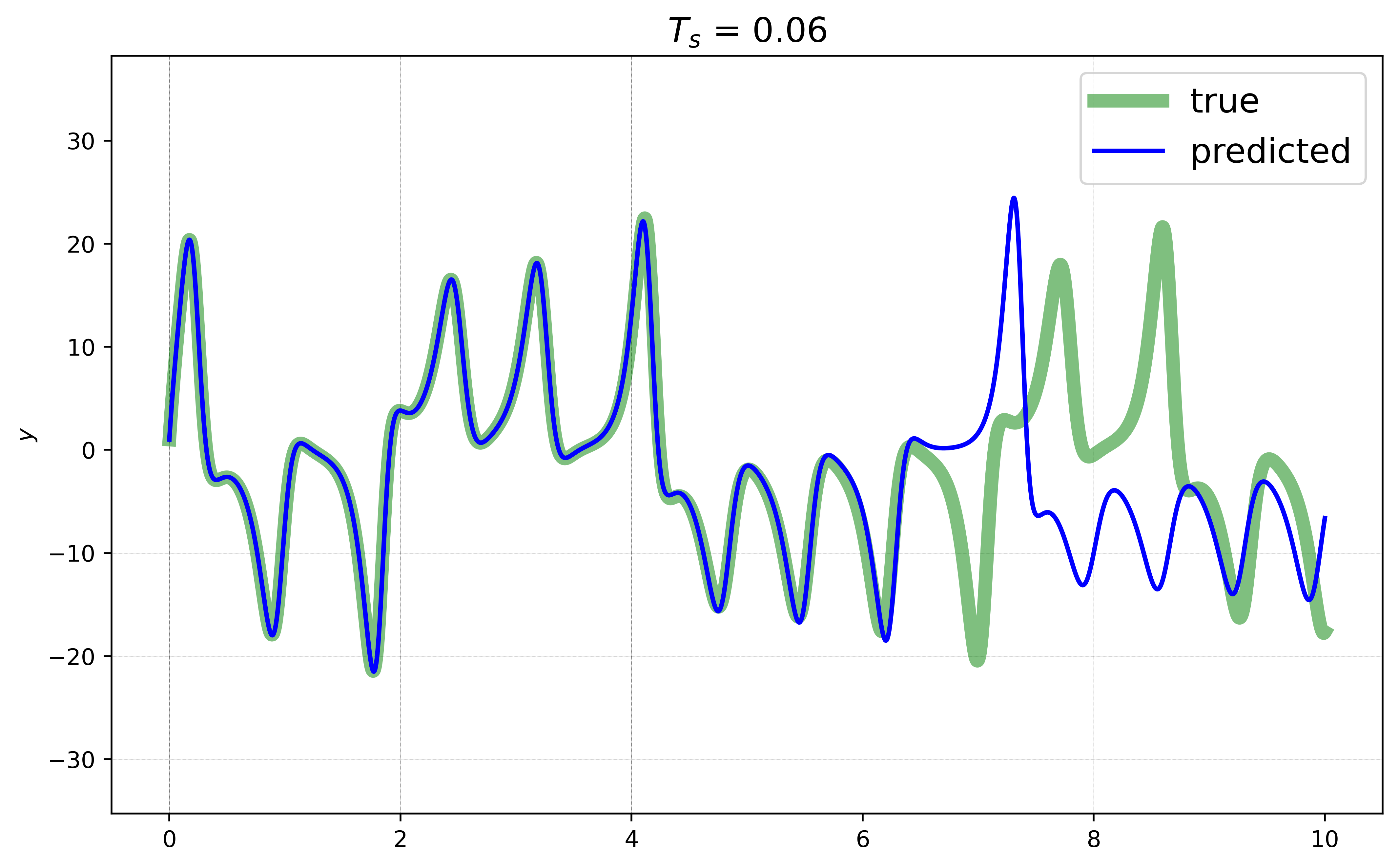}
    \end{minipage}
    \end{center}
\end{minipage}
\caption{
Data-driven reconstruction of the Lorenz attractor \eqref{lorenz} using Algorithm \ref{algorithm: reconstruction, conti. with disc.} with the exponential kernel $(\sigma=80)$ and input $(m,n)=(3,12)$. As data, we use $20$ collections of snapshots at times $0, T_s, \dots, 15T_s$ of trajectories with initial points sampled from the uniform distribution on $[-10,10]^3$, where $T_s = 0.033$ (left) and $0.06$ (right). {\bf Left:} results for $T_s=0.033$; the blue thin curves are the $y$-coordinate of the predicted trajectories and the green thick curves are the $y$-coordinate of the true ones. {\bf Right:} results for $T_s=0.06$; the same plotting convention is used. The initial point for each displayed curve is $(10,1,10)\in \mathbb{R}^3$.
}
\label{fig: reconstruction_lorenz, discrete}
\end{figure}

\subsection{\revised{Error analysis and sample complexity}}
Here, we discuss the error analysis and sample complexity for the JetEDMD algorithm for the parameter $n$ and the number of samples $N$
We consider the Ricker map \cite{Ricker1954}, a one-dimensional dynamical system $f:\mathbb{R} \to \mathbb{R}$ defined by 
\begin{align}
    f(x) = xe^{r(x-1)}.
\end{align}
We consider the  case of $r = 2.8$.
We note that the Ricker map chaotically behaves if $r$ is larger than a constant $r_0 \approx 2.692$.

We consider the exponential kernel $k(x,y) = e^{xy}$ of $\sigma = 1$.
We take $N$ points $x_1,\dots, x_N \in \mathbb{R}$ sampled from the uniform distribution on $[-1/2, 1/2]$ and define $y_i := f(x_i)$.
According to Theorem \ref{thm: convergence of algorithm, discrete}, the JetEDMD matrix $\widehat{\mathbf{C}}(m,n,N) \in \mathbb{R}^{(m+1) \times (m+1)}$ computed using these data converges to the matrix $\mathbf{B}_m: (b_{i,j})_{i,j=1,\dots,m+1}$ defined by 
\begin{align*}
    b_{i,j} := 
    \begin{cases}
        \displaystyle \frac{(i-1)!B_{j-1,i-1}\big(f'(0), f''(0), \dots, f^{(j-i+1)}(0)\big)}{(j-1)!} & \text{ if }j \ge i\\
        0 & \text{ if } j < i,
    \end{cases}
\end{align*}
where $B_{i,j}$ is the exponential Bell polynomial (see \cite[Theorem A in Section 3.4]{Comtet2012}).

The left panel in Figure \ref{fig: error analysis} shows the approximation accuracy and its upper bound described in Theorem \ref{thm: explicit asymptotic, disc}.
Here, we set $m=4, 6$ and $N = 1000$.
The line plot with circular ($m=4$) or triangular ($m=6$) markers is the graph of $\log \|\widehat{\mathbf{C}}(m,n,1000) - \mathbf{B}_6\|$.
The thick solid ($m=4$) or dotted ($m=6$) line is that of $0.5n\log n - (m+1)\log n + \mathrm{const.}$ from $n=6$ to $35$.
We observe that the decay rate of the actual error is consistent with the theoretical upper bound.
In the case of $m = n$, the matrix $\widehat{\mathbf{C}}(5,5,1000)$ coincides with the EDMD matrix using monomials of degree up to $5$.
We observe that it does not approximate the matrix $ \mathbf{B}_5$.
We note that the constant $C$ of the upper bound in Theorem \ref{thm: explicit asymptotic, disc} is not explicit as it depends on various parameters except $n$.
Therefore, we choose the $y$-intercepts of the solid and dotted line plots arbitrarily, so the vertical gaps between the two lines and the lines with markers are not so meaningful.
The plots are only meant to compare their shapes or decay rates.

The right panel of Figure \ref{fig: error analysis} illustrate the sample complexity.
Here, we set $m=6$ and $n=33$.
The curve shows the means of $\log$ of errors across $5000$ trials for each $N = 34,\dots,200$, with the shaded band showing 1 standard deviation around the mean.
For small $N$ (about 34 - 44), the error is relatively large and has high variation, indicating an unstable regime due to the sampling noise.
As $N$ increases, the means of errors drop sharply and the band shrinks.
For $N\ge 150$, the curve of mean errors flattens around $\log(\text{error})\approx -4$ with a narrow shaded band.
The overlap of the bands across $N$ suggests only marginal improvement beyond this point even if $N$ increases.
We observe that there exists a optimal number of samples to estimate the target matrix, which is an crucial future work.

\begin{figure}
\begin{minipage}[c]{1.0\linewidth}
    \begin{center}
    \begin{minipage}[c]{0.45\linewidth}
      \includegraphics[keepaspectratio, width=\linewidth]{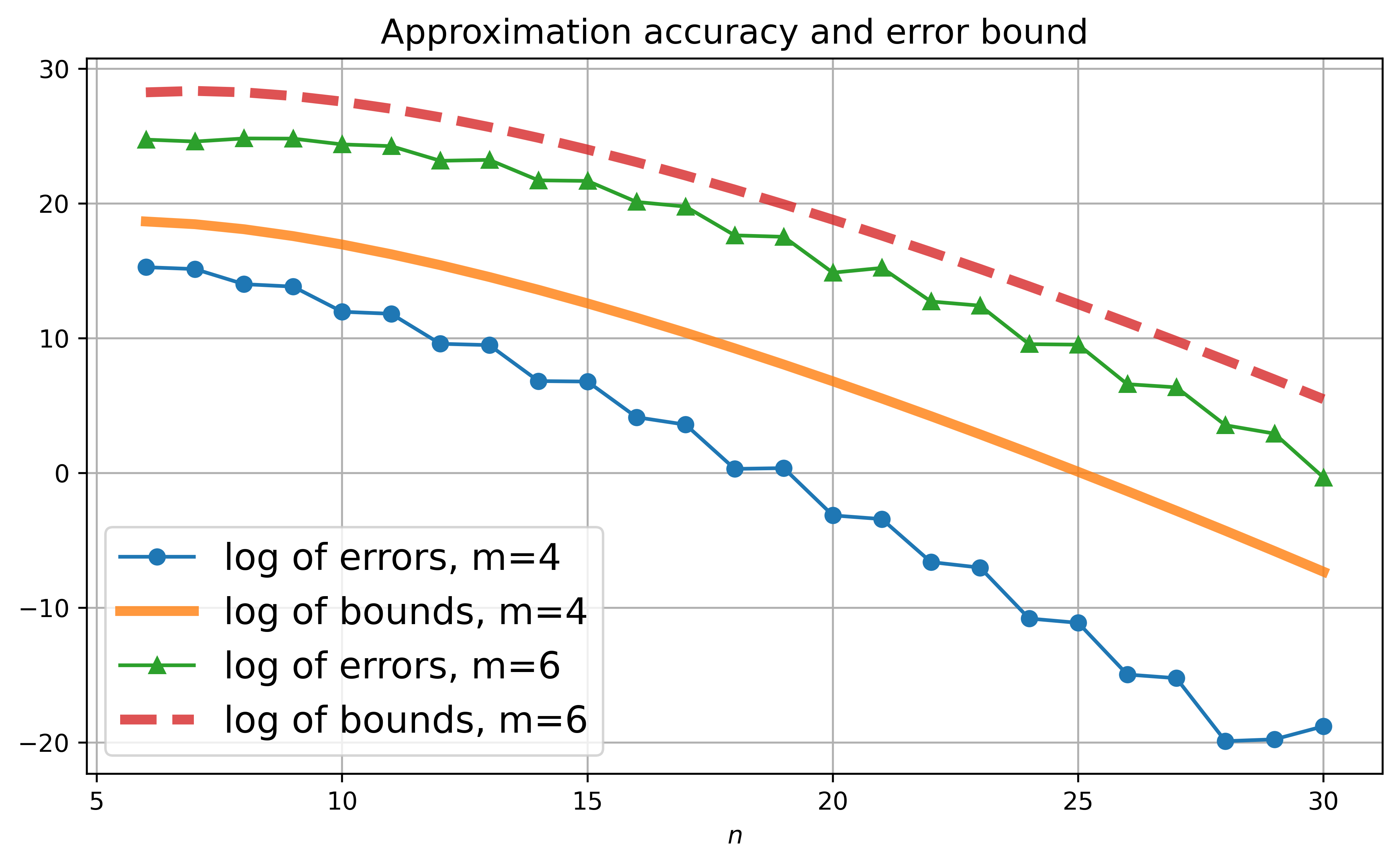}
    \end{minipage}
    \begin{minipage}[c]{0.45\linewidth}
      \includegraphics[keepaspectratio, width=\linewidth]{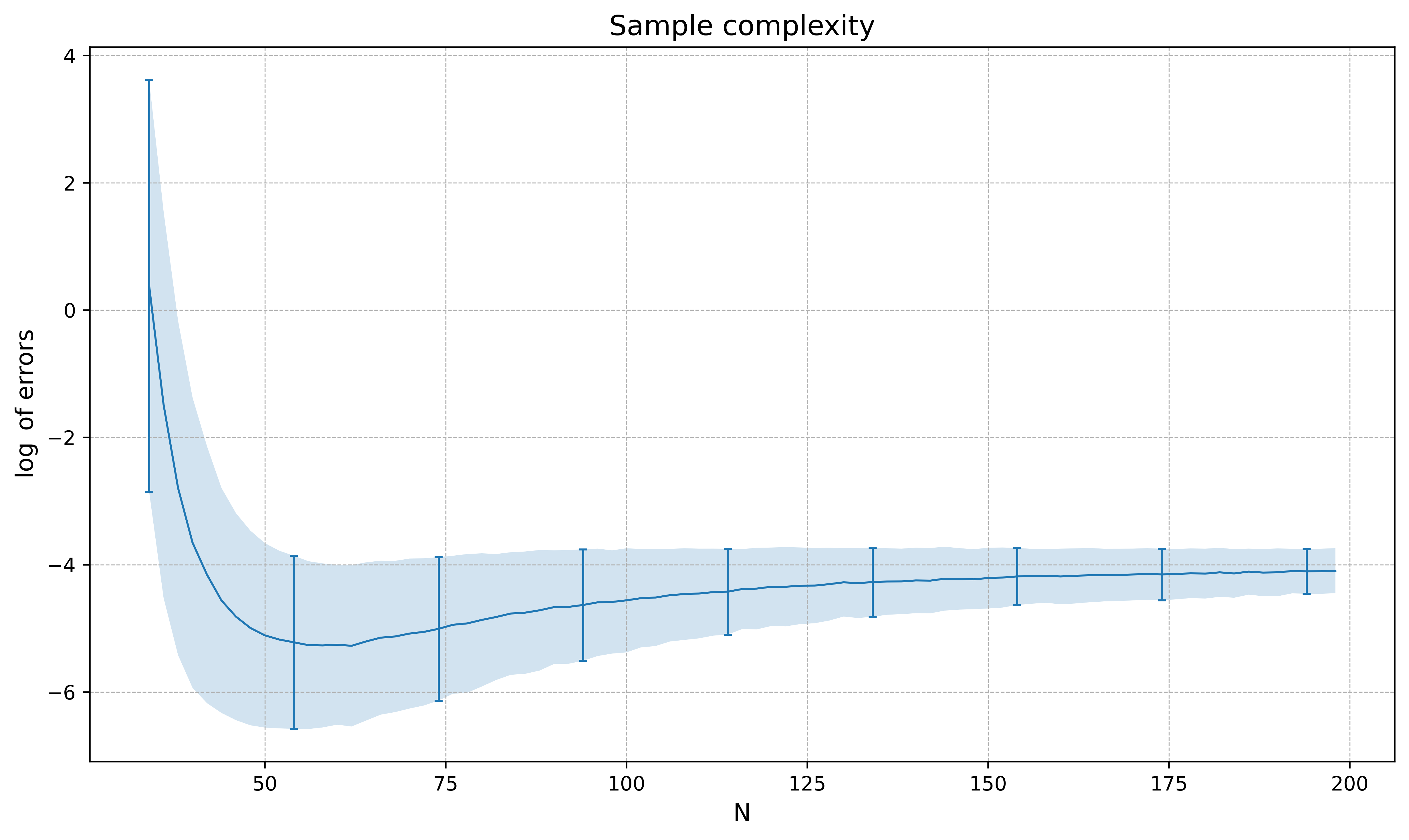}
    \end{minipage}
    \end{center}
\end{minipage}
\caption{
    {\bf Left:} Approximation error versus the upper bound in Theorem \ref{thm: explicit asymptotic, disc} with $m=4,6$ and $N=1000$. 
    {\bf Right:} Sample complexity with $m=6$ and $n=33$. The curve shows the mean log error over 5000 trials for $N=34, \dots, 200$, and the shaded band indicates one standard deviation. 
}
\label{fig: error analysis}
\end{figure}

\subsection{Relation among Jet DMD, Extended DMD, and Kernel DMD}\label{subsec: relation to EDMD}
The matrix $ \mathbf{V}_n^Y (\mathbf{V}_n^X)^\dagger$ appearing on the left-hand side of \eqref{computation of hatCnm} in Algorithm \ref{algorithm: data-driven PF operator estimation, discrete} coincides with the transpose of the matrix that appears in EDMD \cite{Williams2015} using $\{v_{p,\alpha}\}_{|\alpha| \le n}$ as the observable functions.
In the framework EDMD, they regard the matrix $ \mathbf{V}_n^Y (\mathbf{V}_n^X)^\dagger$ as the correct approximation of the Koopman operator.
However, as shown in Theorem \ref{thm: error analysis for PF operator for algorithm, disc}, this matrix will not correctly approximate the operator (see the discussion after Theorem \ref{thm: error analysis for PF operator for algorithm, disc}).
Therefore, JetEDMD provides an appropriate refinement of EDMD.
While JetEDMD involves only a very simple procedure of truncation to a leading principal submatrix of the matrix $ \mathbf{V}_n^Y (\mathbf{V}_n^X)^\dagger$, it enables a clear depiction of the eigenvalues as in Figure \ref{fig: difference of spectra quadratic map} in Section \ref{sec: introduction}.

When the positive definite kernel $k(x,y)$ is the exponential kernel $e^{x^\top y}$ and the dynamical system has a fixed point at origin, it corresponds to the EDMD with the monomials.
According to the Stone--Weirestrass theorem, any continuous function can be approximated with polynomials on a compact set with arbitrary precision, providing a certain theoretical justification to the use of polynomials.
However, in the context of JetEDMD, the polynomials appear much more naturally and canonically through a process involving the jets and the RKHS as in Section \ref{sec: theory}.
Furthermore, by considering rigged Hilbert space, the Koopman operator approximated by JetEDMD can be clearly understood as an approximation of the extended Koopman operator explained in Section \ref{sec: generalized spectrum}.

We also provide a remark on the relation between JetEDMD and the Kernel DMD (KDMD) \cite{Kawahara, WRK2015}.
We impose the condition that $r_n = \mathop{\rm dim}V_{p,n}$ is less than or equal to the number $N$ of samples in the framework of JetEDMD, but the matrix $ \mathbf{V}_n^Y (\mathbf{V}_n^X)^\dagger$ itself may be considered on the opposite case, namely the case of $r_n > N$.
In this case, the eigenvalues of $ \mathbf{V}_n^Y (\mathbf{V}_n^X)^\dagger$ approximately correspond to those computed by KDMD.
In fact, assume that $r_n > N$ and $\mathbf{V}_n^X$ is a full column rank matrix. 
By \cite{4aa77599-3504-3cfd-9aba-80cca6f578f2}, we have
\begin{align*}
    (\mathbf{V}_n^X)^\dagger\mathbf{G}_n^{1/2} = (\mathbf{G}_n^{-1/2}\mathbf{V}_n^X)^\dagger 
= \left((\mathbf{V}_n^X)^*\mathbf{G}_n^{-1}\mathbf{V}_n^X\right)^{-1}(\mathbf{V}_n^X)^*\mathbf{G}_n^{-1/2}. \label{eq1DMD}
\end{align*}
Thus, we have
\[ (\mathbf{V}_n^X)^\dagger \mathbf{V}_n^Y = \left((\mathbf{V}_n^X)^*\mathbf{G}_n^{-1}\mathbf{V}_n^X\right)^{-1}(\mathbf{V}_n^X)^*\mathbf{G}_n^{-1}\mathbf{V}_n^Y.\]
By Proposition \ref{coefficient of minimizer}, we see that 
\[(\mathbf{V}_n^X)^*\mathbf{G}_n^{-1}\mathbf{V}_n^Y = \left( \overline{\langle \pi_nk_{x_i}, \pi_nk_{y_i} \rangle_H} \right)_{i,j =1,\dots, N} 
\underset{n \to \infty}{\longrightarrow} \left(k(y_i, x_j) \right)_{i,j =1,\dots, N}.\]
By \cite[Theorem 1.3.22.]{MR2978290}, the nonzero eigenvalues of $\mathbf{V}_n^Y (\mathbf{V}_n^X)^\dagger$ coincide with those of $ (\mathbf{V}_n^X)^\dagger\mathbf{V}_n^Y$.
Therefore, we conclude that the computation of the eigenvalues of $\mathbf{V}_n^Y (\mathbf{V}_n^X)^\dagger$ is essentially equivalent to that of $\left(k(x_i, x_j) \right)_{i,j =1,\dots, N}^{-1}\left(k(y_i, x_j) \right)_{i,j =1,\dots, N}$, that is the objective matrix of the KDMD. 

\section{Conclusion}
This paper introduces a novel approach to estimating the Koopman operator on the RKHS, through the development of the JetEDMD.
In the context of data analysis, the $L^2$-space has been the central stage for Koopman analysis, originated from Koopman's work.
The Koopman operator on RKHS and its application in data analysis are relatively new fields.
Through this paper, we intend to show that the RKHS is also quite promising venue for developing the theory of the Koopman operator.
In fact, we present accurate convergence results, backing the performance of JetEDMD for RKHSs for two positive definite kernels, the exponential kernel and Gaussian kernel. 
We also show that some existing methods, such as EDMD with monomials, is considered within the framework of RKHS.
Furthermore, we investigate into the spectral analysis of Koopman operator.
More precisely. we introduce the notion of the extended Koopman operator in the framework of the rigged Hilbert space to achieve the deeper understanding of the ``Koopman eigenfunctions'', leading to a promising method to analyze the spectrum of the Koopman operator.

Although we only consider the the exponential kernel and Gaussian kernel, exploring the performance and properties of JetEDMDs with other kernels is also quite an important challenge.
\revised{Furthermore, relaxing the assumptions on the smoothness and analyticity of the dynamical system $f$, and extending our considerations to cases with unbounded state spaces, present important directions for future research.}
This paper also clarifies to which the ``Koopman eigenfunction'' belongs by regarding it as an eigenvector of the extended Koopman operator.
However, the definition of the extended Koopman operator is rather abstract, and its theoretical interpretation still is needed to be explored.
Additionally, elucidating the relationship between the eigenvectors of the extended Koopman operator and the dynamical systems is also a significant problem since the former must capture some important characteristics of dynamical system although it is not an eigenfunction in the usual sense.

In conclusion, we believe that this paper significantly advances the  study of Koopman operators on RKHSs.
We hope the ideas and methodologies presented in this paper will serve as a foundation for further development of Koopman operators.

\bibliography{reference}
\bibliographystyle{plain}

\appendix

\section{Additional Numerical Experiments}

\subsection{Eigenvalues for van der Pol oscillators}
Figure \ref{fig: vdP_eigenvalues_exp2} describes the estimation of eigenvalues of the Perron--Frobenius operators $A_F^*|_{V_{p,m}}$ of the van der Pol oscillator \eqref{vdP} for $\mu=1$, $2$, and $3$, using Algorithm \ref{algorithm: data-driven PF operator estimation, continuous} with the exponential kernel $k^{\rm e}(x,y) = e^{(x-b)^\top (y-b)/\sigma^2}$ and the Gaussian kernel $k^{\rm g}(x,y) = e^{-|x-y|^2/2\sigma^2}$ with $\sigma=2$ and $b=0$.
We use $m=5$, $n=7$, and $N=36$ samples from the uniform distribution on $[-1,1]^2$ for the exponential kernel, while $m=5$, $n=9$, and $N=66$ for the Gaussian kernel.
We use the exact velocities at the samples to compute $\widehat{\mathbf{A}}$.
Basically, the exponential kernel needs less $n$ and $N$ than the Gaussian kernel.
In the cases of $\mu=1,3$, the Hausdorff distance between the estimated eigenvalues and the true ones is less than $10^{-9}$, indicating that $\widehat{\mathbf{A}}$ approximates $A_F^*|_{V_{p,m}}$ in high precision, but for $\mu=2$, the Hausdorff distance becomes greater than $10^{-2}$.
This is due to the fact that the linear map $A_F^*|_{V_{p,m}}$ is not diagonalizable in the case of $\mu=2$, resulting in errors in the numerical computation of eigenvalues.
However, as shown in Section \ref{sec: system identification}, we demonstrate the high performance in the tasks of identification of differential equations.
It indicates that the matrix $\widehat{\mathbf{A}}$ itself provides a correct estimation of $A_F^*|_{V_{p,m}}$.


\begin{figure}
\begin{minipage}[c]{1.0\linewidth}
    \begin{minipage}[c]{0.33\linewidth}
      \includegraphics[keepaspectratio, width=\linewidth]{images/vdP_mu10_exp_m5_n7_sigma200_N36_seed1_conti_eigenvalues.png}
    \end{minipage}\hfill 
    \begin{minipage}[c]{0.33\linewidth}
      \includegraphics[keepaspectratio, width=\linewidth]{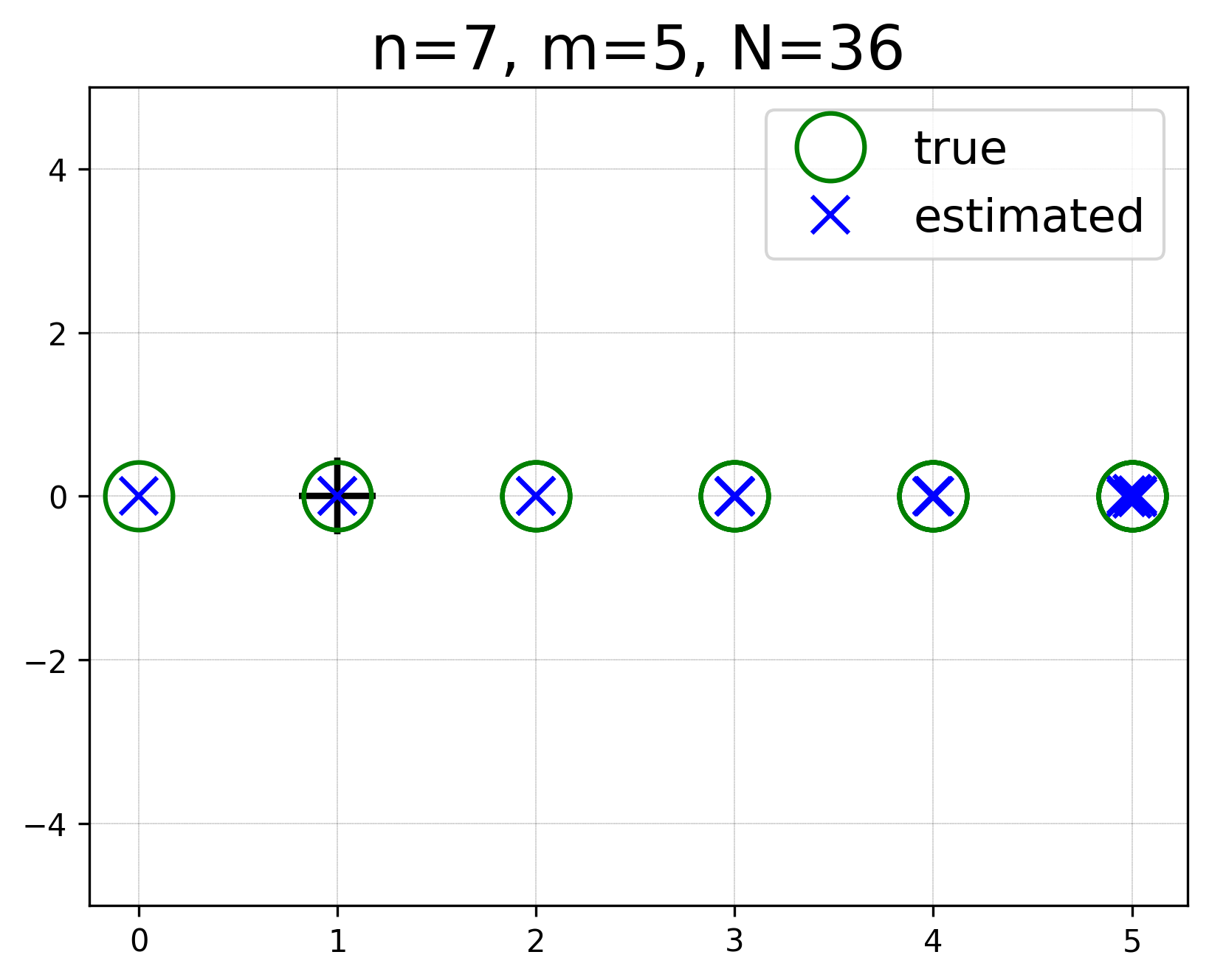}
    \end{minipage}\hfill 
    \begin{minipage}[c]{0.33\linewidth}
      \includegraphics[keepaspectratio, width=\linewidth]{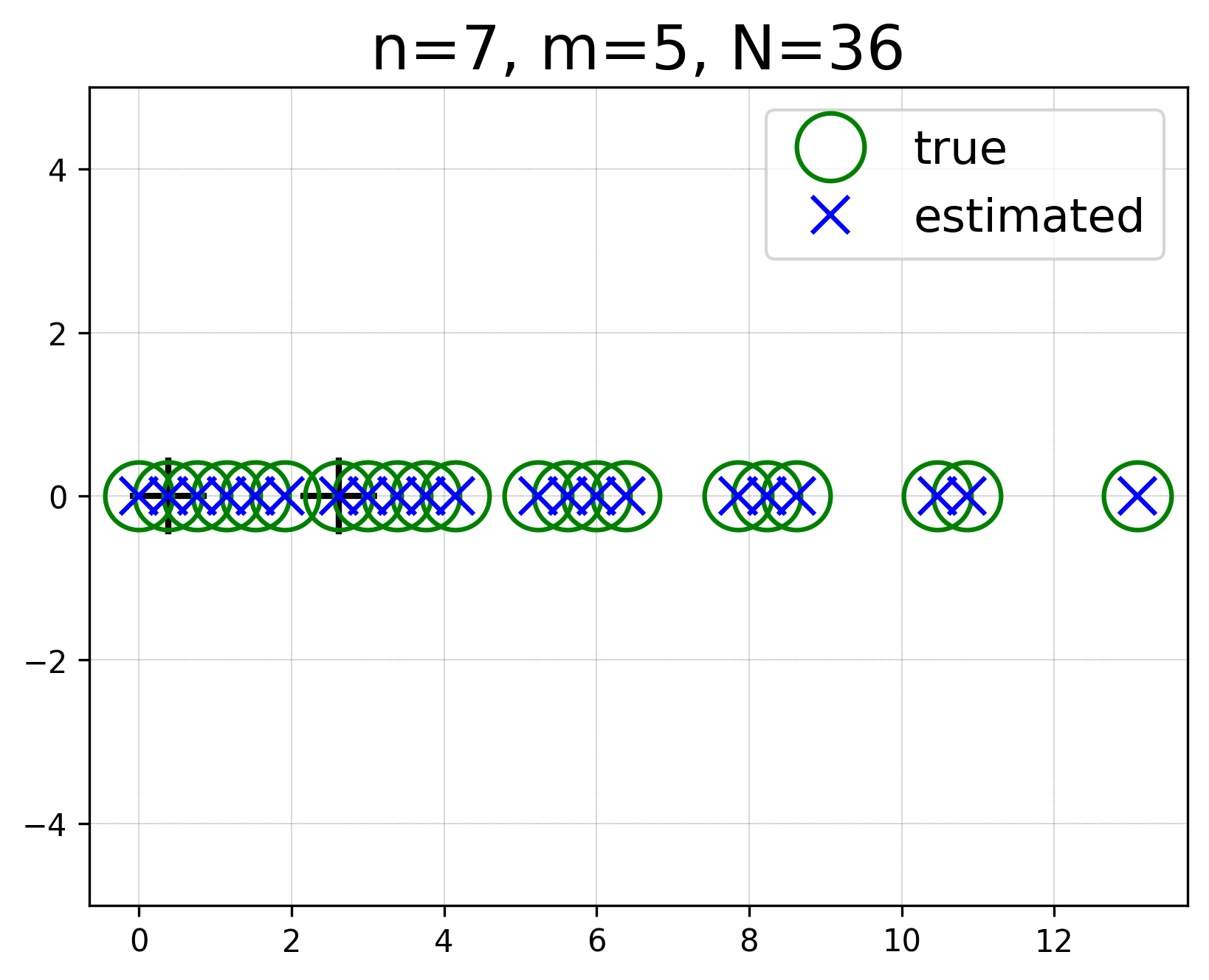}
    \end{minipage}
\end{minipage}
\begin{minipage}[c]{1.0\linewidth}
    \begin{minipage}[c]{0.33\linewidth}
      \includegraphics[keepaspectratio, width=\linewidth]{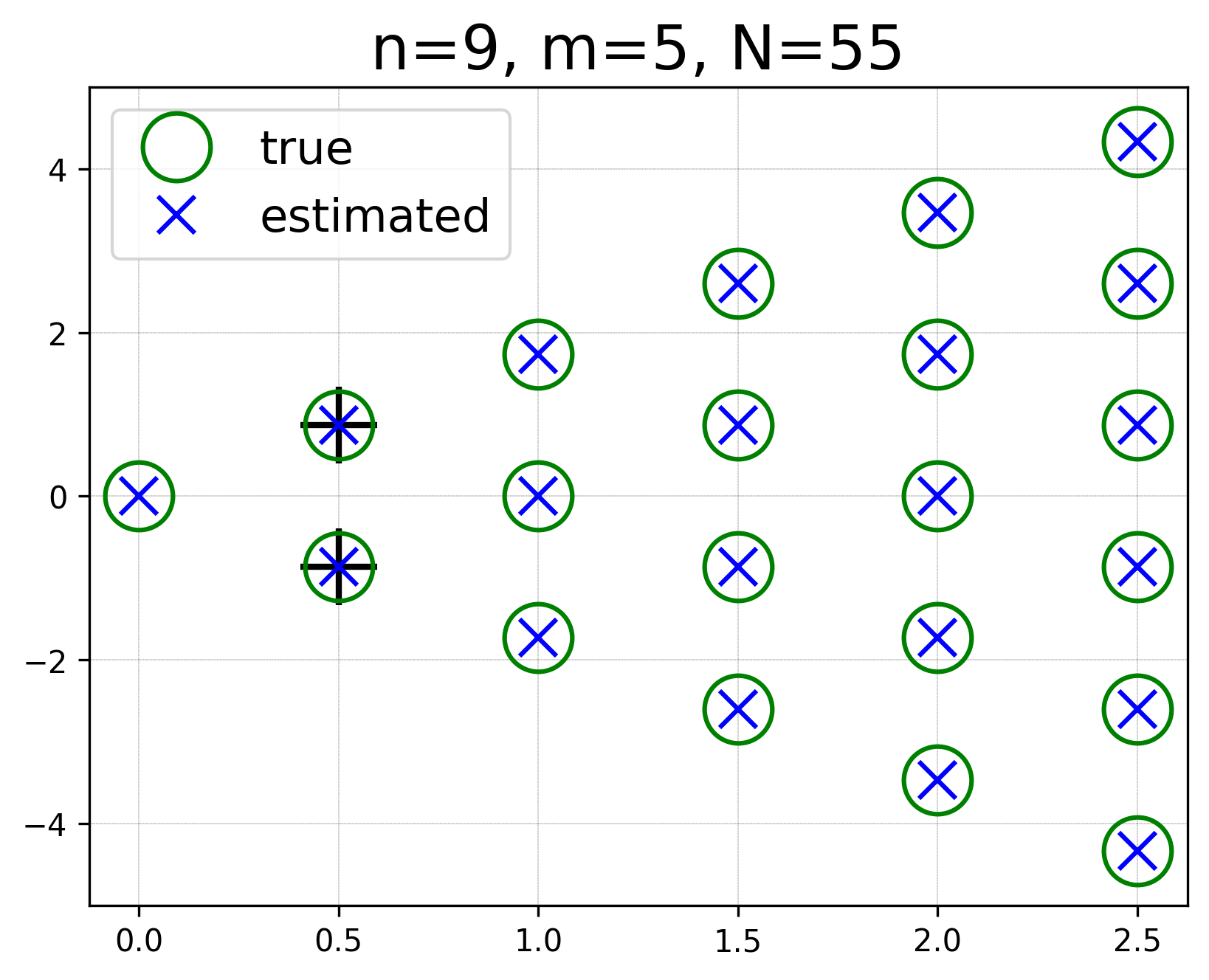}
    \end{minipage}\hfill 
    \begin{minipage}[c]{0.33\linewidth}
      \includegraphics[keepaspectratio, width=\linewidth]{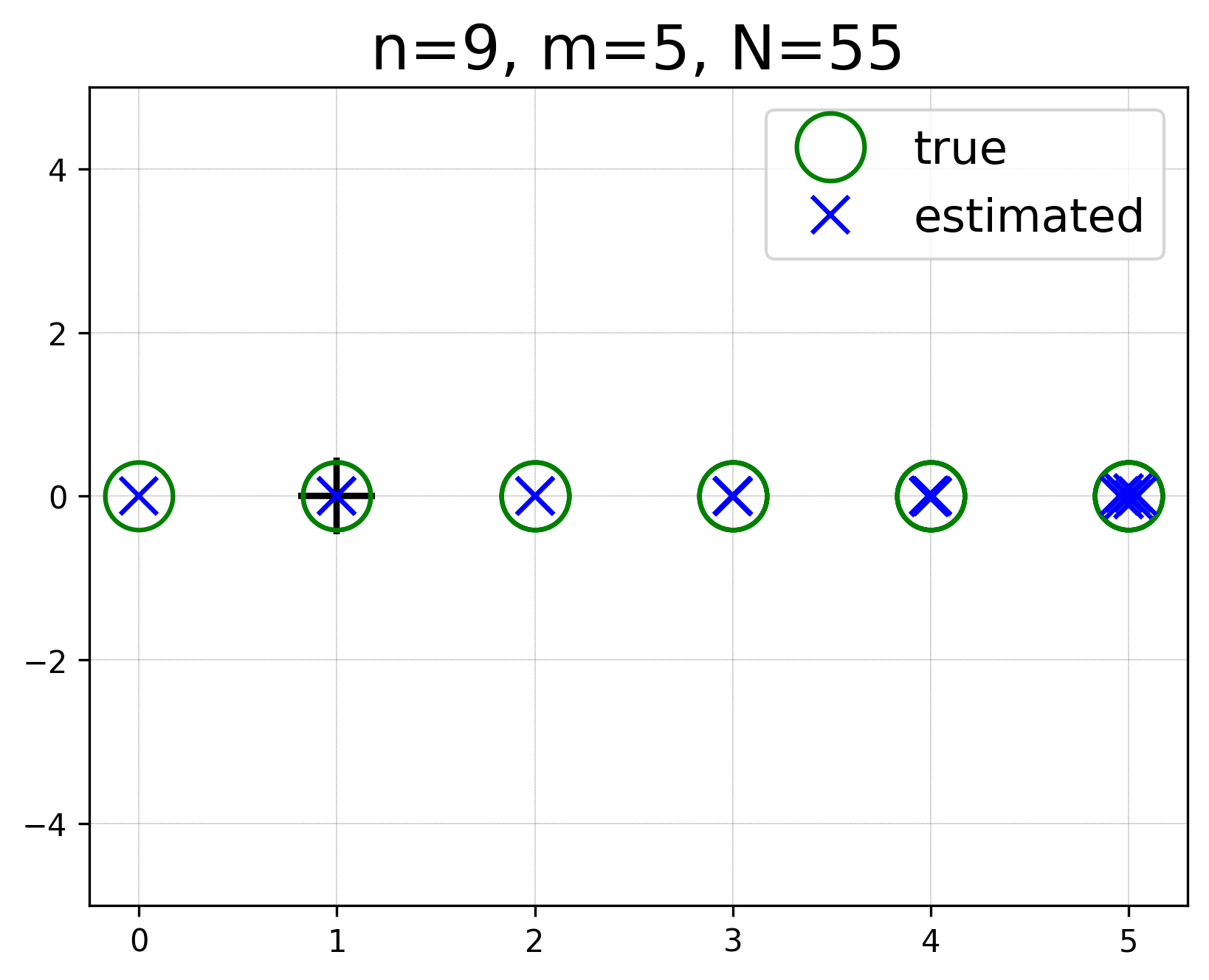}
    \end{minipage}\hfill 
    \begin{minipage}[c]{0.33\linewidth}
      \includegraphics[keepaspectratio, width=\linewidth]{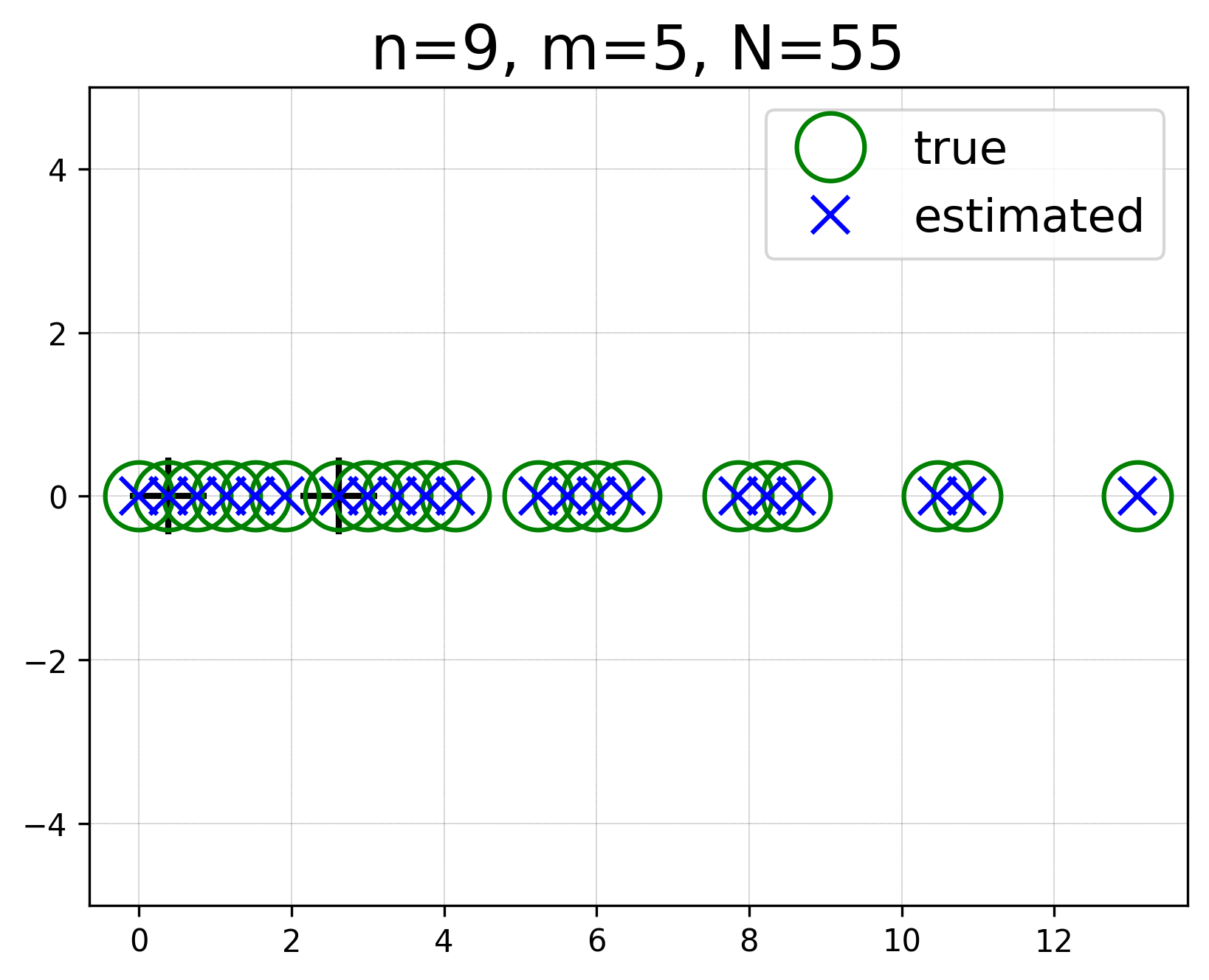}
    \end{minipage}
\end{minipage}
\caption{
Estimation of eigenvalues using Algorithm \ref{algorithm: data-driven PF operator estimation, continuous} for van der Pol oscillators with $\mu=1$ (left), $2$ (middle), and $3$ (right).
On the top row, we use the exponential kernel $k^{\rm e}(x,y) = e^{x^\top y/4}$ with $m=5$, $n=7$, $p=(0,0)$, $N=36$ samples from the uniform distribution on $[-1,1]^2$, and the exact velocities on them.
On the bottom row, we use the Gaussian kernel $k^{\rm g}(x,y) = e^{-|x-y|^2/8}$ with $m=5$, $n=9$, $p=(0,0)$, $N=55$ samples from the uniform distribution on $[-1,1]^2$, and the exact velocities on them.
The blue $\times$'s indicate the eigenvalues of the estimated Perron--Frobenius operator $\widehat{\mathbf{A}}$ and the green circles indicate the eigenvalues of $A_F|_{V_{p,m}}$.
The two $+$'s indicate the eigenvalues of the Jacobian matrix of the vector field $F(x,y) = (y,\,\mu(1-x^2)y-x)$ at $p=(0,0)$.
}
\label{fig: vdP_eigenvalues_exp2}
\end{figure}

\subsection{Eigenfunctions for van der Pol oscillators}
Figure \ref{fig: vdP_eigenfunctions_mu1} is a collection of the approximated eigenfunctions of the extended Koopman operator $A_F^\times$ of the van der Pol oscillator for $\mu=1$ and $\mu=3$ via Algorithm \ref{algorithm: eigenfunction, continuous} using the exponential kernel $e^{x^\top y/1.4}$.
We take $m=20$, $n=24$, $p=(0,0)$, $N=7000$ samples from the uniform distribution on $[-1,1]^2$, and the exact velocities on them.
The left two panels are the absolute values and arguments of the eigenfunction corresponding to the eigenvalue $0.5+0.87\mathrm{i} \approx (1+\sqrt{3}\mathrm{i})/2$.
The right panel is the eigenfunction corresponding to the eigenvalue $1$.
It can be observed that the absolute value of the estimated eigenfunction increases rapidly once it exceeds the limit cycle. 
Similar figures are obtained in \cite{6760712, 7384725}, and our method can be considered as a data-driven version of the Taylor expansion method proposed in these articles.

\begin{figure}
\begin{minipage}[c]{1.0\linewidth}
    \begin{minipage}[c]{0.64\linewidth}
      \includegraphics[keepaspectratio, width=\linewidth]{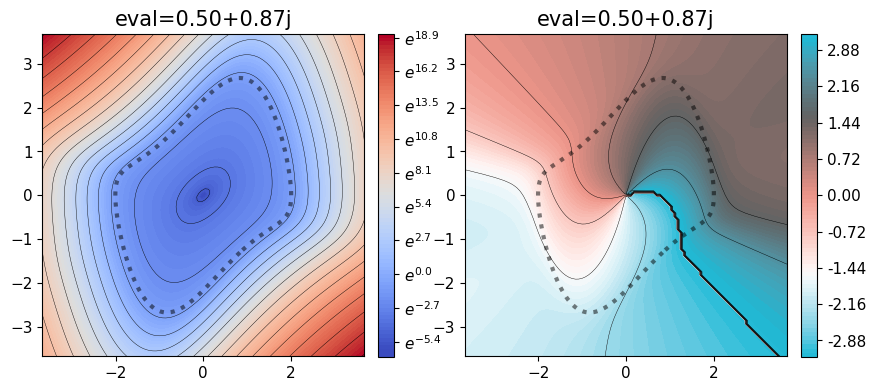}
    \end{minipage}\hfill 
    \begin{minipage}[c]{0.33\linewidth}
      \includegraphics[keepaspectratio, width=\linewidth]{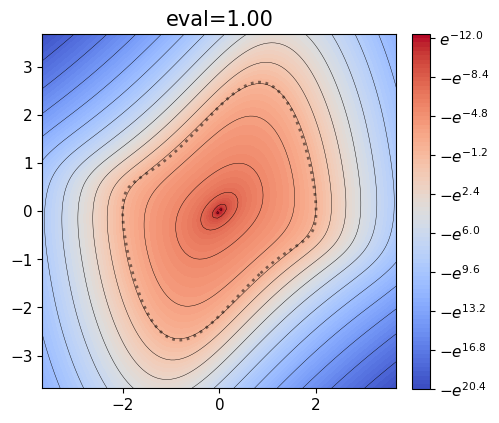}
    \end{minipage}
\end{minipage}
\caption{
Approximated eigenfunction of the extended Koopman operator $A_F^\times$ of the van der Pol oscillator for $\mu=3$ via Algorithm \ref{algorithm: eigenfunction, continuous} using the exponential kernel $e^{x^\top y/1.4}$ with input $m=20$, $n=22$, $p=(0,0)$, $N=7000$ samples from the uniform distribution on $[-1,1]^2$, and the exact velocities on them.
The left and middle panels are the heat map with contour lines corresponding to the absolute values and arguments of the eigenfunction for eigenvalue $0.5 + 0.87\mathrm{i}$.
The right panel is the heat map with contour lines corresponding to the eigenfunction for eigenvalue $1.0$.
The dotted lines indicate the limit cycles.
}
\label{fig: vdP_eigenfunctions_mu1}
\end{figure}

Figure \ref{fig: vdP_eigenfunctions_mu3} describes the approximated eigenfunction of the extended Koopman operator $A_F^\times$ of the van der Pol oscillator for $\mu=3$, and a capability to capture a characteristic of the dynamical system.
The left panel is the heat map with contour lines corresponding to the eigenfunctions of eigenvalue $1.0$ via Algorithm \ref{algorithm: eigenfunction, continuous} using the exponential kernel $e^{x^\top y/1.4}$.
We take $m=16$, $n=21$, $p=(0,0)$, $N=7000$ samples from the uniform distribution on $[-1,1]^2$, and the exact velocities on them.
Looking at the left panel, unlike the case of $\mu=1$, the contour lines of the eigenfunction form a distinctive inclined ``$S$'' shape around the origin, while, as shown in the middle panel, the dynamics with initial values near the origin exhibits highly skewed behavior over time.
The approximated eigenfunction captures this behavior of the dynamical system around the origin.

\begin{figure}
\begin{minipage}[c]{1.0\linewidth}
    \begin{minipage}[c]{0.38\linewidth}
      \includegraphics[keepaspectratio, width=\linewidth]{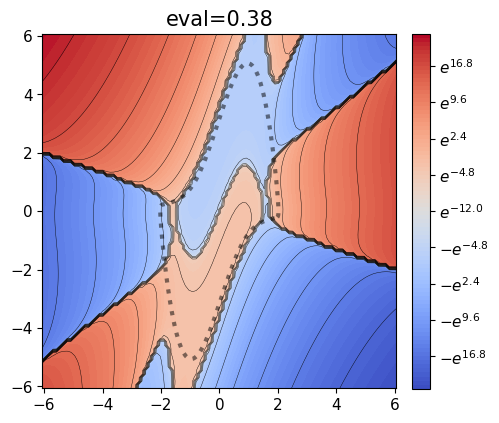}
    \end{minipage}\hfill 
    \begin{minipage}[c]{0.30\linewidth}
      \includegraphics[keepaspectratio, width=\linewidth]{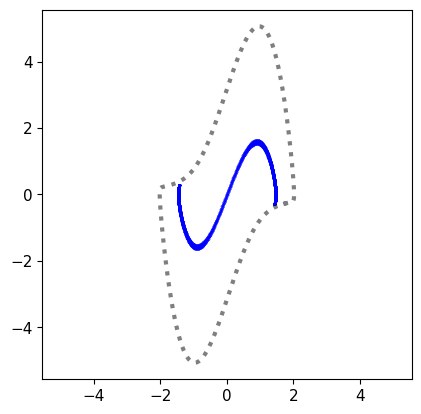}
    \end{minipage}\hfill 
    \begin{minipage}[c]{0.31\linewidth}
      \includegraphics[keepaspectratio, width=\linewidth]{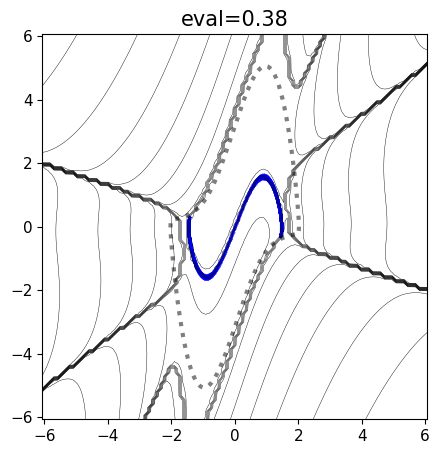}
    \end{minipage}
\end{minipage}
\caption{
Estimated eigenfunction of the extended Koopman operator $A_F^\times$ of the van der Pol oscillator for $\mu=3$ via Algorithm \ref{algorithm: eigenfunction, continuous} using the exponential kernel $e^{x^\top y/1.4}$ with input $m=16$, $n=21$, $p=(0,0)$, $N=7000$ samples from the uniform distribution on $[-1,1]^2$, and the exact velocities on them.
The left panel is the heat map with contour lines corresponding to the eigenfunctions of eigenvalue $0.38$.
The middle panel is the image of the small rectangle domain $[-0.01,0.01]^2$ under the flow map $\phi^t$  at $t=3$.
The right panel is a combination of the middle panel and the contour lines of the left panel.
The dotted lines indicate the limit cycles.
}
\label{fig: vdP_eigenfunctions_mu3}
\end{figure}

\subsection{Eigenvalues and eigenfunctions for Duffing oscillators}


Figure \ref{fig: duffing_eigenvalues} describes the estimation of eigenvalues of the Perron--Frobenius operators $A_F^*|_{V_{p,m}}$ of the Duffing oscillator \eqref{duffing}  using Algorithm \ref{algorithm: data-driven PF operator estimation, continuous} with the exponential kernel $k^{\rm e}(x,y) = e^{-(x-b)^\top (y-b)/\sigma^2}$ with $\sigma=1$ and $b=0$.
We take $m=5$, $n=10$ and pick $N=66$ samples from the uniform distribution on $[-2,2]^2$.
We use the exact velocities at the samples to compute $\widehat{\mathbf{A}}$.
In each case, the Hausdorff distance between the estimated eigenvalues and the true ones is less than $10^{-7}$, indicating the matrix $\widehat{\mathbf{A}}$ well approximates the Perron--Frobenius operator $A_F^*|_{V_{p,m}}$.

\begin{figure}
\begin{minipage}[c]{1.0\linewidth}
    \begin{minipage}[c]{0.33\linewidth}
      \includegraphics[keepaspectratio, width=\linewidth]{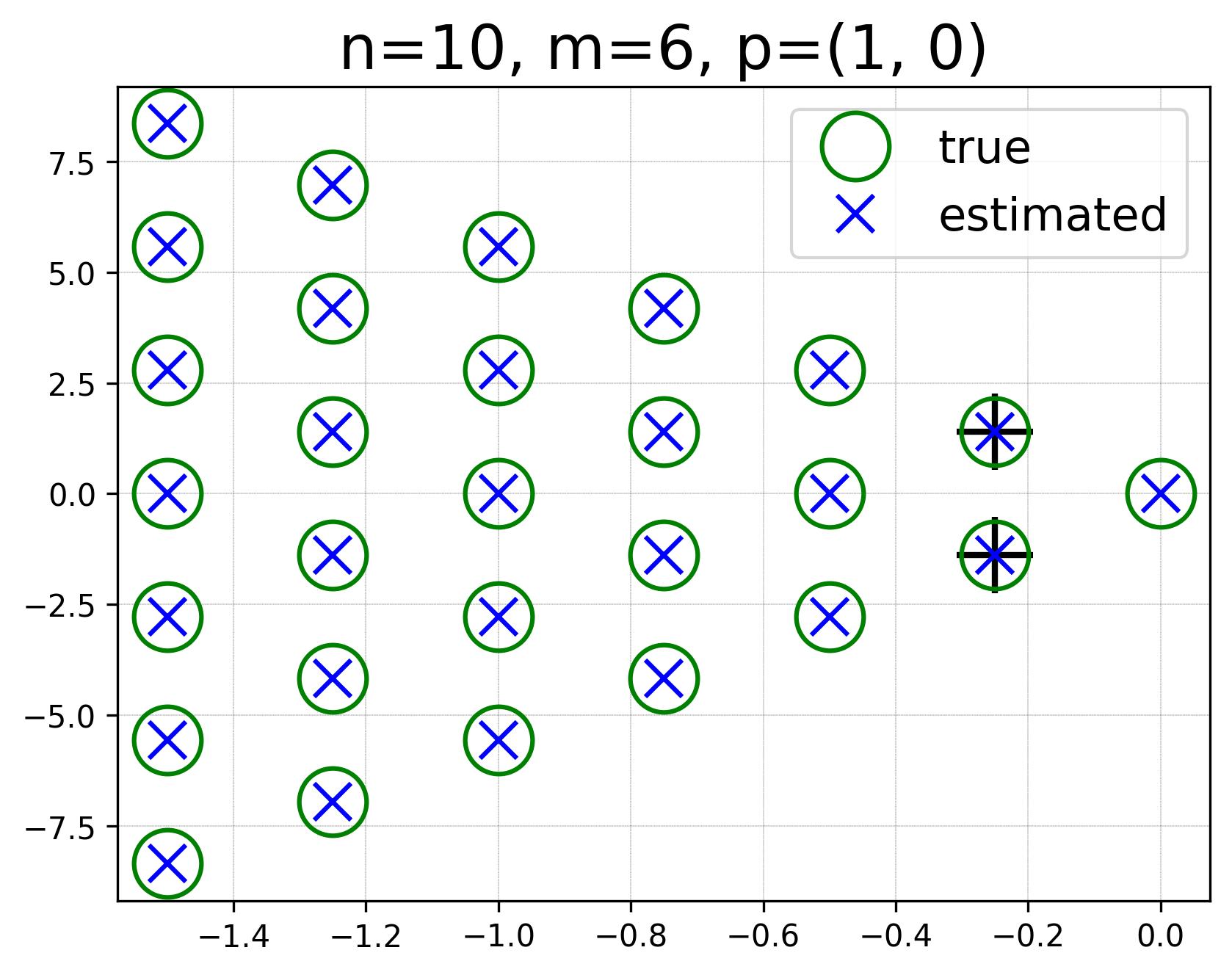}
    \end{minipage}\hfill 
    \begin{minipage}[c]{0.33\linewidth}
      \includegraphics[keepaspectratio, width=\linewidth]{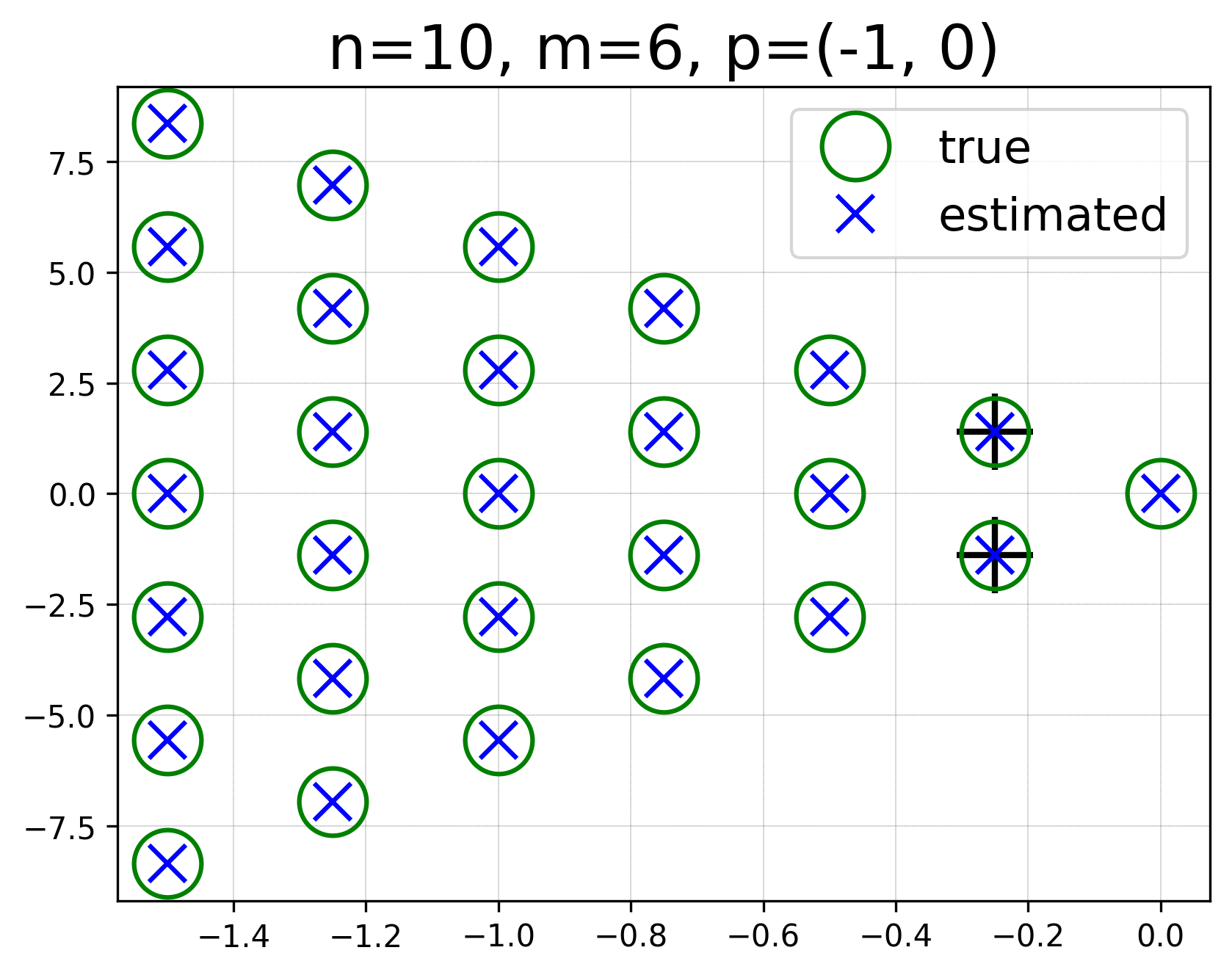}
    \end{minipage}\hfill 
    \begin{minipage}[c]{0.33\linewidth}
      \includegraphics[keepaspectratio, width=\linewidth]{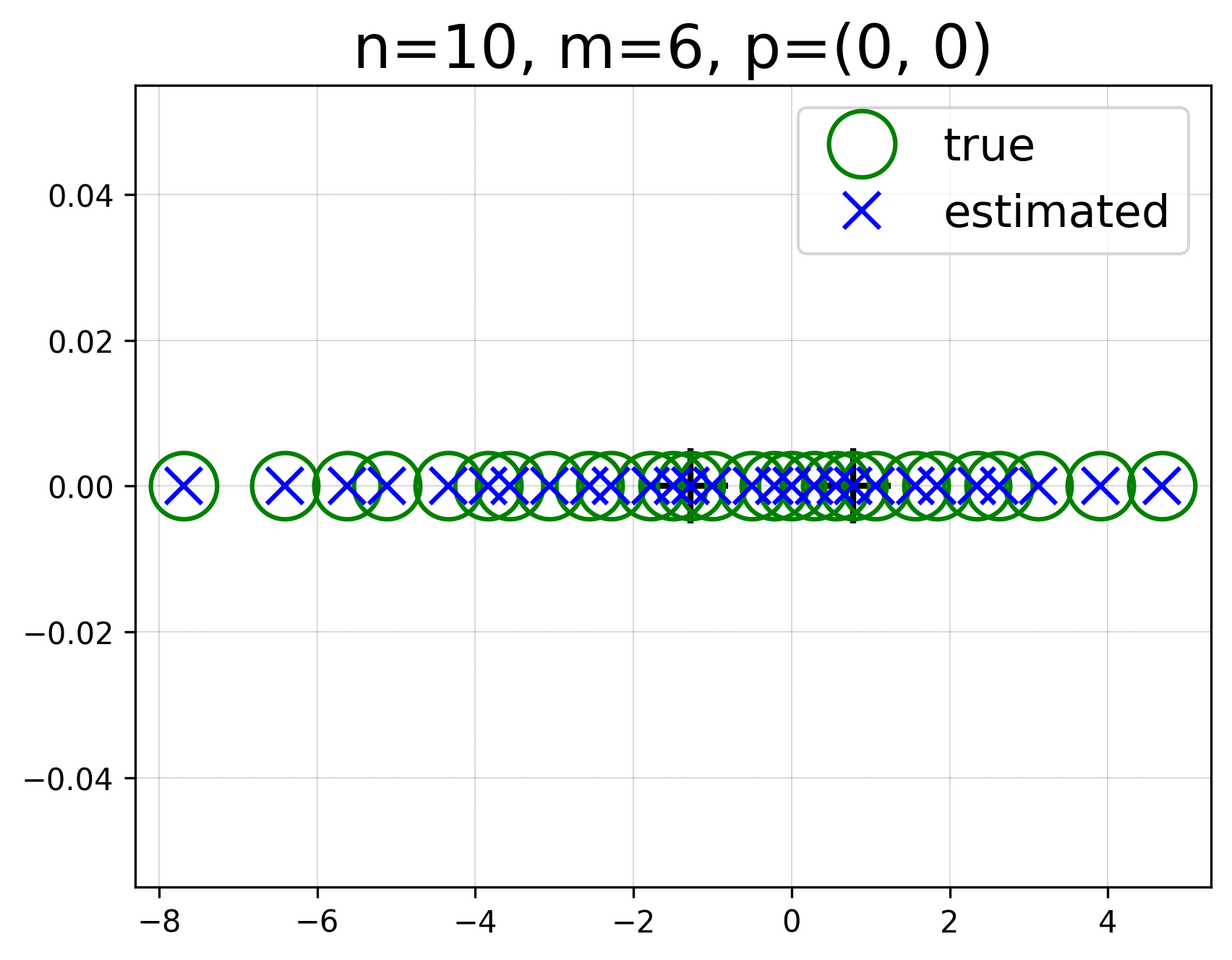}
    \end{minipage}
\end{minipage}
\caption{
Estimated eigenvalues of the Perron--Frobenius operators $A_F^*|_{V_{p,m}}$ of the Duffing oscillator \eqref{duffing} via Algorithm \ref{algorithm: eigenfunction, continuous} using the exponential kernel $e^{x^\top y}$ with input $m=6$, $n=10$, the equilibrium points, $N=66$ samples from the uniform distribution on $[-2,2]^2$, and the exact velocities on them.
The left, middle and right panels correspond to the eigenvalues computed using the three equilibrium points $(1,0)$, $(-1,0)$, $(0,0)$, respectively.
The blue $\times$'s indicate the eigenvalues of the estimated Perron--Frobenius operator $\widehat{\mathbf{A}}$ and the green circles indicate the eigenvalues of $A_F|_{V_{p,m}}$.
The two $+$'s indicate the eigenvalues of the Jacobian matrix of the vector field $F(x,y) = (y,\,-0.5y+x-x^3)$ at the equilibrium point.
}
\label{fig: duffing_eigenvalues}
\end{figure}

Figure \ref{fig: duffing_eigenfunctions_single_equilibrium_point} is a collection of the approximated eigenfunctions of the Duffing oscillator \eqref{duffing}.
for the eigenvalues $-0.25 + 1.39\mathrm{i} \approx  (-1 + \sqrt{31}\mathrm{i})/4 $ (left and middle-left), $ 0.78\approx (-1 + \sqrt{17})/4$ (middle-right), and $1.28 \approx (-1 - \sqrt{17})/4$ (right).
We use the exponential kernel  $k^{\rm e}(x,y) = e^{-(x-b)^\top (y-b)/\sigma^2}$ with $\sigma=1$ and $b=0$, and take $m_1=10$, $n_1=12$, $N_1=3000$ samples from the uniform distribution on $[-1.5,1.5] \times [-0.5,0.5]$, and the exact velocities on them as input.
The left panel is the heat map with contour lines of the absolute value of the approximated eigenfunction for the eigenvalue $ (-1 + \sqrt{31}\mathrm{i})/4 \approx -0.25 + 1.39\mathrm{i}$ with the equilibrium point $(-1,0)$.
The middle-right and right panels are the heat maps with contour lines of the eigenfunction for the eigenvalues $ (-1 + \sqrt{17})/4 \approx 0.78$ and $(-1 - \sqrt{17})/4 \approx 1.28$, respectively.
The eigenfunction for the positive eigenvalue captures the attracting direction, while that of the negative eigenvalue does the repelling direction around the origin.

\begin{figure}
\begin{minipage}[c]{1.0\linewidth}
    \begin{minipage}[c]{0.26\linewidth}
      \includegraphics[keepaspectratio, width=\linewidth]{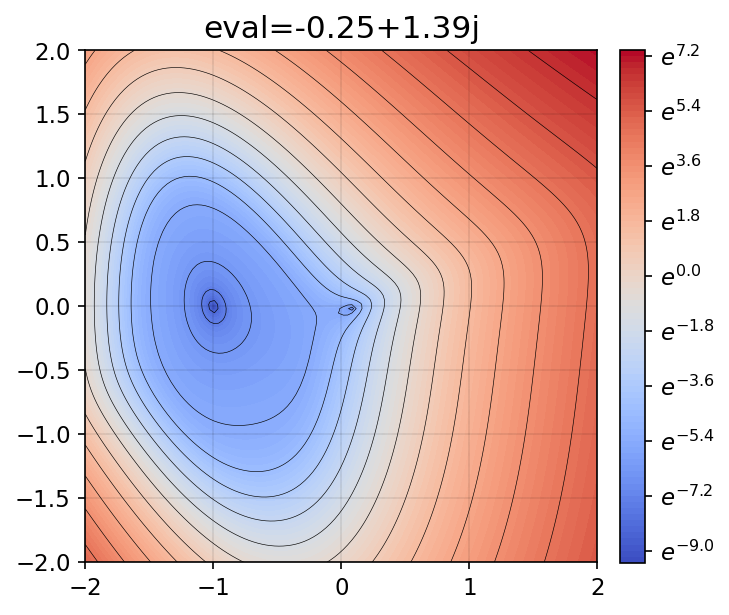}
    \end{minipage}\hfill 
    \begin{minipage}[c]{0.22\linewidth}
      \includegraphics[keepaspectratio, width=\linewidth]{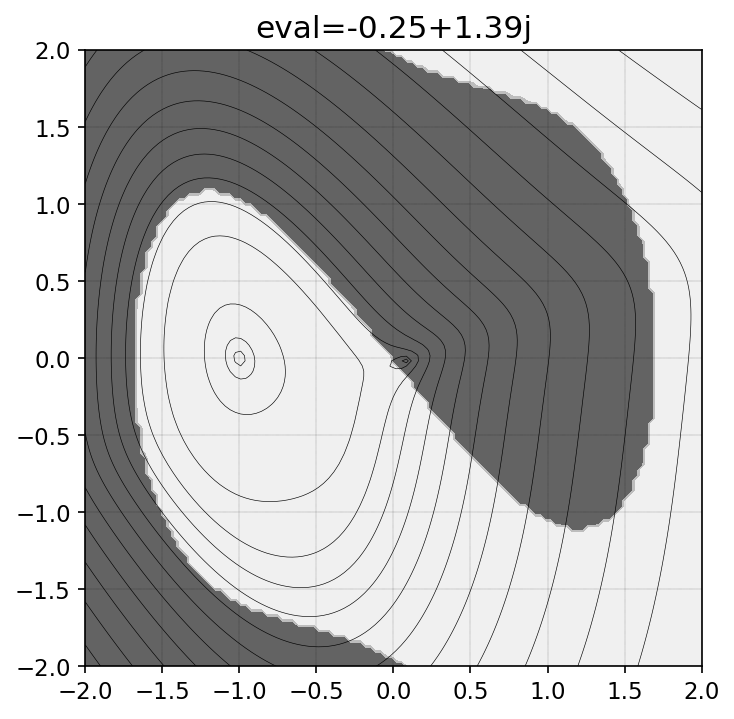}
    \end{minipage}\hfill 
    \begin{minipage}[c]{0.255\linewidth}
      \includegraphics[keepaspectratio, width=\linewidth]{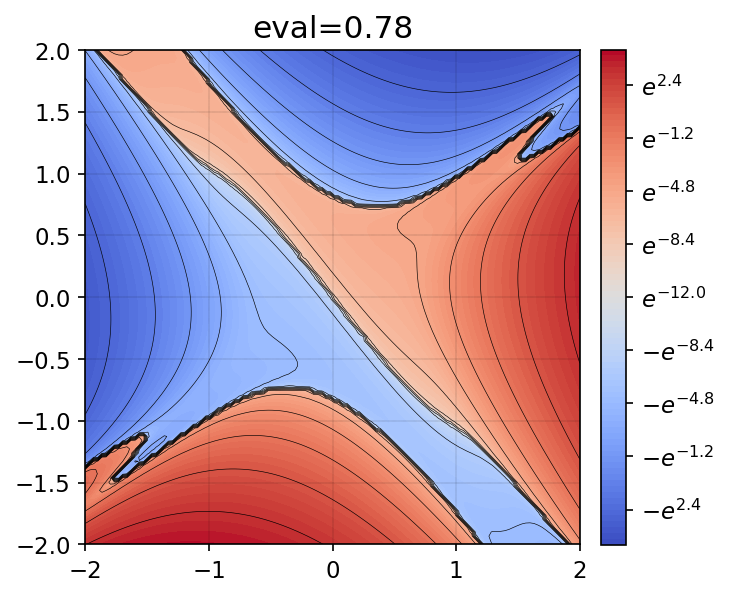}
    \end{minipage}\hfill 
    \begin{minipage}[c]{0.26\linewidth}
      \includegraphics[keepaspectratio, width=\linewidth]{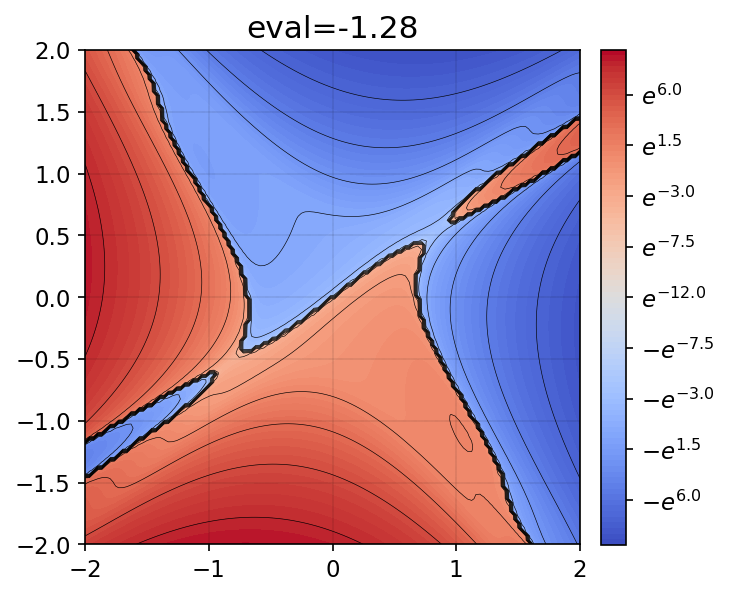}
    \end{minipage}
\end{minipage}
\caption{
Estimated eigenfunction of the extended Koopman operator $A_F^\times$ of the Duffing oscillator \eqref{duffing} via Algorithm \ref{algorithm: eigenfunction, continuous} using the exponential kernel $e^{x^\top y}$ with input $m_1=10$, $n_1=16$, $p_1=(-1,0)$ (left and left middle), $p_1=(0,0)$ (right middle and right), $N_1=3000$ samples from the uniform distribution on $[-1.5,1.5] \times [-0.5, 0.5]$, and the exact velocities on them.
The left panel is the heat maps with contour lines corresponding to the absolute value of the eigenfunction for eigenvalue $(-1 + \sqrt{31}i)/4$, computed using equilibrium point $(-1,0)$.
The middle-left panel is a combination of the contour lines of the left panel and the domain of attraction of the dynamical system.
The middle-right and right panels are the heat maps with contour lines corresponding to the eigenfunctions for eigenvalues $(-1 + \sqrt{17})/4$ and $(-1 - \sqrt{17})/4$, respectively, computed using the equilibrium point $(0,0)$.
}
\label{fig: duffing_eigenfunctions_single_equilibrium_point}
\end{figure}

Figure \ref{fig: duffing_eigenfunctions_two_equilibrium_points2} describes the approximated eigenfunctions for $-1$ using the exponential kernel $e^{x^\top y/\sigma^2}$ with $\sigma = 0.5$ and $1$.
We take two equilibrium points, $p_1=(-1,0)$ and $p_2=(1,0)$.
Then, we set $m_1=m_2=10$, $n_1=n_2=16$, 
 and take $N_1=N_2=7000$ samples from the uniform distribution on $[-1.5,1.5] \times [-0.5,0.5]$, and the exact velocities on them as input.
Here, we draw graphs of even and odd eigenfunctions constructed as in the following procedure:
first, we note that the linear operator $C_{-1}: H \to H; h \mapsto h((-1)\times\cdot)$ induces a Hermitian unitary operator.
Moreover, $C_{-1}$ is commutative with $A_F^*$ and satisfies $C_{-1}(V_{p_1,n}) = V_{p_2,n}$.
Thus, for the eigenfunction $v$ of the extended Koopman operator $A_F^\times$ in $V_{p_1,n}$, the image $C_{-1}v$ corresponds to an eigenvector for the same eigenvalue in $V_{p_2,n}$.
Therefore, we canonically construct even and odd eigenfunctions $v \pm C_{-1}v$ of the extended Koopman operator $A_F^\times$.
In the case of $\sigma = 0.5$, two peaks appear at the two equilibrium points, while they disappear in the case of $\sigma =1$.
This phenomenon is caused by the computation of the inverse matrix of $\widetilde{\mathbf{G}}_m := (\mathbf{G}_m^{ij})_{i,j}$ in \eqref{computation of H}.
In this computation, we employed the pseudo inverse of $\widetilde{\mathbf{G}}_m$ instead of its actual inverse.
When $\sigma=1$ and $m=10$, the condition number of $\widetilde{\mathbf{G}}_m$ becomes very large, thus, the smaller eigenvalues of $\widetilde{\mathbf{G}}_m$ are automatically truncated in the process of computing the pseudo inverse of $\widetilde{\mathbf{G}}_m$.
Indeed,  it can be seen that peaks appear on the equilibrium points in the case of a small $m$ with $\sigma=1$, but disappear if we intentionally truncate the small eigenvalues of $\widetilde{\mathbf{G}}_m$.
While this operation of ``truncating small eigenvalues'' thought to eliminate peaks at the equilibrium points, the appearance of contour lines along the boundary of the regions of attractions is quite interesting.
\begin{figure}
\begin{minipage}[c]{1.0\linewidth}
   \begin{minipage}[c]{0.25\linewidth}
      \includegraphics[keepaspectratio, width=\linewidth]{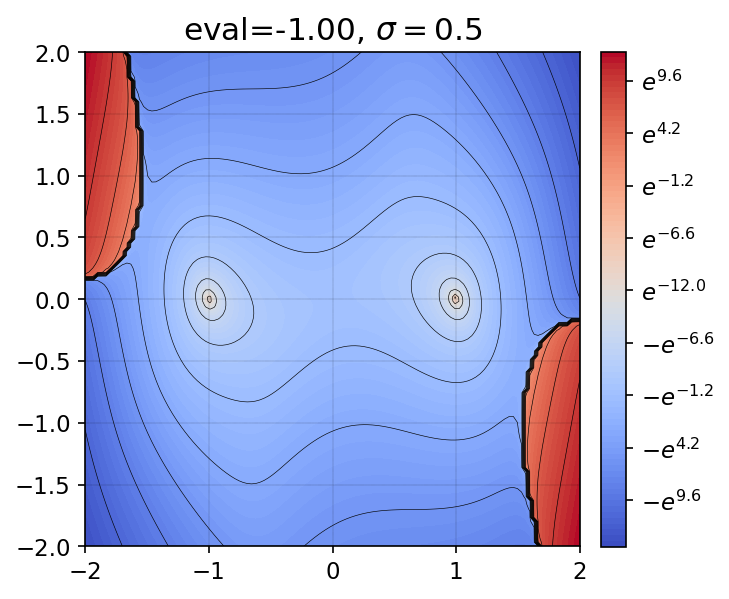}
    \end{minipage}\hfill 
    \begin{minipage}[c]{0.25\linewidth}
      \includegraphics[keepaspectratio, width=\linewidth]{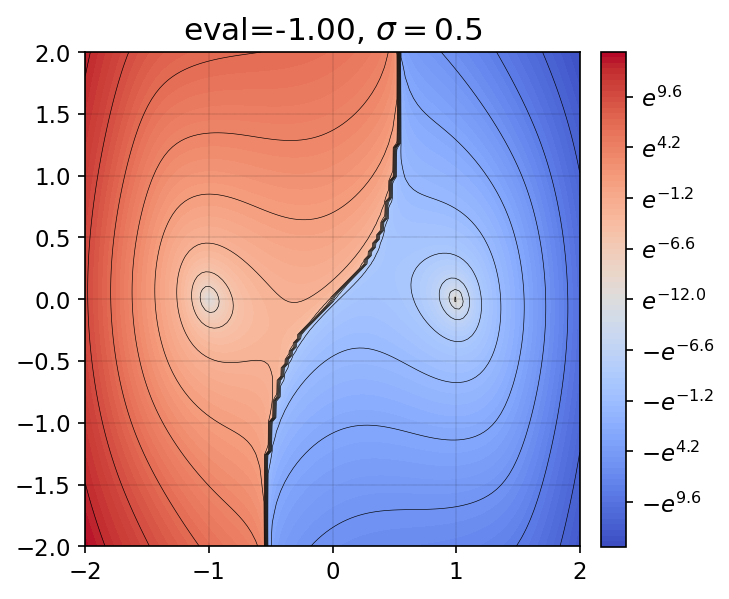}
    \end{minipage}\hfill 
    \begin{minipage}[c]{0.25\linewidth}
      \includegraphics[keepaspectratio, width=\linewidth]{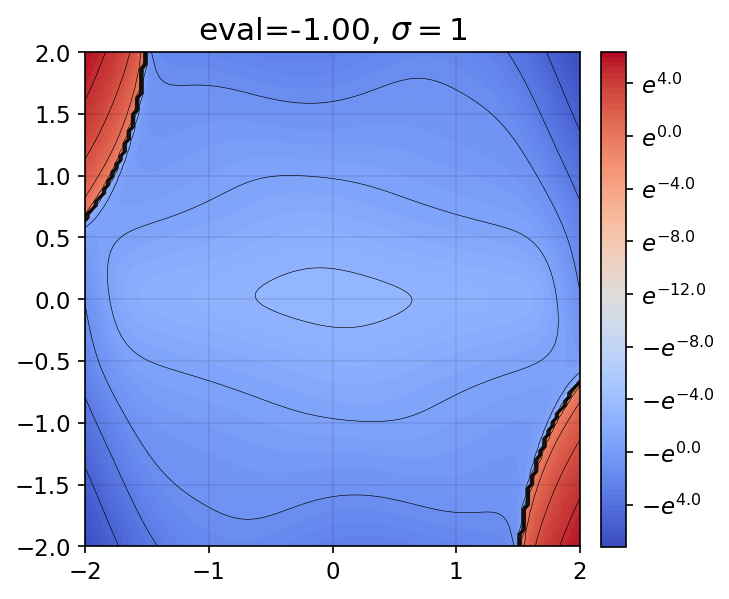}
    \end{minipage}\hfill 
    \begin{minipage}[c]{0.25\linewidth}
      \includegraphics[keepaspectratio, width=\linewidth]{images/duffing_realev_exp_m10_attr_ep_53_sigma100_epidx01_minus_withContour_eigenfunction.png}
    \end{minipage}
\end{minipage}
\caption{
Estimated eigenfunctions for the eigenvalue $-1$ of the extended Koopman operator $A_F^\times$ of the Duffing oscillator \eqref{duffing} via Algorithm \ref{algorithm: eigenfunction, continuous} using the exponential kernels $e^{x^\top y}$ (left and middle-left) and $e^{x^\top y/0.5}$ (right and middle-right) with input $m_1=m_2=10$, $n_1=n_2=16$, $p_1=(-1,0)$, $p_2=(1,0)$, $N_1=N_2=7000$ samples from the uniform distribution on $[-1.5,1.5] \times [-0.5,0.5]$, and the exact velocities on them.
}
\label{fig: duffing_eigenfunctions_two_equilibrium_points2}
\end{figure}

\end{document}